\documentclass[11pt,reqno]{amsart}
\usepackage{amssymb}
\usepackage{amsfonts}
\usepackage{amsmath}
\usepackage{stmaryrd}
\usepackage{physics}
\usepackage{braket}
\usepackage{dsfont}
\usepackage{bbold}
\usepackage{graphicx}
\usepackage{relsize}
\usepackage{makecell}
\usepackage[table,xcdraw]{xcolor}
\usepackage{enumerate}
\usepackage[pagebackref, colorlinks = true, linkcolor = blue, urlcolor  = blue, citecolor = red]{hyperref}
\usepackage[margin=1in]{geometry}
\usepackage{enumitem}
\usepackage{mathtools}
\usepackage{graphbox}
\usepackage{comment}
\usepackage{float}
\usepackage{hhline}
\usepackage[capitalize]{cleveref}
\usepackage{accents}
\usepackage[export]{adjustbox}
\usepackage{array}
\usepackage{caption} 
\definecolor{seagreen}{rgb}{0.18, 0.545, 0.341}

\makeatletter
\def\overUnderArrow{\@ifnextchar[\overUnderArrow@i{\overUnderArrow@i[]}}
\def\overUnderArrow@i[#1]#2#3{
  \ifx\relax#1\relax\array[b]{c}\overset{\text{#2}}{\downarrow}\\#3\endarray
  \else\ifx\relax#2\relax
    \array[t]{c}#3\\\underset{\text{#1}}{\uparrow}\endarray
  \else
    \array{c}\overset{\text{#2}}{\uparrow}\\#3\\\underset{\text{#1}}{\downarrow}\endarray
  \fi\fi}
\makeatother

\renewcommand{\epsilon}{\varepsilon}
\renewcommand{\phi}{\varphi}

\newcommand{\M}[1]{\mathcal{M}_{#1}(\mathbb{C})}

\newcommand{\Wg}{\operatorname{Wg}}
\newcommand{\Mob}{\operatorname{M\ddot{o}b}}
\renewcommand{\ring}[1]{\accentset{\circ}{#1}}
\def\restrict#1{\raise-.5ex\hbox{\ensuremath|}_{#1}}
\newcommand{\Cat}{\operatorname{Cat}}

\newtheorem{theorem}{Theorem}[section]
\newtheorem*{definition*}{Definition}
\newtheorem{proposition}[theorem]{Proposition}
\newtheorem{corollary}[theorem]{Corollary}
\newtheorem{lemma}[theorem]{Lemma}

\newtheorem{question}[theorem]{Question}

\newtheorem*{conjecture*}{Conjecture}

\theoremstyle{definition}
\newtheorem{definition}[theorem]{Definition}
\theoremstyle{definition}
\newtheorem{remark}[theorem]{Remark}
\theoremstyle{definition}
\newtheorem{example}[theorem]{Example}

\newcommand\vertarrowbox[3][6ex]{
  \begin{array}[t]{@{}c@{}} #2 \\
  \left\uparrow\vcenter{\hrule height #1}\right.\kern-\nulldelimiterspace\\
  \makebox[0pt]{\scriptsize#3}
  \end{array}
}

\definecolor{darkgreen}{rgb}{0,0.392,0}

\newtheorem{innercustomgeneric}{\customgenericname}
\providecommand{\customgenericname}{}
\newcommand{\newcustomtheorem}[2]{
  \newenvironment{#1}[1]
  {
   \ifdefined\crefalias\crefalias{innercustomgeneric}{#2}\fi
   \renewcommand\customgenericname{#2}
   \renewcommand\theinnercustomgeneric{##1}
   \innercustomgeneric
  }
  {\endinnercustomgeneric}
  \ifdefined\crefname\crefname{#2}{#2}{#2s}\fi
}

\newcustomtheorem{customthm}{Theorem}
\newcustomtheorem{customlemma}{Lemma}

\DeclareMathOperator*{\Plim}{\mathbb{P}-lim}

\newcommand{\A}{\mathcal{A}}
\newcommand{\B}{\mathcal{B}}
\newcommand{\D}{\mathcal{D}}
\newcommand{\E}{\mathbb{E}}
\newcommand{\Ec}{\mathcal{E}}

\newcommand{\W}{\mathcal{W}}

\newcommand{\Comp}{\mathbb{C}}

\newcommand{\id}{\mathrm{id}}
\newcommand{\la}{\langle}
\newcommand{\ra}{\rangle}
\newcommand{\Om}{\Omega}
\newcommand{\om}{\omega}
\newcommand{\eps}{\varepsilon}
\renewcommand{\phi}{\varphi}
\newcommand{\Sindep}{S_{\textrm{indep}}}
\newcommand{\transp}{\operatorname{\top}}

\newcommand{\cyc}{\operatorname{Cycles}}

\makeatletter

\newcommand*{\underarrow}{\def\@underarrow{\relax}\@ifstar{\@@underarrow}{\def\@underarrow{\hidewidth}\@@underarrow}}
\newcommand*{\@@underarrow}[2][]{\underset{\@underarrow\substack{\uparrow\if\relax\detokenize{#1}\relax\else\\#1\fi}\@underarrow}{#2}}

\newcommand*{\overarrow}{\def\@overarrow{\relax}\@ifstar{\@@overarrow}{\def\@overarrow{\hidewidth}\@@overarrow}}
\newcommand*{\@@overarrow}[2][]{\overset{\@overarrow\substack{\if\relax\detokenize{#1}\relax\else#1\\\fi\downarrow}\@overarrow}{#2}}
\makeatother

\author{Ion Nechita}
\email{nechita@irsamc.ups-tlse.fr}
\address{Laboratoire de Physique Th\'eorique, Universit\'e de Toulouse, CNRS, UPS, France}

\author{Sang-Jun Park}
\email{spark@irsamc.ups-tlse.fr}
\address{Laboratoire de Physique Th\'eorique, Universit\'e de Toulouse, CNRS, UPS, France}

\title[Tensor free probability theory]{Tensor free probability theory:\\asymptotic tensor freeness and central limit theorem}

\begin{document}
\begin{abstract}
Voiculescu's notion of asymptotic free independence applies to a wide range of random matrices, including those that are independent and unitarily invariant. In this work, we generalize this notion by considering random matrices with a tensor product structure that are invariant under the action of local unitary matrices. Assuming the existence of the \emph{tensor distribution} limit described by tuples of permutations, we show that an independent family of local unitary invariant random matrices satisfies asymptotically a novel form of freeness, which we term \emph{tensor freeness}. It can be defined via the vanishing of mixed \emph{tensor free cumulants}, allowing the joint tensor distribution of tensor free elements to be described in terms of that of individual elements. We present several applications of these results in the context of random matrices with a tensor product structure, such as partial transpositions of (local) unitarily invariant random matrices and tensor embeddings of random matrices. Furthermore, we propose a tensor free version of the central limit theorem, which extends and recovers several previous results for tensor products of free variables.
\end{abstract}

\maketitle

\tableofcontents

\newpage

\section{Introduction}

{
Tensors \cite{kolda2009tensor,landsberg2012tensors} are fundamental mathematical objects that naturally generalize matrices to higher-dimensional settings. Formally, a tensor can be represented as a multidimensional array indexed by multiple indices, each taking values within a finite set. From a more abstract perspective, tensors are simply elements of a tensor product space that correspond to multilinear maps. Tensors are instrumental across various domains, including algebraic geometry \cite{landsberg2012tensors}, representation theory \cite{fulton1991representation}, combinatorics \cite{tao2016capset}, quantum physics \cite{orus2014practical,hayden2016holographic,bruzda2024rank}, and computer science \cite{kourtis2019fast}, due to their versatility in representing complex multilinear relations and structures. Their significance extends beyond pure mathematics; tensors play essential roles in scientific applications, such as quantum entanglement in physics \cite{jivulescu2022multipartite,bruzda2024rank} high-dimensional data analysis \cite{liu2021tensors}, machine learning \cite{sidiropoulos2017tensor}, and signal processing \cite{cichocki2015tensor}, reflecting their crucial role in modern pure and applied science.

Random tensors naturally extend the theory of random matrices, enriching the probabilistic framework to accommodate higher-dimensional data structures and interactions. \emph{Random Matrix Theory} (RMT) \cite{mehta2004random,anderson2010introduction,bai2010spectral,tao2012topics,mingo2017free} focuses on the study of matrices with random entries, providing insights into the statistical properties of their eigenvalues and eigenvectors. Originally developed to model the energy levels of complex quantum systems in nuclear physics, RMT has since found applications across various disciplines, including number theory, statistical mechanics, and wireless communications. Its ability to describe universal behaviors in large complex systems makes it a powerful tool for understanding both the behavior of generic matrices and complex phenomena where deterministic approaches are challenging.

Random matrix theory has been profoundly impacted by \emph{Free Probability Theory} \cite{voiculescu1992free,hiai2000semicircle,nica2006lectures,mingo2017free}, introduced by Voiculescu \cite{voiculescu1985symmetries}, which offers powerful tools for analyzing asymptotic behaviors of large-dimensional random matrices through concepts like free independence and free cumulants. Free probability extends classical probability to non-commutative settings, particularly focusing on the study of non-commutative random variables and non-commutative analogues of cornerstone results in classical probability, such as the \emph{Central Limit Theorem} (CLT). 

In a similar fashion, random tensor theory aims to address asymptotic statistical properties and independence notions for tensor ensembles, motivating the generalization of free probability concepts to tensor frameworks. This tensorial perspective broadens the applicability of these ideas to new classes of mathematical and physical problems, such as tensor PCA \cite{montanari2014statistical,perry2020statistical} or quantum information theory \cite{collins2013matrix,cheng2024random,fitter2024max}, where tensor ensembles and symmetries are more natural than their simpler matricial counterparts. Distributional symmetries, such as \emph{unitary invariance} (UI), are ubiquitous in random matrix theory. For random tensors, \emph{local unitary invariance} (LUI) corresponds to tensor models that are invariant under tensor products of local unitary matrices; such symmetries are more natural in the tensor case, being also motivated by various models related to quantum entanglement. 

Recent advances in random tensor theory have significantly expanded the field. Kunisky, Moore, and Wein \cite{KMW24} introduced tensor cumulants to investigate statistical inference problems in high-dimensional invariant distributions, providing explicit and nearly orthogonal bases for invariant polynomial analysis. Their work clarified computational phase transitions and hardness thresholds within statistical inference tasks such as Tensor PCA. Bonnin and Bordenave \cite{BB24} established a systematic definition of freeness applicable to tensors of arbitrary orders, generalizing the classical matrix-oriented definitions of freeness. They introduced corresponding tensor free cumulants and proved asymptotic freeness results for basic tensor models, complemented by Schwinger-Dyson equations specific to tensors. Collins, Gurau, and Lionni \cite{collins2024free} developed concepts of tensorial free cumulants and tensor freeness specifically for locally unitary invariant random tensors. Their work rigorously defined tensorial analogs of classical free cumulants and demonstrated the additive property of these cumulants for sums of independent random tensors, establishing essential algebraic structures that underpin the theory. 

The present paper contributes to this active research direction in several ways: 
\begin{itemize}
    \item On the \emph{conceptual level}, we introduce and study in detail several generalizations of concepts from free probability theory to the world of tensors, contributing to the development of random tensor theory. We gather some of the main aspects in \cref{tbl:free-vs-tensor-free-concepts}.

\begin{table}[htb!]
\begin{tabular}{|c|c|c|}
\hline
\rowcolor[HTML]{EFEFEF} Free Probability Theory                         & {Tensor} Free Probability Theory                                                                   & Reference          \\ \hline
\begin{tabular}[c]{@{}c@{}}n.c.~probability space\\ $(\A,\phi)$\end{tabular}            & \begin{tabular}[c]{@{}c@{}}$r$-partite tensor prob.~sp.\\ $(\A,\phi, {(\phi_{\underline{\alpha}})})$\end{tabular} & \cref{def:tensor-ncps} \\ \hline
\begin{tabular}[c]{@{}c@{}}moments\\ $\phi_\pi(x_1, \ldots, x_p)$, \, $\pi \in NC(p)$\end{tabular}                                        & \begin{tabular}[c]{@{}c@{}}tensor moments\\ $\phi_{\underline{\alpha}}(x_1, \ldots, x_p)$, \, $\underline{\alpha} \in (S_p)^r$\end{tabular}                                                                         &  \cref{def:tensor-ncps}                         \\ \hline
\begin{tabular}[c]{@{}c@{}}free cumulants\\ $\tilde \kappa_\pi(x_1, \ldots, x_p)$, \, $\pi \in NC(p)$ \end{tabular}                                 & \begin{tabular}[c]{@{}c@{}}tensor free cumulants\\ $\kappa_{\underline{\alpha}}(x_1, \ldots, x_p)$, \, $\underline{\alpha} \in (S_p)^r$\end{tabular}                                                                  & \cref{def:free-tensor-cumulants}                         \\ \hline
\begin{tabular}[c]{@{}c@{}}free independence\\ $\tilde \kappa_p(\ldots, \textcolor{seagreen}{x},\ldots, \textcolor{orange}{y}, \ldots) = 0$\end{tabular}                              & \begin{tabular}[c]{@{}c@{}}tensor free independence\\ $\kappa_{\underline{\alpha} \, \text{irred.}}(\ldots, \textcolor{seagreen}{x},\ldots, \textcolor{orange}{y}, \ldots) = 0$ \end{tabular}                                                               & \cref{def:tensor-freeness}                         \\ \hline
\begin{tabular}[c]{@{}c@{}}\vspace{-.2cm}asympt.~freeness\\of indep.~UI rand.~mat.\\ $\{U X_i U^* \stackrel{\text{dist.}}{=} \!\!\!\overUnderArrow[\text{indep.}]{}{X_i}\!\!\!\} \xrightarrow[N \to \infty]{} \{\!\!\!\overUnderArrow[\text{free}]{}{x_i}\!\!\!\}$\end{tabular} & \begin{tabular}[c]{@{}c@{}}\vspace{-.2cm}asympt.~tensor freeness\\ of indep.~LUI rand.~mat.\\ $\{\overUnderArrow[$U_\otimes = \otimes_{s=1}^r U_s$]{}{U_\otimes X_i  U_\otimes^*} \stackrel{\text{dist.}}{=} \!\!\!\overUnderArrow[\text{indep.}]{}{X_i}\!\!\!\} \xrightarrow[D \to \infty]{} \{\!\!\!\!\!\!\overUnderArrow[\text{tensor free}]{}{x_i}\!\!\!\!\!\!\}$\end{tabular}                                 & \cref{thm-locui-tensorfree}                        \\ \hline
\begin{tabular}[c]{@{}c@{}}free CLT\\ $\vspace{0.2cm}\!\!\!\!\!\!\!\!\!\!\!\!\!\!\overUnderArrow[]{\qquad\quad centered, free variables}{\displaystyle{\frac{x_1 + \cdots + x_N}{\sqrt N}}} \!\!\!\!\!\!\!\!\!\!\!\!\!\!\!\xrightarrow[N \to \infty]{} \tilde{\kappa}^{1/2}_2(x)\!\!\!\!\!\!\!\!\!\!\!\!\!\!\!\!\!\!\!\!\!\!\!\!\!\!\!\!\!\!\!\overUnderArrow[ semicircular r.v.\qquad\qquad\qquad]{}{s}$\!\!\!\!\!\!\!\!\!\!\!\!\!\!\!\!\!\!\!\!\!\!\!\!\!\!\!\!\!\!\end{tabular}                                       & \begin{tabular}[c]{@{}c@{}}tensor free CLT\\ $\vspace{0.2cm}\!\!\!\!\!\!\!\!\!\!\!\!\!\!\!\!\!\!\!\!\!\!\!\!\!\!\overUnderArrow[]{\qquad\qquad\qquad centered, tensor free variables}{\displaystyle{\frac{x_1 + \cdots + x_N}{\sqrt N}}} \!\!\!\!\!\!\!\!\!\!\!\!\!\!\!\!\!\!\!\!\!\!\!\!\! \xrightarrow{N \to \infty} \!\!\!\!\!\displaystyle{\sum_{\underline{\alpha}\in (S_2)^r\setminus \{\underline{\id_2}\}}} \!\!\!\!\!\kappa_{\underline{\alpha}}^{1/2}(x)\!\!\!\!\!\!\!\!\!\!\!\!\!\!\!\!\!\!\!\!\!\!\overUnderArrow[tensor free r.v.\qquad\qquad]{}{s_{\underline{\alpha}}}\!\!\!\!\!\!\!\!\!\!\!\!\!\!\!\!\!\!\!\!\!\!$ \end{tabular}                                                                        & \cref{thm-TensorCLT}                         \\ \hline
\end{tabular}
\caption{Main conceptual contributions to tensor free probability theory and their counterparts in free probability theory.}
\label{tbl:free-vs-tensor-free-concepts}
\end{table}

    \item On the \emph{applicative level}, we apply the results obtained to study the asymptotic behavior of families of random matrices having various distributional symmetries and independence structures. The common feature of these random matrix models is that the corresponding linear operators act on $r$-partite tensor products of vector spaces. We prove that, under various hypotheses: 
    \begin{itemize}
        \item Independent and unitarily invariant families are asymptotically tensor free (\cref{cor-UITensorFree}). Similar results hold for orthogonally invariant random matrices (\cref{cor-OITensorFree}).
        \item Independent and \emph{local} unitarily invariant families are asymptotically tensor free (\cref{thm-locui-tensorfree}). Similar results hold for orthogonally invariant random matrices (\cref{thm-LocOI-tensorfree}).
        \item Partial transpositions of unitarily invariant random matrices are asymptotically tensor free (\cref{thm-LocUITranspose}). Similar results hold for independent families (\cref{thm-indepTranspose}) and orthogonally invariant random matrices (\cref{thm-OITranspose}).
        \item Tensor embeddings of a bipartite unitarily invariant random matrix are asymptotically tensor free (\cref{thm:embed-different-spaces}).
        \item Tensor freely independent, identically distributed elements satisfy a tensor free central limit theorem (\cref{thm-TensorCLT}). In the bipartite case, we obtain the distribution of the central limit (\cref{thm-CLTBipartite}).
    \end{itemize}
\end{itemize}

Our work considers exclusively matrices acting on a tensor product vector space, whereas the recent articles \cite{KMW24,BB24,collins2024free} deal with tensors of arbitrary rank. This focus allows us to follow closely Speicher's combinatorial formulation of Voiculescu's free probability theory, providing us with the tools to investigate relevant problems in random matrix theory related to multi-matrix models and their partial transpositions, as well as tensor generalizations of cornerstone results in non-commutative probability theory such as the tensor free central limit theorem.   

Importantly, the notion of tensor distribution is \emph{richer} than the usual notion of distribution from non-commutative probability theory, since it encapsulates a larger class of trace invariants. This means that the associated notion of freeness, tensor freeness, is different than Voiculescu's free independence. Although tensor freeness corresponds to the usual notion of freeness in the matrix case ($r=1$), we show in \cref{ex-TensorFreeNonFree} that already for $r=2$, there are examples of tensor freely independent elements that are not free in the usual sense. In the case of multipartite random matrices having (global) unitary invariance, the two notions are equivalent asymptotically, see \cref{thm-ui-tensorfree} and \cref{thm-oi-tensorfree} for the orthogonal case. Let us also mention that in the case of tensor embeddings of bipartite random matrices, tensor free independence captures more general situations (\cref{thm:embed-different-spaces}) than Voiculescu's freeness (\cref{cor:embeddings-free}).

Having established that tensor distributions and tensor free independence capture more general situations than the usual notions of non-commutative distributions and free independence, we would like to point out the newly introduced concepts allow us, in some cases, to obtain results about free (asymptotic) independence. This is the case, for example, in \cref{thm-indepTranspose} (part (2)), \cref{thm-OITranspose}, and \cref{cor:embeddings-free}. These results are either new or generalize previously known facts about the asymptotic freeness of certain classes of random matrices. Their proofs depend crucially on the newly introduced concept of tensor free independence, which provides a unified combinatorial framework capable of analyzing the asymptotic behavior of various random matrix models with tensor structures.

We present below two of the main results of this work that we also touched upon in \cref{tbl:free-vs-tensor-free-concepts}. 

First, let us state in more detail our main result regarding the asymptotic tensor freeness of independent families of random (local) unitarily invariant families of random matrices. 
}

\begin{theorem} [\cref{thm-indepTranspose}, Asymptotic tensor freeness of (local) unitary invariant random matrices]
Let $X_N^{(1)},\ldots, X_N^{(L)}$ be independent families of $N^r\times N^r$ random matrices, and let us identify $\M{N^r}\cong \M{N}^{\otimes r}$.
\begin{enumerate}
    \item If each $X_N^{(i)}$ is local unitary invariant, has tensor factorization property \cref{eq-condition-TensorFact}, and converges in tensor distribution, then the $L$ families of all partial transposes 
            $$\Big( \big\{(t_1\otimes \cdots \otimes t_r)(X_N^{(i)}): t_1,\ldots, t_r\in \{\id_N, \top\} \big\}\Big)_{i\in [L]}$$
    are asymptotically tensor free as $N\to \infty$, where $\top:\M{N}\to \M{N}$ is the transpose map.
    
    \item If each $X_N^{(i)}$ is unitary invariant, has the factorization property \cref{eq-condition-Fact}, and converges in distribution, then all the $2^r L$ random matrices $\left((t_1\otimes\cdots\otimes t_r)\big(X_N^{(i)}\big) \right)_{i\in [L],\, t_1,\ldots, t_r\in \{\id_N,\top\}}$ are both asymptotically free and asymptotically tensor free as $N\to \infty$.
\end{enumerate}
\end{theorem}

This result actually follows from a combination of two special cases. First, when partial transposition is not considered, we recover the tensor free independence between $L$ independent, LUI matrices $X_N^{(1)},\ldots, X_N^{(L)}$ (\cref{thm-locui-tensorfree}). Second, in the case $L=1$ and $X_N=X_N^{(1)}$, the existence of tensor distribution limit of $X_N$ automatically guarantees the existence of joint limit distribution of the family $\{(t_1\otimes \cdots \otimes t_r)(X_N): t_1,\ldots, t_r\in \{\id_N, \top\} \big\}$, under the assumption of local unitary invariance, and further implies (tensor) freeness between them when $X_N$ is globally UI (\cref{thm-LocUITranspose}). In particular, this recovers and generalizes the previous findings from \cite{MP24,PY24}. Additionally, we establish analogous results for \textit{(local) orthogonal invariant} random matrices; we refer to \cref{thm-oi-tensorfree,thm-LocOI-tensorfree,thm-OITranspose},

Secondly, we present our main results regarding the tensor free central limit theorem. 

\begin{theorem} [\cref{thm-TensorCLT,thm-CLTBipartite}, Central limit theorems for tensor free elements]

Let $\{x_i\}_{i=1}^{\infty}$ be a family of identically tensor distributed, and tensor free elements in an $r$-partite algebraic tensor probability space $(\A, \varphi, (\varphi_{\underline{\alpha}}))$. {Furthermore, suppose $\kappa_{\underline{\alpha}}(x_1)\geq 0$ for all $\underline{\alpha}\in (S_2)^r\setminus \{\underline{\id_2}\}$. Then we have
    $$\frac{x_1+\cdots+x_N-N\varphi(x_1)}{\sqrt{N}}\to \sum_{\underline{\alpha}\in (S_2)^r\setminus \{\underline{\id_2}\}} \sqrt{\kappa_{\underline{\alpha}}(x_1)}\,s_{\underline{\alpha}} \;\;\text{ in tensor distribution as $N\to \infty$},$$
where $(s_{\underline{\alpha}})_{\underline{\alpha}\in (S_2)^r\setminus \{\underline{\id_2}\}}$ are tensor free family of semicircular elements.

}

In the bipartite case ($r=2$), the limiting random variable above has (usual) distribution
    $$\big(D_{\sqrt{\kappa_{\gamma_2,\id_2}(x_1)}}[\mu_{SC}] * D_{\sqrt{\kappa_{\gamma_2,\id_2}(x_1)}}[\mu_{SC}]\big)\boxplus D_{\sqrt{\kappa_{\gamma_2,\gamma_2}(x_1)}}[\mu_{SC}],$$
where $*$ denotes the classical convolution, $\boxplus$ denotes the (additive) free convolution, $\mu_{SC}$ is the standard semicircular distribution, and $D$ denotes the dilation operator.
\end{theorem}

In the framework of tensor free probability, the central limit is characterized not by a single universal element but by a linear combination of $2^r-1$ semicircular elements $(s_{\underline{\alpha}})_{\underline{\alpha}\in (S_2)^r\setminus \{\underline{\id_2}\}}$, with coefficients determined by the tensor free cumulants of order 2 of the variables $x_i$.
Actually, the family $(s_{\underline{\alpha}})$ can always be constructed as the limit of \textit{tensor GUE models} and turns out to be \textit{$\eps$-free} \cite{CC21,CGVH24} for some choice of an adjacency matrix $\eps$; see \cref{rmk-TensorCLT} for details. Moreover, our result for the bipartite case recovers and extends the previous results on tensor products of free variables \cite{LSY24,Sko24}, as discussed in \cref{sec-CLTProdFree}.

\medskip

Note that our results are restricted to \textit{first-order behavior}. In particular, the assumptions we require for tensor freeness in this paper are minimal in some sense. The study of higher-order limits \cite{CMSS07,collins2024free} and strong convergence is postponed for future work.

\medskip

{
This paper is organized as follows. \cref{sec:preliminary} reviews preliminary concepts including permutations, partitions, free cumulants, Weingarten calculus, and distributional symmetries of random matrices. \cref{sec-TensorNCPS} introduces the algebraic framework of $r$-partite tensor probability spaces based on tensor trace invariants and defines convergence in tensor distribution. \cref{sec:tensor-free-cumulants} defines tensor free cumulants as a generalization of free cumulants using moment-cumulant formulas involving permutation tuples. \cref{sec:tensor-free-independence} introduces the central concept of tensor free independence, characterized by vanishing mixed tensor free cumulants, and examines its fundamental properties. \cref{sec-UItensorfree} is dedicated to the study of globally invariant random matrices from a tensor free probability perspective. In \cref{sec-LocUITensorFree} we present a key result showing that independent, locally invariant random matrices that converge in tensor distribution are asymptotically tensor free and discuss further asymptotic properties of LUI random matrices. In \cref{sec-TensorFreeNonIndep} we consider models of non-independent matrices, proving asymptotic tensor freeness for partial transposes (\cref{thm-indepTranspose}) and tensor embeddings (\cref{thm:embed-different-spaces}) of certain random matrices. Finally, \cref{sec:tensor-free-CLT} develops a tensor free central limit theorem (\cref{thm-TensorCLT}), demonstrating convergence towards a combination of tensor free semicircular elements.   

}

\section{Background on random matrix theory and free probability}\label{sec:preliminary}

In this section, we provide several preliminary notions and results from combinatorics, random matrix theory, and free probability. We also introduce notation used throughout the paper.

\subsection{Permutations and partitions}\label{sec:preliminary-permutations}

Throughout this paper, we denote by $S(A)$ the symmetric group acting on a finite set $A$ and by $S_p:=S([p])$ the symmetric group of order $p$, where $[p]:=\{1,2,\ldots, p\}$. For $\sigma\in S_p$, we use $\# \sigma$ to denote the \emph{number of disjoint cycles} in $\sigma$ and $|\sigma|$ to denote the \textit{length} of $\sigma$, that is, the minimum number of transpositions whose product equals $\sigma$. Both $\# \sigma$ and $|\sigma|$ depend only on the conjugacy class of $\sigma$, and they satisfy the relation $\#\sigma+|\sigma|=p$. The notation $|\cdot|$ is used in this paper exclusively to denote the length of permutations and not for \textit{cardinality} of sets; we use $\operatorname{Card}(\cdot)$ for the latter. {For these basic statistics on $S_p$, we refer the reader to \cite[Lecture 23]{nica2006lectures}.} Moreover, we use the notation 
$$\gamma_p:=(1\,2\,\cdots\,p)\in S_p$$
for the \emph{full cycle permutation} {and we shall denote by $\cyc(\sigma)$ the set of cycles of the permutation $\sigma$ (including singletons)}.

The length function $|\cdot|$ induces the metric on $S_p$, $d(\alpha,\beta):=|\alpha^{-1}\beta|$, from the fact that it satisfies the triangle inequality $|\alpha|+|\beta|\geq |\alpha\beta|$ for all $\alpha,\beta\in S_p$ and $|\sigma|=0$ if and only if $\sigma=\id_p$ (where $\id_p$ denotes the identity permutation), see \cite[Proposition 23.9]{nica2006lectures}. Let us define the following \emph{partial order relation} on $S_p$: $\alpha\leq \beta$ if they satisfy the equality
\begin{equation} \label{eq-PermGeod}
    d(\id_p,\alpha)+d(\alpha,\beta)=d(\id_p,\beta) \iff |\alpha|+|\alpha^{-1}\beta|=|\beta|,
\end{equation}
i.e.~the path $\id_p \to \alpha \to \beta$ is a \emph{geodesic}. Then it is straightforward to check that $(S_p,\leq)$ is a partially ordered set. For $\sigma\in S_p$, let us denote by $S_{NC}(\sigma)$ the set of permutations $\alpha$ such that $\alpha\leq \sigma$.

Let $\mathcal{P}(A)$ the set of all \emph{partitions} of the finite set $A$. If $A$ is totally ordered, we denote by $NC(A)$ the set of \textit{non-crossing partitions} of $A$, that is partitions $\pi$ for which there do not exist $i<j<k<l\in A$ such that $i \stackrel{\pi}{\sim} k$ and $j \stackrel{\pi}{\sim} l$; we refer the reader to \cite[Lectures 9 and 10]{nica2006lectures} for the combinatorics of non-crossing partitions. Furthermore, $\mathcal{P}_2(A)$ denotes the set of all pair partitions (or \emph{pairings}) of $A$, i.e.~partitions $\pi$ all whose blocks are of size $2$. As before, we simply denote by $\mathcal{P}(p)$, $NC(p)$, and $\mathcal{P}_2(p)$ when $A=[p]$. We again denote by $\leq$ the \textit{reversed refinement order} for partitions, i.e.~$\tau\leq \pi$ if each block of $\tau$ is completely contained in one of the blocks of $\pi$. Note that $(\mathcal{P}(p),\leq)$ is a \textit{lattice}: for all two partitions $\pi,\rho\in \mathcal{P}(p)$, there exists
\begin{enumerate}
    \item the \textit{join} $\pi\vee \rho\in \mathcal{P}(p)$ which is a minimum partition satisfying $\pi\vee \rho\geq \pi$ and $\pi\vee \rho\geq \rho$,

    \item the \textit{meet} $\pi\wedge \rho\in \mathcal{P}(p)$ which is a maximum partition satisfying $\pi\wedge \rho\leq \pi$ and $\pi\wedge \rho\leq \rho$.
\end{enumerate}
The poset $(NC(p),\leq)$ is also a lattice \cite[Proposition 9.17]{nica2006lectures}, and both sets $\mathcal{P}(p)$ and $NC(p)$ share the same minimum partition $0_p=\{\{1\},\ldots, \{p\}\}$ and maximum partition $1_p=\{\{1,\ldots, p\}\}$. On the other hand, the two join operations $\vee_{\mathcal{P}}$ on $\mathcal{P}(p)$ and $\vee_{NC}$ on $NC(p)$ are distinct in general: if we take two non-crossing partitions $\pi=\{\{1,3\},\{2\},\{4\}\}, \rho=\{\{2,4\},\{1\},\{3\}\}$ in $NC(4)$, then $\pi\vee_{\mathcal{P}} \rho=\{\{1,3\},\{2,4\}\}$ while $\pi\vee_{NC} \rho=1_4$. However, we always have $\wedge_{\mathcal{P}}=\wedge_{NC}$.

The set $S_{NC}(\sigma)$ can be completely described through the comparison with the non-crossing partitions. To this end, let $\Pi:S(A)\to \mathcal{P}(A)$ be the natural projection map by identifying each cycle of $\sigma\in S(A)$ as a block of $\Pi(\sigma)${, by discarding the order of its elements}. Furthermore, let us denote the \textit{disjoint product} of permutations $\sigma_i\in S(A_i)$ by $\bigsqcup_{i=1}^k \sigma_i\in S(\bigsqcup_{i=1}^k A_i)$, and analogously the disjoint concatenation $\bigsqcup_{i=1}^k \pi_i\in \mathcal{P}(\bigsqcup_{i=1}^k A_i)$ of partitions $\pi_i\in \mathcal{P}(A_i)$. For example, we can write the cycle decomposition of $\sigma\in S_p$ as $\sigma=\bigsqcup_i c_i$ when each $c_i$ is a full cycle on a set $A_i$ and $[p]=\bigsqcup_i A_i$.

The following facts are well-known in combinatorics. We refer to \cite{biane1997some,nica2006lectures, mingo2017free} for the proof.

\begin{proposition} \label{prop:lattice-structure}
For the full cycle $\gamma_p=(1\,2\,\cdots\,p)\in S_p$, the map $\Pi:S_p\to \mathcal{P}(p)$ induces the natural order isomorphism between $(S_{NC}(\gamma_p),\leq)$ and $(NC(p),\leq)$. Specifically,
\begin{enumerate}
    \item $\alpha\in S_{NC}(\gamma_p)$ if and only if $\Pi(\alpha)\in NC(p)$ and every cycle of $\alpha$ can be written in the form $(j_1\,j_2\,\cdots j_l)$ for $j_1<j_2<\cdots j_l$,

    \item For $\alpha,\beta\in S_{NC}(\gamma_p)$, $\alpha\leq \beta$ if and only if $\Pi(\alpha)\leq \Pi(\beta)$.
\end{enumerate}
\end{proposition}

\begin{proposition} \label{prop:lattice-structure2}
Let $\sigma\in S_p$ and $\sigma=\bigsqcup_{i=1}^k c_i$ be the cycle decomposition of $\sigma$. Then $S_{NC}(\sigma)\cong S_{NC}(c_1)\times\cdots \times S_{NC}(c_k)$, in the sense that $\alpha\in S_{NC}(\sigma)$ if and only if we can write $\alpha=\bigsqcup_{i=1}^k\alpha_i$ for some $\alpha_i\in S_{NC}(c_i)$ for each $i$. In particular, $\leq$ induces a lattice structure on $S_{NC}(\sigma)$ for every permutation $\sigma$.

Furthermore, elements in $S_{NC}(\sigma)$ are completely determined by their associated partitions: if $\alpha,\alpha'\in S_{NC}(\sigma)$ and $\Pi(\alpha)=\Pi(\alpha')$, then $\alpha=\alpha'$.
\end{proposition}

{As an example, for a permutation $\sigma ={(1\; 7\; 3) (2\; 5\; 6\; 4)}\in S_{7}$, we have
\begin{align*}
    (1\;3)(2\;4)(5\;6)\leq \sigma &\text{ and } (2\;6\;4)\leq \sigma,  
    \\
    (2\;3)\not\leq \sigma,\; (2\;6)(4\;5)\not\leq \sigma, &\text{ and } (2\;4\;6)\not\leq \sigma;
\end{align*}
{the last three relations fail, respectively, because: $2$ and $3$ do not belong to the same cycle of $\sigma$, the cycles $(26)$ and $(45)$ are crossing in $\sigma$, and the cycle $(246)$ is ordered differently in $\sigma$.}

Thanks to the lattice structure on $S_{NC}(\sigma)$, the join and meet operations on $S_{NC}(\sigma)$ is well-defined. We remark that the meet operation $\wedge$ is independent of the choice of $\sigma$: if $\alpha,\beta\in S_{NC}(\sigma)\cap S_{NC}(\sigma')$, then $\alpha\wedge_{S_{NC}(\sigma)}\beta=\alpha\wedge_{S_{NC}(\sigma')}\beta$. However, join operation $\vee$ may depend on $\sigma$: if $\alpha,\beta,\sigma,\sigma'\in S_5$ are 
    $$\alpha=(1\,2)(3\,4\,5),\;\; \beta=(2\,3),\;\; \sigma=\gamma_5,\;\;\sigma'=(1\,2\,5\,3\,4),$$
then $\alpha,\beta\in S_{NC}(\sigma)\cap S_{NC}(\sigma')$, but $\alpha\vee_{S_{NC}(\sigma)}\beta=\sigma$ while $\alpha\vee_{S_{NC}(\sigma')}\beta=\sigma'$.
}

{In this paper, we shall often consider \emph{tuples} of permutations and partitions in many kinds of computations. We shall use \emph{underlined symbols} to denote such tuples of objects. For example, we usually denote permutations and partitions by Greek letters, e.g.~$\alpha,\beta,\sigma,\tau \in S_p$ and $\pi,\tau\in \mathcal{P}(p)$; an $r$-tuple of permutations and partitions will be denoted thus by $\underline{\alpha} \in S_p^r$ and $\underline{\pi}\in \mathcal{P}(p)^r$, etc. In other words, $\underline{\alpha} = (\alpha_1, \alpha_2, \ldots, \alpha_r)$ and $\underline{\pi}=(\pi_1,\ldots, \pi_r)$. Furthermore, we naturally extend the operations on tuples of permutations or partitions in a component-wise manner: for instance,
\begin{align*}
    \underline{\sigma}\cdot\underline{\tau}&:=(\sigma_1\tau_1,\ldots, \sigma_r\tau_r), \quad  \underline{\sigma}^{-1}:=(\sigma_1^{-1},\ldots, \sigma_r^{-1}), \quad \alpha\underline{\sigma}\beta:=(\alpha\sigma_1\beta,\ldots, \alpha \sigma_r \beta)\\
    \underline{\pi}\vee\underline{\pi'}&:=(\pi_1\vee\pi'_1,\ldots, \pi_r\vee\pi'_r), \quad \underline{\pi}\wedge \underline{\pi'}:=(\pi_1\wedge \pi'_1,\ldots, \pi_r\wedge \pi'_r)\\
    \underline{\pi}\vee \rho&:=(\pi_1\vee\rho,\ldots, \pi_r\vee\rho), \quad \underline{\pi}\wedge \rho:=(\pi_1\wedge \rho,\ldots, \pi_r\wedge \rho),
\end{align*}
for $\underline{\sigma},\underline{\tau}\in S_p^r$, $\alpha,\beta\in S_p$, $\underline{\pi},\underline{\pi'}\in \mathcal{P}(p)^r$, and $\rho\in \mathcal{P}(p)$. Since most of our work is about generalizing Voiculescu's free probability theory (and especially Speicher's combinatorial approach to it) to the $r$-tensor setting, this notation will be very useful in our work. Most of the combinatorial machinery of free probability theory (moments, cumulants, non-crossing partitions, etc) will be generalized to the $r$-tensor setting by replacing the objects in the free case by $r$-tuples of objects in the tensor free case.
}

{
\begin{definition}
For positive integers $r,p$, we endow the set $(S_p)^r$ of $r$-tuples of permutations of $p$ elements with a \textit{partial order}, as follows: for $\underline \alpha, \underline \beta \in (S_p)^r$, 
$$ \underline \alpha \leq \underline \beta \qquad \text{ if } \qquad \forall s \in {[r]}, \quad  |\alpha_s| + |\alpha_s^{-1} \beta_s| = |\beta_s|.$$
In other words, $\underline \alpha \leq \underline \beta$ if for all $s=1,2,\ldots, r$, the permutation $\alpha_s$ lies on the geodesic $\id_p \to \alpha_s \to \beta_s$ between the identity permutation $\id_p$ and the permutation $\beta_s$.
{We also write $\underline{\alpha} \in S_{NC}(\underline{\beta})$.}
\end{definition}

With this definition, the poset $\big((S_p)^r, \leq \big)$ is the $r$-times direct product (see \cite[Definition 9.27]{nica2006lectures}) of the poset $(S_p,\leq)$. Using the characterization of geodesics in $S_p$ from \cref{prop:lattice-structure,prop:lattice-structure2}, we can describe the partial order on $(S_p)^r$ even more explicitly: $\underline \alpha \leq \underline \beta$ if the following conditions hold for all $s \in [p]$: 
\begin{itemize}
	\item every cycle of $\alpha_s$ is contained inside some cycle of $\beta_s$ (as sets)
	\item the cycles of $\alpha_s$ induce non-crossing partitions on the cycles of $\beta_s$
	\item the cycles of $\alpha_s$ have the same cyclic order as the cycles of $\beta_s$ they belong to.
\end{itemize}

Note that the poset $\big( (S_p)^r, \leq \big)$ is \emph{not} a lattice for $r \geq 1$ and $p \geq 3$. For example, in $S_3$ the two full cycles $(123)$ and $(132)$ do not admit a common larger element: $\sigma \geq (123) \implies \sigma = (123)$ and similarly for $(132)$. The same two full cycles do not admit a largest common smaller element: the three transpositions $(12)(3)$, $(13)(2)$, and $(1)(23)$ are all common smaller elements, and maximal with respect to $\leq$. This is in contrast with the case of the poset of non-crossing partitions, which is a lattice. In a similar vein, for all permutations $\alpha \in S_p$, the poset $S_{NC}(\alpha)$ is a lattice, see \cref{prop:lattice-structure}.

}

{Let us further collect several properties of \textit{pairings} discussed in \cite{MP13,MP19}{; these properties will be particularly useful in the study of transpositions of random matrix models and, more generally, in the study of random matrix models with orthogonal invariance}. First, we denote by $\mathcal{P}_2(\pm p)$ the set of pairings on the set $[\pm p] := [p] \sqcup [-p]$, where $[-p] := \{-1, -2, \dots, -p\}$. Throughout this paper, we will fix the pairing $\delta := (1, -1)(2, -2) \cdots (p, -p) \in \mathcal{P}_2(\pm p)$ and consider two natural embeddings
\begin{enumerate}
    \item $\mathcal{P}_2(\pm p)\hookrightarrow S_{\pm p}:=S([\pm p])$, \quad $\pi\mapsto \pi$ (identifying each pair as a transposition), 

    \item $S_p\hookrightarrow S_{\pm p}$, \quad $\alpha\mapsto \alpha:= \alpha\sqcup \id_{[-p]}$.
\end{enumerate}
Following these identifications, one can consider two other embeddings 
\begin{enumerate}
    \item $S_p\hookrightarrow \mathcal{P}_2(\pm p)$, \quad $\alpha\mapsto \alpha\delta\alpha^{-1}=(\alpha(1), -1)\cdots (\alpha(p), -p)$,

    \item $S_p\hookrightarrow S_{\pm p}$, \quad $\alpha\mapsto \alpha\delta\alpha^{-1}\delta = \alpha \sqcup (\delta\alpha\delta)^{-1}$,
\end{enumerate}
as in \cite{MP19,collins2024free}. Here the multiplications above are well-defined in $S_{\pm p}$. One can greatly benefit from these identifications when we deal with the join operation between two pairings, as shown in the lemmas below.

\begin{lemma}[{\cite[Lemma 2]{MP13}}] \label{lem-PairingSup}
Let $\pi,\rho \in \mathcal{P}_2(\pm p)$ be pairings and $c= (i_1\;i_2\;\cdots\; i_n)$ a cycle of $\pi\rho\in S_{\pm p}$. Let $j_k=\rho(i_k)$. Then $(j_n\; j_{n-1}\; \cdots \; j_1)=\rho c^{-1}\rho$ is also a cycle of $\pi\rho$, these two cycles are distinct, and $\{i_1, j_1, i_2,j_2, \ldots, i_n, j_n\}$ is a block of $\pi\vee \rho\in \mathcal{P}(\pm p)$. In particular:
$$2\underbrace{\#(\pi\vee \rho)}_{\substack{\text{number of blocks of}\\\text{the partition $\pi \vee \rho$}}} =\underbrace{\#(\pi\rho)}_{\substack{\text{number of cycles of}\\\text{the permutation $\pi\rho$}}}.$$
\end{lemma}

\begin{lemma} [{\cite[Section 5]{MP19}}] \label{lem-PairingEmb}
The map $\alpha \mapsto \alpha\delta\alpha^{-1}$ is a bijection from $S_p$ onto $\mathcal{P}_2^{\delta}(\pm p)$, the set of pairings $\pi$ such that $\pi(k)\in [-p]$ for all $k\in [p]$ (or equivalently, $\pi\delta$ leaves $[p]$ invariant: $\pi\delta\big([p])=\pi([-p])\subset [p]$). Furthermore, $\pi\in \mathcal{P}^{\delta}_2(\pm p) \mapsto \pi\delta\big|_{[p]}\in S_p$ is the inverse map.
\end{lemma}

Recall from the definition of disjoint product that $\bigsqcup_{k=1}^p S(\{k,-k\})$ is the set of permutations $\eps\in S_{\pm p}$ such that $\eps(k)\in \{\pm k\}$ for each $k\in [p]$. In particular, $\delta\in \bigsqcup_{k=1}^p S(\{k,-k\})$, every element $\eps\in \bigsqcup_{k=1}^p S(\{k,-k\})$ satisfies $\eps^2=\id_{\pm p}$, and all such $\eps$ commute with each other.

\begin{lemma} \label{lem-pairings}
\begin{enumerate}
    \item For every $\pi\in \mathcal{P}_2(\pm p)$, one can write $\pi=\eps \sigma\delta\sigma^{-1}\eps$ or equivalently,
        $$\pi\delta=(\eps\sigma\eps)\delta(\eps\sigma\eps)^{-1}\delta=\eps\sigma\delta\sigma^{-1}\delta\eps,$$
    for some $\sigma\in S_p$ and $\eps\in \bigsqcup_{k=1}^p S(\{k,-k\})\subset S_{\pm p}$. Moreover, we have $|\pi\delta|=2|\sigma|$.

    \item For $\pi\in \mathcal{P}_2(\pm p)$ and $\alpha\in S_p$, we have $\pi\delta\in S_{NC}(\alpha\delta\alpha^{-1}\delta)$ if and only if there exists $\sigma\in S_{NC}(\alpha)$ such that $\pi=\sigma\delta\sigma^{-1}$. Moreover, the corresponding $\sigma$ is uniquely determined.

    \item For $\pi,\rho\in \mathcal{P}_2(\pm p)$, we have $\rho\delta\in S_{NC}(\pi\delta)$ if and only if there exists $\sigma, \tau\in S_p$ and $\eps\in \bigsqcup_{k=1}^p S(\{k,-k\})$ such that $\tau\in S_{NC}(\sigma)$ and
        $$\pi=\eps\sigma\delta\sigma^{-1}\eps,\quad \rho=\eps\tau\delta\tau^{-1}\eps.$$
\end{enumerate}
\end{lemma}
\begin{proof}
(1) According to \cref{lem-PairingSup}, the cycle decomposition of $\pi\delta$ is of the form
    $\pi\delta=c_1c_1'\cdots c_l c_l'$,
where the cycle $c_k=(i_1\,\cdots\, i_n)$ induces another cycle $c_k'=(-i_n\,\cdots\,-i_1)=\delta c_k^{-1}\delta$. Since all the cycles $c_1,\ldots, c_l, c_1',\ldots, c_k'$ commute, we can write
    $$\pi\delta=(c_1\cdots c_l)\delta (c_1\cdots c_l)^{-1}\delta.$$
Note that the permutation $c_1\cdots c_l$ stabilizes a set $I_p$ (i.e., $(c_1\cdots c_l)\big|_{I_p}=\id_{I_p}$) having cardinality $p$ and $I_p\sqcup (-I_p)= [\pm p]$ since $c_k$ and $c_k'$ are disjoint cycles acting on the same set up to opposite signs. Therefore, we can find a permutation $\eps\in \bigsqcup_{k=1}^p S(\{k,-k\})$ such that $\sigma=\eps(c_1\cdots c_l)\eps$ stabilizes the set $\eps(I_p)=[-p]$, i.e., $\sigma\in S_p$. Since $\eps$ commutes with $\delta$, we have
    $$\pi\delta=(\eps\sigma\eps)\delta (\eps\sigma\eps)^{-1}\delta=\eps\sigma\delta\sigma^{-1}\delta\eps.$$
Moreover, we have $|\pi\delta|=|\sigma\delta\sigma^{-1}\delta|=|\sigma|+|\delta\sigma^{-1}\delta|=2|\sigma|$ since $\sigma$ and $\delta\sigma^{-1}\delta$ have disjoint cycle decompositions (after ignoring all singletons).

(2) Since $\alpha\delta\alpha^{-1}\delta=\alpha\sqcup (\delta\alpha\delta)^{-1}$ leaves $[p]$ invariant, \cref{prop:lattice-structure2} implies that $\pi\delta$ also leaves $[p]$ invariant. Therefore, \cref{lem-PairingEmb} implies that there exists unique $\sigma\in S_p$ such that $\pi=\sigma\delta\sigma^{-1}$. Furthermore, $\sigma\in S_{NC}(\alpha)$ again by \cref{prop:lattice-structure2}.

(3) Let us write $\pi=\eps\sigma\delta\sigma^{-1}\eps$ as in (1). Then the condition $\rho\delta\in S_{NC}(\pi\delta)$ is equivalent to
\begin{align*}
    |\rho\delta|+|(\rho\delta)^{-1}(\eps\sigma\delta\sigma^{-1}\eps\delta)|=|\eps\sigma\delta\sigma^{-1}\eps\delta| &\iff |\eps\rho\delta\eps|+|(\eps\rho\delta\eps)^{-1}(\sigma\delta\sigma^{-1}\delta)|=|\sigma\delta\sigma^{-1}\delta|\\
    &\iff (\eps\rho\eps)\delta\in S_{NC}(\sigma \delta\sigma^{-1}\delta),
\end{align*}
since $\eps$ and $\delta$ commute. Therefore, the conclusion follows from (2).
\end{proof}

We present in \cref{fig:example-pairing-epsilon-delta} the decomposition of the pairing $\pi = (1\, 2)(-2 \, 3)(-3\, 4)(-1 \, -4)$ according to the result above:
$$(1\, 2)(-2 \, 3)(-3\, 4)(-1 \, -4) =  \underbrace{(-1 \, 1)}_{\epsilon} \cdot \underbrace{(1\, 2\, 3\, 4)}_{\sigma} \cdot \delta \cdot (1\, 2\, 3\, 4)^{-1} \cdot (-1 \, 1).$$

\begin{figure}[!htb]
    \centering
    \includegraphics[width=0.8\linewidth]{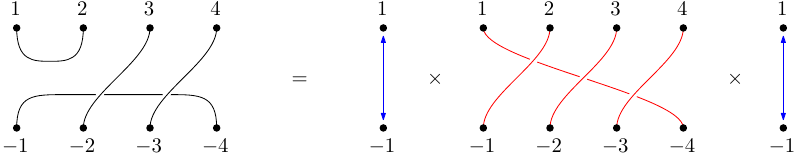}
    \caption{Factoring of $\pi = (1\, 2)(-2 \, 3)(-3\, 4)(-1 \, -4)$ as $\pi = \textcolor{blue}{\eps} \textcolor{red}{\sigma\delta\sigma^{-1}}\textcolor{blue}{\eps}$, with $\textcolor{blue}{\eps} = (-1 \, 1)$ and $\textcolor{red}{\sigma\delta\sigma^{-1}} = \prod_{i=1}^4 (-i\;\sigma(i))$.}
    \label{fig:example-pairing-epsilon-delta}
\end{figure}
}

\subsection{Moments and free cumulants associated to permutations} \label{sec:preliminary-FreeProb}

Let us recall basic notions from Free Probability Theory{; the reader can consult the reference texts \cite{voiculescu1992free, nica2006lectures, mingo2017free} for further information}. A \textit{non-commutative probability space} $(\A,\varphi)$ consists of a unital algebra $\A$ over $\Comp$ and a unital linear functional $\varphi:\A\to \Comp$ (called an \textit{expectation} or a \textit{state}). For our purposes, we always assume that $\varphi$ is \textit{tracial}, i.e.~$\varphi(ab)=\varphi(ba)$ for all $a,b\in \A$. Elements in $\A$ are called \textit{non-commutative random variables}. The \textit{distribution} of a family $\W\subset \A$ is the data of (joint) moments $\big(\varphi(P(\W))\big)_P$, where $P$ runs over all \emph{non-commutative} complex polynomials on finitely many variables, evaluated on entries from $\W$. Using linearity, the state $\varphi$ is completely determined by moments of monomials
    $$\varphi(x_1 \cdots x_p); \quad p\geq 1,\;\;,x_1,\ldots, x_p\in \A.$$
For a set $I$, unital subalgebras $(\A_i)_{i\in I}$ of $\A$ are called \textit{freely independent} (or \textit{free}) if every alternating product of centered elements is centered, i.e.,
    $$x_i\in \A_{f(i)}, \;\varphi(x_i)=0 \text{ for $i\in [p]$ and $f(i)\neq f(i+1)$ for $i=1,\ldots, p-1$} \implies \varphi(x_1\cdots x_p)=0.$$
Furthermore, subsets $(\W_i)_{i\in I}$ are called free if the unital algebras generated by each $\W_i$ are free. The definition of freeness provides a universal rule to compute the joint distribution of $\bigcup_{i\in I}\A_i$ (resp. $\bigcup_{i\in I}\W_i$) in terms of marginal distributions of $\A_i$ (resp. $\W_i$) (\cite[Lecture 5]{nica2006lectures}).

We sometimes consider the case where $\A$ has an \textit{involution} $*:\A\to \A,\;\; a\mapsto a^*$, which is conjugate-linear and $(ab)^*=b^*a^*$ for $a,b\in \A$. We call $(\A,\varphi)$ a \textit{$*$-probability space} if it is non-commutative probability space with an involution and $\varphi(a^*a)\geq 0$ for all $a\in \A$. Then the \emph{$*$-distribution} of $\W\subset \A$ is defined as the distribution of $\W\cup \W^*$ and the \emph{$*$-free independence} between $\W_1,\ldots, \W_L$ is defined as the freeness between $\W_1\cup \W_1^*,\ldots, \W_L\cup \W_L^*$. However, for full generality, non-commutative probability spaces in this paper have the minimal structure and may not necessarily have an involution unless otherwise specified.

The notion of free cumulants, introduced by Speicher \cite{speicher1994multiplicative}, provides a useful tool when we deal with free independence. For each $p\geq 1$ and permutation $\sigma\in S_p$, we first associate a multilinear functional $\varphi_{\sigma}:\A^p\to \mathbb{C}$ defined by
\begin{equation} \label{eq-MomentPerm}
    \varphi_{\sigma}(x_1,\ldots, x_p)=\prod_{\substack{c\in \cyc(\sigma)\\ c=(i_1\,i_2\,\cdots\,i_n)}} \varphi(x_{i_1}\cdots x_{i_n}).
\end{equation}
For example, we have
    $$\varphi_{(13)(624)(5)}(x_1,\ldots, x_6)=\varphi(x_1x_3)\varphi(x_6 x_2 x_4) \varphi(x_5), \quad \varphi_{\gamma_p}(x_1,\ldots, x_p)=\varphi(x_1\cdots x_p).$$
This notation mimics the one when the algebra $\A$ is a matrix algebra, and the functional $\phi$ is the trace; see e.g. \cref{eq-SingleTraceInv}.

Then for $\sigma\in S_p$, we define the \textit{free cumulants} $\tilde{\kappa}_{\sigma}:\A^p\to \Comp$ as multilinear functionals satisfying the so-called \textit{free moment-cumulant relation} 
\begin{equation} \label{eq-FreeMomentCumulant}
    \varphi_{\sigma}(x_1, \ldots, x_p)=\sum_{\tau\in S_{NC}(\sigma)} \tilde{\kappa}_{\tau}(x_1,\ldots, x_p), \quad \sigma\in S_p,\;\; x_1,\ldots, x_p\in \A.
\end{equation}
From the M\"{o}bius inversion formula over each lattice $S_{NC}(\sigma)$, one can invert the above relation and write the formula of $\tilde{\kappa}_{\sigma}$ in terms of $\varphi_{\tau}$'s:
\begin{equation} \label{eq-FreeMomentCumulant2}
    \tilde{\kappa}_{\sigma}(x_1, \ldots, x_p)=\sum_{\tau\in S_{NC}(\sigma)} \varphi_{\tau}(x_1,\ldots, x_p) \Mob(\tau^{-1}\sigma), \quad \sigma\in S_p,\;\; x_1,\ldots, x_p\in \A,
\end{equation}
where the \textit{M\"{o}bius function} $\Mob:S_p\to \Comp$ is defined by
    $$\Mob(\sigma) := \prod_{c \in \cyc(\sigma)} (-1)^{|c|} \operatorname{Cat}_{|c|}$$
with the Catalan numbers $\operatorname{Cat}_n := \frac{1}{n+1} \binom{2n}{n}$ (recall that the length $|c|$ of a cycle $c$ equals the cardinality of $c$ minus $1$). {In particular, this shows that the implicit definition of free cumulant $\tilde \kappa_\tau$ from \cref{eq-FreeMomentCumulant} does not depend on the permutation $\sigma \geq \tau$.} {Finally, we denote by $\varphi_{\sigma}(x):=\varphi_{\sigma}(x,x,\ldots, x)$ and $\tilde{\kappa}_{\sigma}(x):=\tilde{\kappa}_{\sigma}(x,x,\ldots, x)$ for simplicity.} Let us recall the \textit{multiplicativity} of $(\varphi_{\sigma})_{\sigma\in \bigsqcup_{p\geq 1}S_p}$ and $(\tilde{\kappa}_{\sigma})_{\sigma\in \bigsqcup_{p\geq 1}S_p}$ \cite[Lecture 11]{nica2006lectures}: for $\alpha\in S_p$, $\beta\in S_q$, and $x_1,\ldots, x_p,y_1,\ldots, y_q\in \A$,
\begin{align*}
    \varphi_{\alpha\sqcup\beta} (x_1,\ldots, x_p,y_1,\ldots, y_q) &= \varphi_{\alpha} (x_1,\ldots, x_p)\varphi_{\beta} (y_1,\ldots, y_q),\\
    \tilde{\kappa}_{\alpha\sqcup\beta} (x_1,\ldots, x_p,y_1,\ldots, y_q) &= \tilde{\kappa}_{\alpha} (x_1,\ldots, x_p)\tilde{\kappa}_{\beta} (y_1,\ldots, y_q).
\end{align*}

The free cumulants are the free analogs of \textit{classical cumulants} $k_p:(L^{\infty-}(\mathbb{P}))^p\to \Comp$, where $L^{\infty-}(\mathbb{P})=\bigcap_{p=1}^{\infty}L^p(\Om,\mathcal{F},\mathbb{P})$ is the set of (complex-valued) random variables with respect to a probability space $(\Om,\mathcal{F},\mathbb{P})$ having finite moments of all orders (see \cite[Examples 1.4.(1) and Definition 11.28]{nica2006lectures}). The classical moment-cumulant relations correspond to
\begin{align*}
    \E_\pi(X_1,\cdots ,X_p) &= \sum_{\rho\leq \pi} k_{\rho}(X_1,\ldots, X_p),\\
    k_\pi(X_1,\cdots ,X_p) &= \sum_{\rho\leq \pi} \E_{\rho}(X_1,\ldots, X_p)\Mob_{\mathcal{P}}(\rho,\pi),
\end{align*}
for $\pi\in \mathcal{P}(p)$, where $\E_{\pi}(X_1,\ldots, X_p):=\prod_{B\in \pi}\E\Big[\prod_{j\in B} X_j\Big]$ and $\Mob_{\mathcal{P}}$ denotes the M\"{o}bius function associated from the lattice $(\mathcal{P}(p),\leq)$ (we refer to \cite[Exercise 10.31-33]{nica2006lectures}).

\medskip

Let us simply write $\tilde{\kappa}_p:=\tilde{\kappa}_{\gamma_p}$. One can also use the notations $\tilde{\kappa}_{\pi}$ and $\varphi_{\pi}$ for a \textit{non-crossing partitions} $\pi\in NC(p)$, via the identification $S_{NC}(\gamma_p)\cong NC(p)$ in \cref{prop:lattice-structure}. In this case, we have 
    $$\tilde \kappa_{\pi}(x_1,\ldots, x_p)=\prod_{B\in \pi}\tilde{\kappa}_{{\rm Card}(B)}((x_j)_{j\in B}),\quad \varphi_{\pi}(x_1,\cdots x_p)=\prod_{B\in \pi}\varphi \Big(\vec{\prod}_{j\in B}\,x_j \Big),$$
where the product $\vec{\prod}$ is taken in increasing order of the indices $j\in B$. Furthermore, the free moment-cumulant relation becomes
    $$\varphi_{\pi}(x_1,\ldots, x_p)=\sum_{\substack{\rho\in NC(p)\\ \rho\leq \pi}}\tilde{\kappa}_{\rho}(x_1,\ldots, x_p), \quad \pi\in NC(p).$$
While these notations based on non-crossing partitions are more standard in the literature, our notation using permutations will become more natural when we consider the tensor versions of moments and free cumulants later.

One of the most important properties of free cumulants is that the free independence is characterized by the vanishing of mixed free cumulants. For details, we refer to \cite{speicher1994multiplicative} and \cite[Theorem 11.16 and 11.20]{nica2006lectures}.

\begin{theorem} \label{thm-FreeSubsets}
Subsets $\W_1,\ldots, \W_L$ in a non-commutative probability space $(\A,\varphi)$ are freely independent if and only if {every mixed free cumulant vanishes}, i.e., for every $p\geq 2$,
\begin{center}
    $\tilde{\kappa}_p(x_{1},\ldots, x_{p})=0$ whenever $x_{j}\in \A_{f(j)}$ and $f(l)\neq f(k)$ for some $l,k\in [p]$.
\end{center}
\end{theorem}

\cref{thm-FreeSubsets}, combined with the moment-free cumulant relations in \cref{eq-FreeMomentCumulant}, provides another characterization of freeness via the joint distribution of free elements, which is particularly useful for obtaining the asymptotic freeness of random matrices. For a set $I$ and a function $f:[p]\to I$, we introduce a partition $\ker{f}\in \mathcal{P}(p)$ which consists of the inverse images of $f$:
\begin{equation} \label{eq-KernelPartition}
    \ker{f}:=\{f^{-1}(i): i\in f([p])\}.
\end{equation}
That is, $l,k\in [p]$ are in the same block of $\ker{f}$ if and only if $f(l)=f(k)$. {For example, if $I = [3]$ and $f: [5] \to I$ is given by $f(1)=f(3)=f(4) = 1$ and $f(2) = f(5) = 3$, then $\ker f = \left\{ \{1,3,4\}, \{2,5\} \right\}$.}

\begin{corollary} \label{cor-freemoment}
Subsets $\W_1,\ldots, \W_L\subset (\A,\varphi)$ are freely independent if and only if for every function $f:[p]\to [L]$ and $x_j\in \W_{f(j)}$,
\begin{equation} \label{eq-FreeMoment}
    \displaystyle\varphi(x_{1}\cdots x_{p}) = \sum_{\substack{\pi\in NC(p), \\ \pi\leq {\ker f}}}\prod_{i\in f([p])} \tilde{\kappa}_{\pi}(x_1,\ldots, x_p) = \sum_{\substack{\pi\in NC(p), \\ \pi\leq {\ker f}}}\prod_{i\in f([p])} \tilde{\kappa}_{\pi|_{f^{-1}(i)}}((x_j)_{j\in f^{-1}(i)}).
\end{equation}
\end{corollary}

\subsection{Unitary and orthogonal Weingarten calculus}

As the main technical tool, we introduce the (graphical) Weingarten calculus \cite{collins2003moments,collins2006integration,collins2010randoma,CMN22} which is frequently applied to compute many kinds of integrals involving Haar random unitary or orthogonal matrices. 

\begin{definition}
Let $p, d$ be positive integers.
\begin{enumerate}
    \item For $\sigma\in S_p$, the \emph{unitary Weingarten function} $\Wg_d^{(U)}(\sigma)$ is the coefficient of the (pseudo) inverse of the function $\sigma\mapsto d^{\#\sigma}$ under the convolution in the group algebra $\Comp[S_p]$. In other words, for the element $\Phi:=\sum_{\sigma\in S_p} d^{\#\sigma}\,\sigma\in \Comp[S_p]$, we have
        $$\Phi^{-1}=\sum_{\sigma\in S_p} \Wg_d^{(U)}(\sigma)\,\sigma.$$

    \item For $\pi,\rho\in \mathcal{P}_2(\pm p)$, the \emph{orthogonal Weingarten function} $\Wg_d^{(O)}(\pi,\rho)$ is the $(\pi,\rho)$-matrix component of the (pseudo) inverse of the linear transformation $\Psi:\Comp[\mathcal{P}_2(\pm p)]\to \Comp[\mathcal{P}_2(\pm p)]$ defined by $\Psi:=\sum_{\pi,\rho\in \mathcal{P}_2(\pm p)}d^{\#(\pi\vee \rho)} E_{\pi,\rho}$. Here we consider $\Comp[\mathcal{P}_2(\pm p)]$ as a vector space with an orthonormal basis $\{e_{\pi}:\pi\in \mathcal{P}_2(\pm p)\}$, and $\{E_{\pi,\rho}\}_{\pi,\rho\in \mathcal{P}_2(\pm p)}$ denotes the family of matrix units. In other words, $\Psi e_{\pi}=\sum_{\rho\in \mathcal{P}_2(\pm p)} d^{\#(\rho\vee \pi)} e_{\rho}$ and 
        $$\Psi^{-1}=\sum_{\pi,\rho\in \mathcal{P}_2(\pm p)}\Wg_d^{(O)}(\pi,\rho) E_{\pi,\rho}.$$
\end{enumerate}
\end{definition}

It is shown \cite{collins2006integration} that if $d\geq p$, then $\Phi$ and $\Psi$ are invertible, so the Weingarten functions ${\rm Wg}_d^{(U)}$ and ${\rm Wg}_d^{(O)}$ are uniquely defined. Since we are mainly interested in the asymptotic limit $d\to \infty$, we only consider the case $d\geq p$ in this paper. Then $\Wg_d^{(U)}(\sigma)$ and $\Wg_d^{(O)}(\pi,\rho)$ are rational functions of $d$ that depend only on the conjugacy class of the permutations $\sigma\in S_p$ and $\pi\rho \in S_{\pm p}$, respectively. Furthermore, these two functions share similar asymptotics
\begin{align}
    \Wg_d^{(U)} (\sigma) &= \Mob(\sigma) d^{-2p+\#(\sigma)} \big(1+O(d^{-2})\big), \label{eq-WeinAsymp1}\\
    \Wg_d^{(O)} (\pi,\rho) &= \Mob(\pi\vee\rho) d^{-2p+\#(\pi\vee\rho)} \big(1+O(d^{-1})\big). \label{eq-WeinAsymp2}
\end{align}
Here we define $\Mob(\pi\vee\rho):=\Mob(c_1\cdots c_l)=\prod_{i=1}^l {\rm Cat}_{|c_i|}$, where $\pi\rho\in S_{\pm p}$ has the cycle decomposition
    $$\pi\rho=c_1c_1'\cdots c_l c_l'$$
and $|c_i|=|c_i'|$ for each $i=1,\ldots, l$, as in \cref{lem-PairingSup}. Note that $\Mob(\pi\vee \rho)$ is well-defined since the M\"{o}bius function depends only on the conjugacy class of its permutation input. For example, if $\pi=\sigma\delta\sigma^{-1}$ and $\rho=\tau\delta\tau^{-1}$ for some permutations $\sigma,\tau\in S_p$, then
\begin{equation} \label{eq-MobPairing}
    \Mob(\pi\vee \rho)=\Mob(\tau^{-1}\sigma)
\end{equation}
since $\pi\rho$ is in the same conjugacy class with $\tau^{-1}\sigma\delta\sigma^{-1}\tau\delta=(\tau^{-1}\sigma)\sqcup \delta(\tau^{-1}\sigma)^{-1}\delta$.

The functions $\Wg^{(U)}_d$ and $\Wg^{(O)}_d$ are related to integrals with respect to the Haar measure on the unitary group $\mathcal{U}_d$ and (real) orthogonal group $\mathcal{O}_d$, respectively, as in the result below.

\begin{theorem} [{\cite{collins2003moments,collins2006integration}}] \label{thm:Weingarten}
Let $p,d$ be positive integers.
\begin{enumerate}
    \item For tuples of indices $\underline{i}=(i_1,\ldots ,i_p)$, $\underline{i}'=(i_1',\ldots, i'_p)$, $\underline{j}=(j_1,\ldots ,j_p)$, $\underline{j}'=(j_1',\ldots, j_p')$ in $[d]^p$ and for a $d\times d$ Haar unitary matrix $U=(U_{ij})$,
    \begin{equation}\label{eq:UWg}
        \int_{\mathcal{U}_d} U_{i_1j_1} \cdots U_{i_pj_p}
        \overline{U}_{i_1' j_1'} \cdots
        \overline{U}_{i_p',j_p'}\, \mathrm{d}U=
    \sum_{\sigma, \tau\in S_{p}} \delta_{\underline{i}', \underline{i}\circ \sigma}\;\delta_{\underline{j}', \underline{j}\circ \tau} \Wg_d^{(U)}(\tau^{-1}\sigma).
    \end{equation}
    Here we identify each tuple as a function, e.g., $\underline{i}:[p]\to [d]$ with $\underline{i}(k)=i_k$.

    \item For tuples of indices $\underline{i}^{\pm}=(i_1, i_{-1},\ldots ,i_p, i_{-p})$, $\underline{j}^{\pm}=(j_1,j_{-1}\ldots ,j_p, j_{-p})$ in $[d]^{2p}$ and for a $d\times d$ Haar orthogonal matrix $O=(O_{ij})$,
    \begin{equation}\label{eq:OWg}
        \int_{\mathcal{O}_d} O_{i_1j_1}O_{i_{-1}j_{-1}} \cdots O_{i_pj_p}O_{i_{-p}j_{-p}}
        \mathrm{d}O=
        \sum_{\pi,\rho \in \mathcal{P}_2(\pm p)}\delta_{\underline{i}^{\pm},\, \underline{i}^{\pm}\circ \pi} \; \delta_{\underline{j}^{\pm},\, \underline{j}^{\pm}\circ \rho} \Wg_d^{(O)}(\pi,\rho).
    \end{equation}
    Here we identify the two tuples as functions $\underline{i}^{\pm},\underline{j}^{\pm}:[\pm p]\to [d]$.
\end{enumerate} 
\end{theorem}

We shall mainly use the \textit{graphical} formulation of these formulae, introduced in \cite{collins2010randoma}; we refer the reader to the review article \cite{CN16} for a pedagogical presentation of the unitary graphical Weingarten formula. {The graphical notation for tensors is due to \cite{penrose1971applications}, see also \cite{bridgeman2017handwaving} or \cite{taylor2024introduction} for modern presentations.}

We can slightly modify all the argument for unitary integration to obtain {graphical representation} for orthogonal Weingarten calculus which greatly simplifies many computations involving integrations over orthogonal matrices. A \textit{removal} $r\in \mathcal{P}_2(\pm p)^2$ describes a method for pairing the decorations of $O$ boxes in a given diagram $\D$. Suppose $\D$ contains exactly $2p$ boxes associated with the same Haar orthogonal matrix $O$, labeled as $\pm 1,\ldots, \pm p$. Each $O$ box is decorated with white and black markings corresponding to its output and input, respectively. A removal $r=(\pi,\rho)\in \mathcal{P}_2(\pm p)^2$ then transforms $\D$ into a new diagram $\D^{(O)}_r=\D_{\pi,\rho}^{(O)}$ through the following steps: first, erase all $O$ boxes while preserving the attached decorations. Then, use $\pi$ to pair the (inner parts of the) white decorations from the erased $k$-th $O$ box with those from the erased $\pi(k)$-th $O$ box. Similarly, $\rho$ connects the black decorations in the same manner.

We restate below \cref{thm:Weingarten} in the graphical language. 

\begin{theorem} \label{thm-GraphWeingarten}
{
Let $\mathcal D$ be a tensor diagram containing $p$ boxes corresponding to Haar-distributed random \emph{unitary} matrices $U \in \mathcal U_d$ and also $p$ boxes corresponding to $\bar U$. Then:
    $$\E_{U}[\D]=\sum_{r=(\sigma, \tau) \in S_p^2}\D_{r}^{(U)} \Wg_d^{(U)}(\tau^{-1}\sigma),$$
{where $\sigma$ (resp. $\tau$) contributes the connection of output (resp. input) nodes between $k$-th $\bar{U}$ box and $\sigma(k)$-th (resp. $\tau(k)$-th) $U$ box for each $k\in [p]$.}
}

In a similar manner, let $\mathcal D$ be a tensor diagram containing $2p$ boxes corresponding to Haar-distributed random \emph{orthogonal} matrices $O \in \mathcal O_d$. Then:
    $$\E_{O}[\D]=\sum_{r=(\pi,\rho)\in \mathcal{P}_2(\pm p)^2}\D_r^{(O)} \Wg_d^{(O)}(\pi,\rho).$$
\end{theorem}

{
Let us illustrate the result above by a simple example, which is relevant to the setting of this work. We shall consider both the unitary and the orthogonal case, in order to emphasize the differences between them. 

Let $X \in \M{d} \otimes \M{d}$ be a bipartite matrix. We are interested in computing the average 
$$\E[(U \otimes U) X (U \otimes U)^*]=\E[(U \otimes U) X (\bar{U} \otimes \bar{U})^{\top}]\in \M{d}\otimes \M{d}$$
where $U$ is a Haar-distributed random unitary, respectively orthogonal, matrix. We depict the diagram $\mathcal D$ corresponding to this computation in \cref{fig:E-UU-X-UUstar}.

\begin{figure}[htb]
    \centering
    \includegraphics[width=0.35\linewidth]{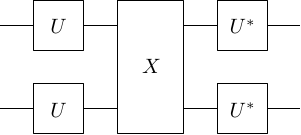}
    \caption{Diagram for the matrix product $(U \otimes U) X (U \otimes U)^*$.}
    \label{fig:E-UU-X-UUstar}
\end{figure}

Let us start with the \emph{unitary case}. The (graphical) unitary Weingarten formula from \cref{thm-GraphWeingarten} expands the expectation as a sum over a pair $\sigma, \tau \in S_2 = \{\id, (12)\}$. The four terms are of the form $\D_{\sigma,\tau}^{(U)} \Wg_d^{(U)}(\sigma,\tau)$, where $\D_{\sigma,\tau}^{(U)}$ is the diagram $\mathcal D$ from \cref{fig:E-UU-X-UUstar} after the removal operation for the permutations $\sigma, \tau$, and $\Wg_d^{(U)}(\sigma,\tau)$ is the corresponding unitary Weingarten function. Note that the permutation $\sigma$ corresponds to the resulting operator, while the permutation $\tau$ is responsible for the scalar coefficient (involving the matrix $X$) in front of it. 

\begin{table}[!htb]
\begin{tabular}{r|c|c|}
\cline{2-3}
& \cellcolor[HTML]{EFEFEF}\textcolor{blue}{$\tau = \id$} & \cellcolor[HTML]{EFEFEF}\textcolor{blue}{$\tau=(12)$} \\ \hline
\multicolumn{1}{|r|}{\cellcolor[HTML]{EFEFEF}\textcolor{red}{$\sigma=\id$}} & \includegraphics[width=130pt,valign=c,margin=10pt]{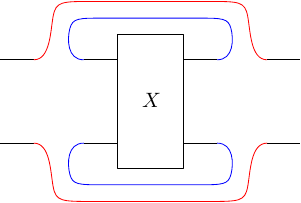} $\cdot \frac{1}{d^2-1}$ & \includegraphics[width=130pt,valign=c,margin=10pt]{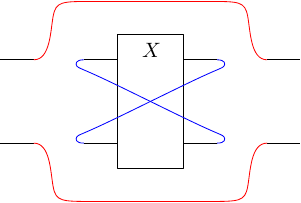} $\cdot \frac{-1}{d(d^2-1)}$                         \\ \hline
\multicolumn{1}{|r|}{\cellcolor[HTML]{EFEFEF}\textcolor{red}{$\sigma=(12)$}} & \includegraphics[width=130pt,valign=c,margin=10pt]{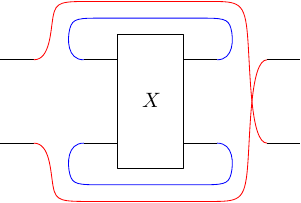} $\cdot \frac{-1}{d(d^2-1)}$                         & \includegraphics[width=130pt,valign=c,margin=10pt]{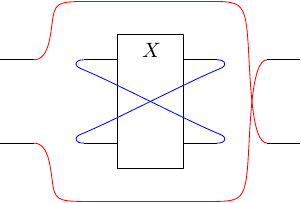} $\cdot \frac{1}{d^2-1}$                         \\ \hline
\end{tabular}
\end{table}

Gathering the four contributions, we obtain: 
$$\E_{\mathcal U_d} [(U \otimes U) X (U \otimes U)^*] = \frac{d\Tr X - \Tr(F_dX)}{d(d^2-1)} \cdot  I_{d^2} + \frac{d \Tr(F_dX) - \Tr X}{d(d^2-1)} \cdot  F_d,$$
where $F_d:=\sum_{i,j=1}^d E_{ij}\otimes E_{ji}$.

\medskip 

Let us consider now the \emph{orthogonal case}: the matrix $U$ is sampled from the Haar measure on the orthogonal group $\mathcal O_d$. In this case, the graphical orthogonal Weingarten formula from \cref{thm-GraphWeingarten} expands the expectation as a sum over two pairings $\pi, \rho \in \mathcal P_2(\pm 2)$. Note that the set $\mathcal P_2(\pm 2)$ has three elements, so the total sum has $9$ terms. The extra terms come from the possible pairing of a $U$ box with the other $U$ box, which was impossible in the unitary case. {Importantly, we label the two orthogonal $U$ boxes on the left of $X$ by, respectively $U^{(+1)}$ and $U^{(+2)}$, and those on the right of the $X$ box by, respectively, $U^{(-1)}$ and $U^{(-2)}$ (see \cref{fig:E-UU-X-UUstar}); this labeling will allow us to explicitly describe how the pairings $\pi$ and $\rho$ from the orthogonal Weingarten calculus act on the diagram.} We present in \cref{fig:example-Wg-O} two such terms. 

\begin{figure}[!htb]
    \centering
    \includegraphics[width=130pt,align=c]{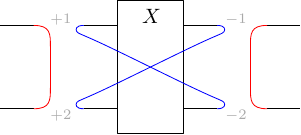} \qquad\qquad\qquad \includegraphics[width=130pt,align=c]{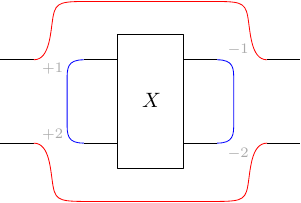}
    \caption{Two removal diagrams obtained from \cref{fig:E-UU-X-UUstar}, using different pairings $\pi,\rho$. On the left panel, we have $(\textcolor{red}{\pi}, \textcolor{blue}{\rho}) = (\textcolor{red}{(1\,2)(-1\,-2)}, \textcolor{blue}{(1\,-2)(-1\,2)})$. On the right panel, we have $(\textcolor{red}{\pi}, \textcolor{blue}{\rho}) = (\textcolor{red}{(1\,-1)(2\,-2)}, \textcolor{blue}{(1\,2)(-1\,-2)})$.}
    \label{fig:example-Wg-O}
\end{figure}

Gathering all the contributions, we have: 
\begin{align*}
    \E_{\mathcal O_d} [(U \otimes U) X (U \otimes U)^*] = &\frac{(d+1)\Tr X - \Tr(F_d X) - \Tr(d\omega_d X)}{d(d+2)(d-1)} \cdot  I_{d^2}\\
    &+ \frac{-\Tr X  + (d+1) \Tr(F_d X) - \Tr(d\omega_d X)}{d(d+2)(d-1)} \cdot  F_d\\
    &+ \frac{-\Tr X - \Tr(F_d X) +(d+1) \Tr(d\omega_d X)}{d(d+2)(d-1)} \cdot  d\omega_d,
\end{align*}
where $d\omega_d$ is the \emph{unnormalized maximally entangled state} 
$$d\omega_d = \sum_{i,j=1}^d e_i \otimes e_i \cdot e_j^* \otimes e_j^* = \sum_{i,j=1}^d E_{ij} \otimes E_{ij} \in \M{d} \otimes \M{d}$$
and the value of the orthogonal Weingarten function, in the case $p=2$, reads (see \cite{collins2006integration} or \cite[Example 3.1]{collins2009some}):
\begin{equation} \label{eq-OWg-dim2}
    \Wg_d^{(O)}(\pi,\rho) = \begin{cases}
    \frac{d+1}{d(d+2)(d-1)} &\text{ if } \pi = \rho\\
    \frac{-1}{d(d+2)(d-1)} &\text{ if } \pi \neq \rho.
\end{cases}
\end{equation}

}

\subsection{Unitary and orthogonal invariant random matrices}

All the random matrices in this paper are assumed to have finite moments of all orders:
    $$X\in \M{N}\otimes L^{\infty -}(\mathbb{P})\cong \mathcal{M}_N\big(L^{\infty-}(\mathbb{P})\big).$$
Note that $\big(\mathcal{M}_N (L^{\infty-}(\mathbb{P})),\E\circ \tr\big)$ is a $*$-probability space, where $\tr:=\frac{1}{N}\Tr$ is the \emph{normalized trace} (see \cite[Example 1.4 (3)]{nica2006lectures}). 

For each $N\geq 1$, let us denote by $\W_N=\{X_{j,N}\}_{j\in J}$ a family of $N\times N$ random matrices. Then we say that $\W_N$ \textit{converges in distribution} as $N\to \infty$ if there is a family $\W=\{x_j\}_{j\in J}$ in a noncommutative probability space $(\A,\varphi)$ such that, for all $p\geq 1$ and $j_1,\ldots, j_p\in J$, we have
    $$\lim_{N\to \infty}\E[\tr(X_{j_1,N}\cdots X_{j_p,N})] = \varphi(x_{j_1} \cdots x_{j_p}).$$
by linearity, this is equivalent to the asymptotic behavior
    $$\lim_{N\to \infty}\E[\tr(P(\W_N))]=\varphi(P(\W))$$
for every noncommutative polynomial $P$ in a finite number of formal variables indexed by $J$. Let us simply write $\W_N\xrightarrow{\text{distr}} \W$. Families $\W_N^{(1)},\ldots, \W_N^{(L)}$ of $N\times N$ random matrices are called \textit{asymptotically free} if $\W_N^{(1)},\ldots, \W_N^{(L)}\xrightarrow{\text{distr}} \W^{(1)},\ldots, \W^{(L)}$ jointly and the families $\W^{(1)},\ldots, \W^{(L)}$ are freely independent.

We remark that these definitions describe the asymptotic behavior in average (with respect to $\E$). Several other modes of convergences have been considered in literature, including
\begin{itemize}
    \item \textit{convergence in probability}: $\tr(X_{j_1,N}\cdots X_{j_p,N})\to \varphi(x_{j_1} \cdots x_{j_p})$ in probability as $N\to \infty$, for all $j_1,\ldots, j_p\in J$,

    \item \textit{almost sure convergence}: $\tr(X_{j_1,N}\cdots X_{j_p,N})\to \varphi(x_{j_1} \cdots x_{j_p})$ almost surely as $N\to \infty$, for all $j_1,\ldots, j_p\in J$.
\end{itemize}
Many results in Random Matrix Theory, stated in terms of the average distribution, also have counterparts in almost sure behavior. In this paper, we primarily focus on the \textit{averaged behavior}, though many of the results have almost sure analogs. We refer to {\cref{rmk-UIasympFree,rmk-UITensorAlmostSure,rmk-LocUIAlmostSure,prop-FactorConvInProb,prop-TensorFactConvProb}} for further discussion on this matter.

\medskip

We are mainly interested in the following symmetry conditions on random matrices.

\begin{definition}
Let $\W_N$ be a family of $N\times N$ random matrices.
\begin{enumerate}
    \item $\W_N$ is called \emph{unitary invariant (UI)} (resp. \emph{orthogonal invariant (OI)}) if the joint probability law of $\W_N$ is the same with that of $U\W_N U^*:=\{UXU^*: X\in \W_N\}$ for all unitary matrices $U\in \mathcal{U}_N$ (resp. orthogonal matrices $U\in \mathcal{O}_N$).

    \item Suppose $N=d_1\cdots d_r$ and let us identify the matrices as elements in $\bigotimes_{s=1}^r \M{d_s}\cong \M{N}$. Then $\W_N$ is called \emph{local-unitary invariant {(LUI)}} (resp. \emph{local-orthogonal invariant {(LOI)}}) if $\W_N$ has the same probability law with $U\W_N U^*$ where $U=\bigotimes_{s=1}^r U_s$ is taken over arbitrary tensor product of unitaries $U_s\in \mathcal{U}_{d_s}$ (resp. orthogonal matrices $U_s\in \mathcal{O}_{d_s}$) for $s\in [r]$.
\end{enumerate}

\end{definition}

\begin{example} \label{ex-UI}
Let us collect several typical examples of (local-) UI and OI random matrices.
\begin{enumerate}
    \item GUE matrices, square Ginibre matrices with i.i.d.~standard complex Gaussian entries, complex Wishart matrices, and Haar random unitary matrices are UI. Moreover, GOE matrices, square  Ginibre matrices with i.i.d.~(real) Gaussian entries, real Wishart matrices, and Haar random orthogonal matrices are OI but not UI.

    \item For any family $\W$ of random matrices and a Haar random unitary matrix $U$ independent from $\W$, $U\W U^*$ is UI. Furthermore, $U\W_N U^*$ is LUI if $N=d_1\cdots d_r$ and $U=\bigotimes_{s=1}^r U_s$ is a Haar random \textit{local-unitary} matrix. That is, $U_1,\ldots, U_r$ are independent Haar random unitaries of size $d_1,\ldots, d_r$, respectively, which are also independent from $\W$. We also have the OI and LOI analogs in straightforward ways.

    \item If $\W^{(1)},\ldots, \W^{(r)}$ are independent families of UI (resp.~OI) random matrices of sizes $d_1,\ldots, d_r$, respectively, then $\W=\{\bigotimes_{s=1}^rX_s:X_s\in \W^{(s)}\}$ is LUI (resp.~LOI).
    
    \item The set $\W_1\cup \W_2$ may not be UI even if $\W_1$ and $\W_2$ are each UI. For example, for a Haar unitary matrix $U$, neither the set $\{U, U^{\top}\}$ nor $\{U,\overline{U}\}$ is UI while each individual matrix $U, U^{\top}$ and $\overline{U}$ is Haar-distributed (\cite[Lemma 13]{MP16}).

    \item If $\W$ is UI, then $\W\cup \W^*$ is UI. Similarly, if $\W$ is OI, then $\W\cup \W^{\top}\cup \overline{\W}\cup \W^*$ is OI.

    \item If $\W$ is LOI, then the family of \emph{partial transposes} $\bigcup_{\underline{\nu}\in \{\id, \top\}^r}\W^{\underline{\nu}}$ is LOI where $\W^{\underline{\nu}}:=\{(\nu_1\otimes\cdots \otimes \nu_r)(X):X\in \W\}$. We refer to \cref{sec-PTTensorfree} for further discussion on partial transposes.
\end{enumerate}
\end{example}

Let us record the following invariance-preserving property of an independent family of random matrices for later use.

\begin{proposition} \label{prop-IndepUI}
Suppose $\W_0, \W_1,\ldots, \W_L$ are independent families of $N\times N$ random matrices, $\W_1,\ldots \W_L$ are unitary invariant, and let $U$ be a $N\times N$ Haar unitary matrix independent from $\bigcup_{i=0}^L \W_i$. Then $U\W_0 U^*, \ldots, U\W_L U^*$ are independent UI families. In particular, $\bigcup_{i=1}^L\W_i$ is UI.
\begin{proof}
$\W_0,\ldots \W_L$ have the same joint distribution with $\W_0, U_1 \W_1U^*, \ldots U_L\W_L U_L^*$ where $U_1,\ldots U_L$ are independent Haar random unitaries that are independent from $\{U\}\cup (\bigcup_{i=0}^L \W_i)$. Then the assertion follows from the fact that $U, UU_1, \ldots UU_L$ are independent Haar unitary matrices. Indeed, for any continuous functions $f_0,\ldots, f_L:\mathcal{U}_N\to \Comp$, we have
\begin{align*}
    \E[f_0(U)f_1(UU_1)\cdots f_L(UU_L)]&= \E_U\big[f_0(U)\,\E_{U_1,\ldots, U_L}[f_1(UU_1)\cdots f_L(UU_L)|\,U]\big]\\
    &=\E[f_0(U)]\,\E[f_1(U_1)]\cdots \E[f_L(U_L)],
\end{align*}
where the last equality follows from \cite[Example 4.1.7]{durrett2019probability} and the invariance property of Haar measures. Therefore, $(U,UU_1,\ldots, UU_L)$ has the same joint probability law with $(U,U_1,\ldots, U_L)$.
\end{proof}
\end{proposition}

Note that the UI condition above is essential to obtain the independence between $U\W_0 U^*, \ldots, U\W_L U^*$. For example, the deterministic diagonal matrix units $E_{11},\ldots E_{NN}\in \M{N}$ are (probabilistically) independent while $U E_{11}U^*, \ldots U E_{NN}U^*$ are not since $\sum_{i=1}^N U E_{ii}U^*\equiv I_N$. Furthermore, we can obtain the OI, LUI, and LOI analogs of \cref{prop-IndepUI} in an obvious manner.

Now let us introduce several technical conditions on families of random matrices, which are frequently considered in relation with their asymptotic behaviors. First, for a permutation $\sigma\in S_p$ we define the \textit{(unitary) trace invariant} $\tr_{\sigma}:\M{N}^p\to \Comp$ as a multilinear functional induced from the $*$-probability space $(\M{N},\tr)$ of $N\times N$ matrices as in \cref{eq-MomentPerm}. In other words, we have $\tr_{\sigma}=\frac{1}{N^{\#\sigma}}\Tr_{\sigma}$, where
\begin{equation} \label{eq-SingleTraceInv}
    \Tr_{\sigma}(X^{(1)},\ldots, X^{(p)}):=\prod_{\substack{c\in \cyc(\sigma)\\ c=(i_1\,i_2\,\cdots\,i_n)}} \Tr(X_{i_1}X_{i_2}\cdots X_{i_n})=\sum_{\underline{i},\underline{j}\in [N]^p}\left(\prod_{k=1}^p X^{(k)}_{i_k,j_k}\right)\delta_{\underline{j},\, \underline{i}\circ \sigma}.
\end{equation}

\begin{definition} \label{def-FactProp-BddMoment}
Let $\W_N$ be a family of $N\times N$ random matrices.
\begin{enumerate}
    \item $\W_N$ is said to satisfy the \emph{factorization property} if, for all $X_1,\ldots, X_p,Y_1,\ldots, Y_q\in \W_N$, $\alpha\in S_p$, and $\beta\in S_q$, we have as $N\to \infty$,
    \begin{equation} \label{eq-condition-Fact}
        \E[\tr_{\alpha\sqcup \beta}(X_1,\ldots, X_p,Y_1,\ldots, Y_q)]= \E[\tr_{\alpha}(X_1,\ldots, X_p)]\E[\tr_{\beta}(Y_1,\ldots, Y_q)]+o(1)
    \end{equation}

    \item $\W_N$ is said to satisfy the \emph{bounded moments property} if
    \begin{equation} \label{eq-condition-Bdd}
        \sup_N \big|\E[\tr_{\sigma}(X_1,\ldots, X_p)]\big|<\infty, \quad X_1,\ldots, X_p\in \mathcal{W}_N, \;\;\sigma\in S_p .
    \end{equation}
\end{enumerate}
\end{definition}

By the multilinearity of the trace invariant, the factorization property is equivalent to
\begin{equation} \label{eq-condition-1}
    \E[\tr(P_1(\mathcal{W}_N))\cdots \tr(P_l(\mathcal{W}_N))]=\prod_{j=1}^l \E[\tr(P_j(\mathcal{W}_N))]+o(1).
\end{equation}
for every $l\geq 1$ and for any non-commutative polynomials $P_1,\ldots, P_l$ which has been considered in \cite{collins2003moments,male2020traffic,CDM24}. Furthermore, by the classical moment-cumulant relation, this condition is also equivalent to
\begin{equation} \label{eq-condition-1'}
    k_l(\tr(P_1(\mathcal{W}_N)), \ldots, \tr(P_l(\mathcal{W}_N)))\to 0
\end{equation}
for $l\geq 2$ where $k_l$ is the classical cumulant, as considered in \cite[Definition 2.1]{CMSS07}. On the other hand, $\W_N$ satisfies the bounded moments property if $\sup_N \E[\|X\|_{\infty}]<\infty$ holds for all $X\in \mathcal{W}_N$, or more generally,
\begin{equation} \label{eq-condition-5}
    \sup_N \E[\tr|X|^p]<\infty, \quad p\geq1,\;\; X\in \mathcal{W}_N,
\end{equation}
by the H\"{o}lder inequality. Note that \cref{eq-condition-5} holds whenever $\mathcal{W}_N$ converges in $*$-distribution. 

Actually, the factorization property of random matrices is closely related to convergence \textit{in probability} (denoted by $\Plim$), as stated in the proposition below.

\begin{proposition} \label{prop-FactorConvInProb}
Suppose $\W_N$ \emph{converges in $*$-distribution} to $\W$, i.e.~$\W_N, \W_N^*$ converge jointly in distribution to $\W,\W^*$ (in the sense of expectation) for a family $\W$ in a $*$-probability space $(\A,\varphi)$ and its adjoint set $\W^*=\{x^*:x\in \W\}$. Then the same convergence $\W_N,\W_N^*\xrightarrow{\text{distr}} \W,\W^*$ holds in probability if and only if $\W_N\cup \W_N^*$ satisfies the factorization property.
\end{proposition}

\begin{proof}

The sufficiency is clear: for $X_1,\ldots, X_p\in \W_N\cup \W_N^*$, the factorization property implies that
    $${\rm var}\big(\tr(X_1\cdots X_p)\big)=k_2\big(\tr(X_1\cdots X_p),\tr(X_p^*\cdots X_1^*)\big)\to 0 \quad \text{as $N\to \infty$,}$$
so $\tr(X_1\cdots X_p)$ converges in probability:
$$\Plim_{N \to \infty}\tr(X_1\cdots X_p) = \lim_{N\to \infty}\E[\tr(X_1\cdots X_p)]$$

For the necessity, let $X_1,\ldots, X_L\in \W_N\cup \W_N^*$ and suppose $X_1,\ldots, X_L\to x_1,\ldots, x_L$ jointly in $*$-distribution, both in expectation and in probability. In order to show \cref{eq-condition-1}, we may take $P_j$'s as monomials in $X_1,\ldots, X_L$. Then, for the random variables $Z_N:=\prod_{j=1}^l \tr(P_j(X_1,\ldots, X_L))$, we can repeatedly apply H\"{o}lder's inequality to obtain that 
    $$\E[|Z_N|^2]\leq \prod_{j=1}^l \E\Big[ \big|\tr\big(P_j(X_1,\ldots, X_L)\big)\big|^{2l}\Big]^{1/l}\leq \prod_{j=1}^l \E\Big[ \tr\big(|P_j(X_1,\ldots, X_L)|^{2l}\big)\big|\Big]^{1/l}$$
where the matrix H\"{o}lder's inequality $|\tr(Y)|\leq (\tr|Y|^{2l})^{\frac{1}{2l}}$, $Y\in \M{N}$, was applied for the last inequality above.
Since each $|P_j(X_1,\ldots, X_L)|^{2l}=\big(P_j(X_1,\ldots, X_L)^* P_j(X_1,\ldots, X_L)\big)^l$ is a monomial in $X_1,\ldots, X_L,X_1^*,\ldots, X_L^*$, we have
    $$\limsup_{N\to \infty} \E[|Z_N|^2]\leq \prod_{j=1}^l \varphi\big((P_j(x_1,\ldots, x_L)^*P_j(x_1,\ldots, x_L))^l\big)^{1/l},$$
and in particular, $\sup_N \E[|Z_N|^2]<\infty$. 
Therefore, the random variables $(Z_N)_{N\geq 1}$ are uniformly integrable, and one has
    $$\lim_{N\to \infty}\E[Z_N]=\E\big[\Plim_{N \to \infty} Z_N\big]=\prod_{j=1}^l \varphi(P_j(x_1,\ldots, x_L))$$
from the Vitali convergence theorem (\cite[Theorem 4.6.2 and 4.6.3]{durrett2019probability}).
\end{proof}

On the other hand, a stronger condition called the \textit{bounded cumulants property} \cite{MP13,MP16,MP24} has been introduced to obtain almost-sure (and even stronger) behavior. $\W_N$ satisfies bounded cumulants property if it satisfies the bounded moments property \cref{eq-condition-Bdd} and, for every $l\geq 2$ and for any non-commutative polynomials $P_1,\ldots, P_l$, we have
\begin{equation} \label{eq-condition-BCP}
    \sup_N \big|k_l(\Tr(P_1(\mathcal{W}_N)), \ldots, \Tr(P_l(\mathcal{W}_N))\big)\big|<\infty.
\end{equation}
Clearly the bounded cumulants property implies the factorization property from \cref{eq-condition-1'}. If $\W_N$ satisfies bounded cumulants property and is convergent in $*$-distribution, then the same convergence holds almost surely, by the fact that ${\rm var}\big(\tr(P(\W_N)) \big)=O(N^{-2})$ holds for all non-commutative polynomials $P$ and the application of Borel-Cantelli lemma. Furthermore, it is known \cite[Proposition 16]{MP16} that bounded cumulants property implies
\begin{equation} \label{eq-condition-BCP'}
    \E[\tr\big(P_1(\mathcal{W}_N))\cdots \tr(P_l(\mathcal{W}_N))]=\prod_{j=1}^l \E[\tr(P_l(\mathcal{W}_N))]+O(N^{-2}).
\end{equation}
for all non-commutative polynomials $P_1,\ldots, P_l$. A wide range of random matrices satisfies bounded cumulants property, including GUE, GOE, (real and complex) Wishart matrices, Haar random unitary / orthogonal matrices, and every family of random matrices of the form $U\mathcal{W}_N U^*$ where $\mathcal{W}_N$ is a family of deterministic matrices which is convergent in $*$-distribution and $U$ is a Haar random unitary / orthogonal matrix. Furthermore, if $\mathcal{W}_N$ is UI and satisfies the bounded cumulants property, then $\mathcal{W}_N\cup \mathcal{W}_N^{\top}$ satisfies the bounded cumulants property \cite[Corollary 29]{MP16}.

\medskip

One of our main results \cref{thm-indepTranspose}  characterizes the behavior of independent random matrices and their (partial) transposes within tensor product structures. Regarding independent random matrices, let us recall \textit{Voiculescu's asymptotic freeness theorem} \cite{Voi91, Voi98, collins2003moments} which characterizes the asymptotic freeness of independent UI random matrices.

\begin{theorem} \label{thm-UIasympfree}
Let $\W_N^{(1)},\ldots, \W_N^{(L)}$ be {independent} families of $N\times N$ random matrices such that:
\begin{enumerate}
    \item Each family, except possibly one, is UI.
    
    \item Each family has a first order limit distribution, i.e., it converges in distribution and satisfies the factorization property \cref{eq-condition-1}.
\end{enumerate}
Then $\{\W_N^{(i)}\}_{i\in [L]}$ are asymptotically free as $d\to \infty$. In particular, the family $\bigcup_{i=1}^L \W_N^{(i)}$ converges jointly in distribution and satisfies the factorization property.
\end{theorem}

\begin{remark} \label{rmk-UIasympFree}
Note that \cref{thm-UIasympfree} already includes the result on $*$-distribution behavior: if each $\W_N^{(i)}$ converges in $*$-distribution (both in expectation and in probability), then $(\W_N^{(i)})_{i\in [L]}$ are asymptotically $*$-free (both in expectation and in probability), just by replacing $\W_N$ with $\widetilde{W}_N=\W_N\cup \W_N^*$. Furthermore, we can show that if each family satisfies the stronger condition \cref{eq-condition-BCP'}, then $(\W_N^{(i)})_{i\in [L]}$ are almost surely asymptotically $*$-free.
\end{remark}

Asymptotic behaviors of independent random matrices have been studied across a wide range of classes, including Wigner matrices \cite{Dyk93}, OI random matrices \cite{collins2006integration}, and permutation invariant random matrices \cite{male2020traffic,CDM24}. For our purposes, we recall the asymptotic freeness result of OI random matrices.

\begin{theorem} [\cite{collins2006integration}] \label{thm-OIAsympFree}
Let $(\W_N^{(i)})_{i\in [L]}$ be {independent} families of $N\times N$ random matrices such that each family, except possibly one, is OI.
\begin{enumerate}
    \item If each family $\W_N^{(i)}$ converges in distribution and satisfies the factorization property, and if $\W_N^{(i)}\cup (\W_N^{(i)})^{\top}$ satisfies the bounded moments condition \cref{eq-condition-Bdd},
    then $(\W_N^{(i)})_{i\in [L]}$ are asymptotically free, and the family $\bigcup_{i=1}^L \W_N^{(i)}$ satisfies the factorization property.
    
    \item If each family $\W_N^{(i)}\cup (\W_N^{(i)})^{\top}$ converges in distribution and satisfies the factorization property, then the families $\big(\W_N^{(i)}\cup (\W_N^{(i)})^{\top})_{i\in [L]}$ are asymptotically free, and the family $\bigcup_{i=1}^L \big(\W_N^{(i)}\cup (\W_N^{(i)})^{\top}\big)$ satisfies the factorization property.
\end{enumerate}
\end{theorem}

Note that in \cref{thm-OIAsympFree}, unlike in \cref{thm-UIasympfree}, conditions must be imposed not only on $\W_N^{(i)}$ but also on $(\W_N^{(i)})^{\top}$ to determine the asymptotic behavior of $\bigcup_{i=1}^L \W_N^{(i)}$. This requirement is essential (see the remark below \cite[Theorem 5.2]{collins2006integration}) and can be justified by considering the \textit{orthogonal trace invariants}. 

{
For a pairing $\pi\in \mathcal{P}_2(\pm p)$, define the \textit{(orthogonal) trace invariant} $\Tr_{\pi\vee\delta}:(\M{N})^p\to \Comp$ by
\begin{equation} \label{eq-OITraceInv}
    \Tr_{\pi\vee\delta}(X^{(1)},\ldots, X^{(p)}):=\sum_{\forall\, k,\,i_{\pm k}\in [N]} \left(\prod_{k=1}^p X^{(k)}_{i_{k},i_{-k}}\right)\delta_{\underline{i}^{\pm},\, \underline{i}^{\pm}\circ \pi}
\end{equation}
and its normalized value $\tr_{\pi\vee\delta}:=\frac{1}{N^{\#(\pi\vee\delta)}}\Tr_{\pi\vee\delta}$ (recall that $\underline{i}^{\pm}(k):=i_k$ for $k\in [\pm p]$ as in \cref{eq:OWg}). We can imagine the diagrams of $X^{(k)}$'s whose outputs and inputs are labeled as $k$ and $-k$, resp., and then we connect the legs according to $\pi$. For example, we have $\Tr_{\delta\vee \delta}(X_1,\ldots, X_p)=\prod_{k=1}^p \Tr(X_k)$, $\Tr_{\gamma_p\delta\gamma_p^{-1}\vee\delta}(X_1,\ldots, X_p)=\Tr(X_1\cdots X_p)$, and
    \begin{equation}\label{eq:orthogonal-trace-invariant}
        \Tr_{(1,2)(-1,-2)(-3,-5)(5,4)(-4,3)\vee\delta}(X_1,X_2,X_3,X_4,X_5)=\Tr(X_1 X_2^{\top}) \Tr(X_3 X_5^{\top}X_4).
    \end{equation}    
\begin{figure}
    \centering
    \includegraphics[width=0.8\linewidth]{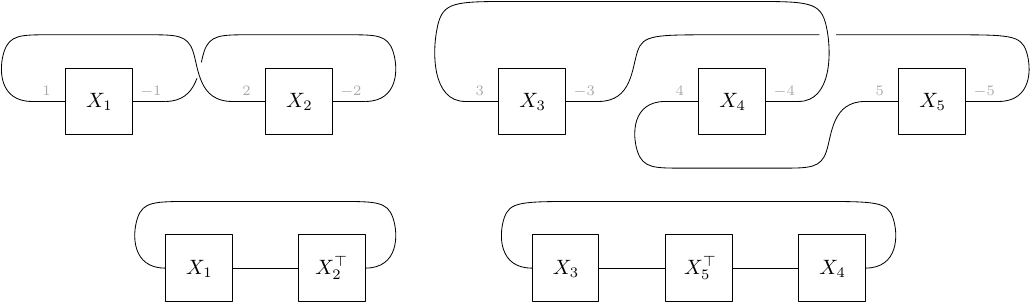}
    \caption{The trace invariant from \cref{eq:orthogonal-trace-invariant}. In the top row, the invariant is constructed from the paring $(1,2)(-1,-2)(-3,-5)(5,4)(-4,3) \in \mathcal P_2(\pm 5)$. The bottom diagram is obtained from the top diagram by moving the boxes and represents the invariant in terms of the matrices and their transposes.}
    \label{fig:orthogonal-trace-invariant}
\end{figure}
See \cref{fig:orthogonal-trace-invariant} for a graphical representation of this invariant. More generally, suppose $\pi=\eps\sigma\delta\sigma^{-1}\eps$ as in \cref{lem-pairings}. Then we have ({\cite[Lemma 5]{MP13}})
\begin{equation} \label{eq-OITraceInvTransp}
    \Tr_{\pi\vee\delta} (X_1,\ldots, X_p)=\Tr_{\sigma}(X_1^{f_1},\ldots, X_p^{f_p}), \quad \tr_{\pi\vee\delta} (X_1,\ldots, X_p)=\tr_{\sigma}(X_1^{f_1}, \ldots, X_p^{f_p})
\end{equation}
where $f_k=\begin{cases}
    1 & \text{if $\eps(k)=k$,}\\
    \top & \text{if $\eps(k)=-k$,}
\end{cases}$ for $k=1,\ldots, p$. In particular, $\tr_{\sigma\delta\sigma^{-1}\vee \delta}=\tr_{\sigma}$ for $\sigma\in S_p$.

Note that $\tr_{\pi\vee\delta}$ satisfies the invariance property under conjugation with respect to orthogonal matrices:
\begin{equation} \label{eq-OITraceInvProp}
    \tr_{\pi\vee\delta} (X_1,\ldots, X_p)=\tr_{\pi\vee\delta} (OX_1 O^{\top},\ldots, O X_p O^{\top}), \quad X_1,\ldots, X_p\in \M{N},\;\; O\in \mathcal{O}_N.
\end{equation}
Conversely, we can show that any multilinear functional $\mathcal{L}:\M{N}^p\to \Comp$ satisfying the invariance property as above can be written as a linear combination of $\{\tr_{\pi\vee \delta}:\pi\in \mathcal{P}_2(\pm p)\}$ {by taking the expectation in  \cref{eq-OITraceInvProp} with respect to a Haar random orthogonal matrix $O$ and applying orthogonal Weingarten calculus (\cref{thm:Weingarten} (2))}. In other words, the family $\{\tr_{\pi\vee \delta}:\pi\in \mathcal{P}_2(\pm p)\}$ contains the full information on the behavior of OI random matrices. Moreover, by the relation \cref{eq-OITraceInvTransp}, this information is equivalently encoded by $\{\tr_{\sigma}:\sigma\in S_p\}$ for OI families and their transposes.
}

\medskip

We finally recall the result in \cite{MP16} describing the asymptotic freeness between UI random matrices and their transposes.

\begin{theorem} \label{thm-UITranspFreeSingle}
Let $\W_N$ be a family of $N\times N$ UI random matrices which converges in distribution and satisfies the factorization property \cref{eq-condition-Fact}. Then $\W_N$ and $\W_N^{\top}$ are asymptotically free as $N\to \infty$, and $\W_N \cup \W_N^{\top}$ satisfies the factorization property.
\end{theorem}

In fact, the original result of \cite{MP16} imposes the stronger condition of bounded cumulants property (\cref{eq-condition-BCP}) to analyze the \textit{second-order behavior} (i.e., fluctuations of moments) of random matrices. The statement in \cref{thm-UITranspFreeSingle} follows as a special case of our more general result, \cref{thm-LocUITranspose}.

\section{Algebraic tensor probability spaces} \label{sec-TensorNCPS}

The main motivation for the current work is to understand the large dimension limits of random matrices acting on a tensor product of vector spaces. This tensor product structure comes with a natural distributional invariance, that induced by tensor products of Haar-distributed random unitary matrices (also known as \emph{local-unitaries} in quantum information theory \cite{chitambar2014everything}). The natural tensor invariants with respect to the action of the local-unitary group are the so-called \emph{(tensor) trace invariants} \cite{vrana2011algebraically, CGL23a, CGL23b, collins2024free, BB24}. We shall first review these invariants for (multipartite) matrices that motivate our main definition later. 

For $p\geq 1$ and for $r$-partite matrices $X^{(1)},\ldots, X^{(p)}\in \bigotimes_{s=1}^r \M{d_s}$, define the \emph{tensor trace invariant} or \emph{local unitary trace invariant} associated to a $r$-tuple of permutations $\underline{\alpha}=(\alpha_1,\ldots, \alpha_r)\in (S_p)^r$ by
\begin{equation}\label{eq:def-trace-invariant}
    \Tr_{\underline{\alpha}}(X^{(1)},\ldots, X^{(p)}):=\sum_{\text{all indices}}\left(\prod_{k=1}^p X_{i_1^{(k)}\ldots i_r^{(k)}, j_1^{(k)}\ldots j_r^{(k)}}^{(k)}\right)\prod_{s=1}^r \left(\prod_{k=1}^p \delta_{i_s^{(\alpha_s(k))},j_s^{(k)}}\right).
\end{equation}
{The equation above is the tensor version of \cref{eq-SingleTraceInv}.} Note that the $r$ permutations connect, on each of the $r$ tensor factors, the $p$ matrices in the following way: for all $i \in [p]$ and all $s \in [r]$, the $s$-th input of the matrix $X^{(i)}$ is connected to the $s$-th output of the matrix $X^{(\alpha_s(i))}$.

For example, for three bipartite ($r=2$) matrices $X^{(1)},X^{(2)},X^{(3)}\in \M{d}^{\otimes 2}$ and two permutations $\alpha=(1,2)(3), \beta=(1)(2,3)\in S_3$, we have
    $$\Tr_{\alpha,\beta}(X^{(1)},X^{(2)},X^{(3)})=\sum_{a_1,a_2,a_3,b_1,b_2,b_3\in [d]}X^{(1)}_{a_1b_1, a_2b_1}X^{(2)}_{a_2b_2, a_1b_3}X^{(3)}_{a_3b_3, a_3b_2}.$$
We refer to \cref{fig:Tensormoment} for the diagram representation of the above value.

\begin{figure}
 \centering
 \includegraphics[width=0.35\linewidth] {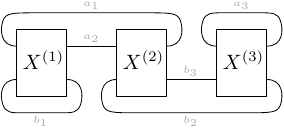}
 \caption{The diagram for the trace invariant $\Tr_{(12)(3), (1)(23)}(X_1,X_2,X_3)$. The summation is over the six indices $a_{1,2,3},b_{1,2,3}$. The first permutation $\alpha_1=(12)(3)$ connects the top legs of the tensors, while the second permutation $\alpha_2=(1)(23)$ connects the bottom legs. On each level $s=1,2$, the input (leg on the RHS) of the $i$-th box is connected to the output (leg on the LHS) of the $\alpha_s(i)$-th box.}
 \label{fig:Tensormoment}
 \end{figure}

In the case where the invariant is described by a single permutation used on all levels $\alpha_s = \sigma \in S_p$ for all $s \in [r]$, the tensor trace invariant reduces to the usual trace invariant (\cref{eq-SingleTraceInv}) of matrices $X^{(i)} \in \M{d_1\cdots d_r}$, ignoring the tensor product structure:  $\Tr_{\underline{\alpha}}=\Tr_{\sigma}$. In particular, remembering that $\gamma_p = (1,2,\ldots, p)$ denotes the full cycle permutation, we have 
$$\Tr_{\underline{\gamma_p}}\left(X^{(1)},\ldots, X^{(p)}\right)=\Tr (X^{(1)}\cdots X^{(p)}).$$

Considering different permutations on the tensor factors captures genuine tensor behavior and, in particular, gives access to important notions such as the \emph{partial trace} $\Tr_2(X) := [\id \otimes \Tr](X)$:
$$\Tr_{\gamma_p,\id_p}(X)=\Tr\big[(\Tr_2 X)^p\big]$$
or the \emph{partial transposition} $X^\Gamma := [\id \otimes \transp](X)$:
    $$\Tr_{\alpha,\beta}(X^{\Gamma})=\Tr_{\alpha,\beta^{-1}}(X).$$

As we have already mentioned, the functionals $\Tr_{\underline{\alpha}}$ are multilinear in $X^{(i)}$'s and invariant under \emph{conjugation by local-unitaries}: for all $U_s \in \mathcal {U}_{d_s}$,
    \begin{equation}
        \Tr_{\underline{\alpha}}(UX^{(1)}U^*,\ldots, UX^{(p)}U^*)=\Tr_{\underline{\alpha}}(X^{(1)},\ldots, X^{(p)}),\qquad \text{where }U:=\bigotimes_{s=1}^r U_s.
    \end{equation} 
Conversely, every multilinear functional $\mathcal{L}:(\bigotimes_{s=1}^r \M{d_s})^p\to \Comp$ having this property can be written as a linear combination of $\Tr_{\underline{\alpha}}$'s. Indeed, for Haar random local-unitary $U=\bigotimes_{s=1}^r U_s$, we can apply the Weingarten calculus \cref{thm:Weingarten} to obtain
\begin{align*}
    &\mathcal{L}(X^{(1)},\ldots, X^{(p)})=\E_{U}\big[\mathcal{L}(UX^{(1)}U^*,\ldots, UX^{(p)}U^*) \big]\\
    &= \sum_{\underline{\sigma},\underline{\tau}\in (S_p)^r} \sum_{\text{all indices}}  \left(\prod_{k=1}^p X_{j_1^{(k)}\ldots j_r^{(k)}, {j'_1}^{(k)}\ldots {j'_r}^{(k)}}^{(k)}\right) \, \mathcal{L}\big(E_{\underline{i}^{(1)}\underline{i'}^{(1)}},\ldots, E_{\underline{i}^{(p)}\underline{i'}^{(p)}}\big)\\
    &\qquad\qquad\qquad\qquad \times \prod_{s=1}^r \left(\prod_{k=1}^p \delta_{i_s^{(\sigma_s(k))}{i'}_s^{(k)}} \delta_{j_s^{(\tau_s(k))}{j'}_s^{(k)}}\right)\Wg_{d_s}^{(U)}(\tau_s^{-1}\sigma_s)\\
    &=\sum_{\underline{\tau}}\Tr_{\underline{\tau}}(X^{(1)},\ldots, X^{(p)})\sum_{\underline{\sigma}}   \sum_{\text{all indices}} \mathcal{L}\big(E_{\underline{i}^{(1)}\underline{i'}^{(1)}},\ldots, E_{\underline{i}^{(p)}\underline{i'}^{(p)}}\big)\prod_{s=1}^r \left(\prod_{k=1}^p \delta_{i_s^{(\sigma_s(k))}{i_s'}^{(k)}}\right)\Wg_{d_s}^{(U)}(\tau_s^{-1}\sigma_s),
\end{align*}
where $E_{\underline{i}\,\underline{j}}:=\bigotimes_{s=1}^r E_{i_s j_s}$ are matrix units in $\bigotimes_{s=1}^r \M{d_s}$.

Finally, let us introduce the \emph{normalized} versions of the tensor trace invariants:     
\begin{equation}\label{eq:def-normalized-trace-invariant}
    \tr_{\underline{\alpha}}:=\frac{1}{d_1^{\#\alpha_1}\cdots d_r^{\#\alpha_r}}\Tr_{\underline{\alpha}}.
\end{equation}
Note that this choice of the normalization is motivated by the applications in random matrix theory, and especially by our main results, \cref{thm-locui-tensorfree,thm-indepTranspose}. Other normalizations are also relevant, see \cite{CGL23b,collins2024free}.

\medskip

Let us now introduce the definition of algebraic tensor probability spaces.

\begin{definition} \label{def:tensor-ncps}
An \textit{$r$-partite (algebraic) tensor probability space} is a triple $(\A,\varphi, (\varphi_{\underline{\alpha}}))$ where $(\A,\varphi)$ is a non-commutative probability space with a tracial unital linear functional $\varphi$, and for each $p\geq 1$ and $\underline{\alpha}\in (S_p)^r$,  $\varphi_{\underline{\alpha}}:\A^p\to \Comp$ is a multilinear functional satisfying the following conditions:
\begin{enumerate}
    \item \emph{Consistency}: $\varphi_{\underline{\alpha}}=\varphi_{\alpha}$ if $\alpha_s\equiv \alpha$ for some $\alpha\in S_p$.
    \label{def:tensor-ncps-consistency}
    
    \item \emph{Permutation invariance}: $\varphi_{\sigma\underline{\alpha} \sigma^{-1}}(x_1,\ldots, x_p)=\varphi_{\underline{\alpha}}(x_{\sigma(1)},\ldots, x_{\sigma(p)})$.
    \label{def:tensor-ncps-permutation-invariance}
    
    \item \emph{Multiplicativity}: for $\underline{\alpha}\in (S_p)^r$, $\underline{\beta}\in (S_q)^r$, and $x_1,\ldots, x_p, y_1,\ldots, y_q\in \A$, we have
        $$\varphi_{\underline{\alpha}\sqcup \underline{\beta}}(x_1,\ldots, x_p,y_1,\ldots y_q)=\varphi_{\underline{\alpha}}(x_1,\ldots, x_p)\varphi_{\underline{\beta}}(y_1,\ldots y_q).$$
    \label{def:tensor-ncps-multiplicativity}
    
    \item \emph{Unitality}: for $p\geq 2$, $\underline{\alpha}\in (S_p)^r$, and $x_1,\ldots, x_p\in \A$ with $x_j=1$ for some $j\in [p]$, we have
        $$\varphi_{\underline{\alpha}}(x_1,\ldots, x_p)=\varphi_{\underline{\alpha}'}(x_1,\ldots, x_{j-1},x_{j+1},\ldots, x_p),$$
    where $\underline{\alpha}'\in (S_{p-1})^r$ is obtained by first ``erasing'' $j$ in every permutation $\alpha_s$ and then reindexing $j+1,\ldots, p$ as $j,\ldots, p-1$ to define $\alpha'_s$. 
    \label{def:tensor-ncps-unitality}

    \item \emph{Substitution property}: if there exist $i \neq j \in [p]$ such that $\alpha_s(i) = j$ for all $s \in [r]$, then 
    $$\varphi_{\underline \alpha}(\underline x) = \varphi_{\underline{\alpha'}}(\underline{x'}),$$
    where 
    $$\underline{x'} = (x_1, \ldots, \underbrace{x_jx_i}_{\text{pos.~$i$}}, \ldots, \underbrace{\emptyset}_{\text{pos.~$j$}}, \ldots, x_p)$$
    and, for all $s \in [p]$, $\alpha'_s\in S([p]\setminus \{j\})$ is defined by
    $$\alpha'_s(k) = \begin{cases} 
    	\alpha_s(k) &\qquad \text{ if } k \neq i\\
	\alpha_s(j) &\qquad \text{ if } k=i.
	\end{cases}$$
    \label{def:tensor-ncps-substitution}
\end{enumerate}
\end{definition}

Note that when using a single non-commutative random variable, we denote by $\varphi_{\underline{\alpha}}(x):=\varphi_{\underline{\alpha}}(x,\ldots, x)$.

{
\begin{remark}
Since the state $\phi$ from the definition above is completely determined by the collection $(\varphi_{\underline{\alpha}})$ using the consistency condition from \cref{def:tensor-ncps-consistency}, one could remove it from the definition. We decided to keep the state $\phi$ in order to introduce tensor probability spaces as generalizations of usual non-commutative probability spaces (which are pairs $(\A, \phi)$). This choice allows us to discuss simultaneously and compare the usual notion of freeness and the newly introduced notion of tensor freeness (e.g.~\cref{prop:tensor-free-tensor-product}), as well as convergence results (e.g.~\cref{thm-ui-tensorfree}).
\end{remark}
}
\begin{remark}
    The conditions from \cref{def:tensor-ncps-consistency}, \cref{def:tensor-ncps-permutation-invariance}, \cref{def:tensor-ncps-substitution} imply that the state $\phi$ must be \emph{tracial}. Indeed, with $\underline{\alpha} = \underline{(12)}$ (i.e. $\alpha_s\equiv (12)$) and $\sigma = (12)$, we have
    $$\phi(x_1 x_2) = \phi_{\underline{(12)}}(x_1, x_2) = \phi_{(12) \cdot \underline{(12)} \cdot (12)^{-1}}(x_2, x_1) = \phi_{\underline{(12)}}(x_2, x_1) = \phi(x_2 x_1).$$
\end{remark}
\begin{remark}\label{rk:unitality}
    Let us comment on the unitality condition from \cref{def:tensor-ncps-unitality}, which might seem peculiar at first sight. Firstly, this condition is analogous to the assumption (Id) in \cite[p.9]{BB24}. For example, it entails that 
    $$\phi_{\substack{(1\,2\,5\,3)(4\,6)\\(1)(3\,4\,5) (2\,6)}}(x_1, 1, x_3, x_4, x_5, x_6) = \phi_{\substack{(1\,4\,2)(3\,5)\\(1)(2\,3\,4)(5)}}(x_1,x_3, x_4, x_5, x_6).$$
    Above, the erasure of $2$ from the permutation $\alpha_1=(1\,2\,5\,3)(4\,6)\in S_6$ gives the permutation $(1\,4\,2)(3\,5)\in S_5$. 

    Note that, in general, the unitality condition implies that all the tensor moments of $1 \in \A$ are equal to 1: 
    $$\forall \underline{\alpha} \in S_p^r, \qquad \phi_{\underline{\alpha}}(1) = 1.$$
    The unitality property will be crucial in \cref{prop-1-TensorFree}.
\end{remark}

Let us point out that the technical conditions in the definition above are very close to the ones used in related work, such as \cite{male2020traffic} or \cite{BB24}.

We consider now an example in order to clarify the substitution property \cref{def:tensor-ncps}-\cref{def:tensor-ncps-substitution}. We have:
	$$\varphi_{(1423),(1342)}(x_1,x_2,x_3,x_4) = \varphi_{(132),(123)}(x_1,x_3,x_2\cdot x_4),$$	
	with the choice $i=4$, $j=2$, see \cref{fig:substitution-property}.
    \begin{figure}[htbp]
        \begin{center}
		\includegraphics{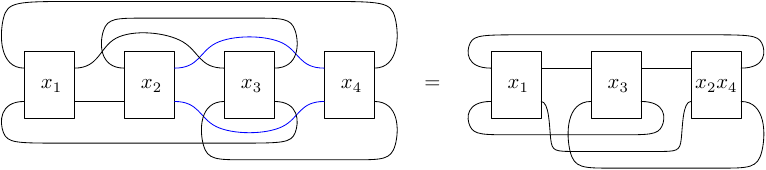}
        \caption{An example of the application of the substitution property from \cref{def:tensor-ncps}. We simplify the diagram on the left-hand-side by noticing that all the wires (blue) going out of $x_4$ connect to $x_2$. Hence, we can take the product in the algebra $x_2 \cdot x_4$ and reduce the number of elements from $p=4$ to $p=3$.}
        \label{fig:substitution-property}
        \end{center}
    \end{figure}

{
\begin{proposition} \label{prop-TensorMomentIndep}
    The family of expectation functionals 
    $$\varphi_{\underline \alpha}, \qquad \underline \alpha \in \bigsqcup_{p=1}^\infty (S_p)^r$$
    satisfying the hypotheses from \cref{def:tensor-ncps} only depends on the independent subfamily 
    $$\varphi_{\underline \alpha}, \qquad \underline \alpha \in \Sindep,$$
    where $\Sindep$ contains one representative from each conjugacy class of permutation $r$-tuples that are simultaneously 
    \begin{itemize}
        \item \emph{irreducible}: $\bigvee_{s=1}^r \Pi(\alpha_s) = 1_p$;
        \item \emph{branching}: there do not exist $i \neq j \in [p]$ such that $\alpha_s(i)=j$ for all $s \in [r]$.
    \end{itemize}
\end{proposition}
 }
\begin{proof}
This follows easily from the various properties of the functionals $\varphi_{\underline{\alpha}}$ from \cref{def:tensor-ncps}. For example, the multiplicativity and the permutation invariance conditions above imply a much stronger multiplicativity property for the functionals $\varphi$. Indeed, if there exists a partition $\pi\in \mathcal{P}(p)$ such that $\pi \geq \Pi(\alpha_s)$ for all $s \in [r]$, then 
$$\varphi_{\underline \alpha}(x_1,\ldots, x_p) = \prod_{b \in \pi} \varphi_{\underline \alpha \restrict{b}}\big((x_j)\restrict b\big).$$
For example, we have 
$$\varphi_{(13)(2),(1)(2)(3)}(x,y,z) = \varphi_{(12),(1)(2)}(x,z) \varphi(y).$$

Similarly, one can use the substitution property to reduce the size of the permutations by taking products in the algebra $\mathcal A$.
\end{proof}

Note that, in the case $r=1$, $\Sindep$ contains only one element:  $\alpha_1 = (1)$, for $p=1$. This $\phi_{(1)}$ functional is precisely the trace $\varphi$ from the usual (i.e. non-tensor) non-commutative probability theory. On the other hand, in the case $r=2$, the set $\Sindep$ is infinite, but we can enumerate its first $p$-levels. At $p=1$, we have $\phi_{(1)^2} =\phi$. At $p=2$, we only have two new functionals $\phi_{\underline{\alpha}}(x,y)$ (instead of $(2!)^2=4)$:
\smallskip
\begin{center}
\begin{tabular}{|
>{\columncolor[HTML]{EFEFEF}}r|c|c|}
\hline
 & \cellcolor[HTML]{EFEFEF}$(1)(2)$ & \cellcolor[HTML]{EFEFEF}$(12)$ \\ \hline
$(1)(2)$ & $\phi(x)\phi(y)$ & $\phi_{(1)(2),(12)}(x,y)$ \\ \hline
$(12)$ & $\phi_{(12),(1)(2)}(x,y)$ & $\phi(xy)$ \\ \hline
\end{tabular}
\end{center}
\smallskip

Let us now discuss examples. 

\begin{example}
    For positive integers $d_1, d_2, \ldots, d_r$, 
    $$\left( \bigotimes_{s=1}^r \M{d_s}, \tr, (\tr_{\underline{\alpha}})_{\underline{\alpha} \in \bigsqcup_{p=1}^\infty S_p^r} \right)$$
    is an $r$-partite tensor probability space. Recall that the \emph{normalized} tensor trace invariants were defined in \cref{eq:def-trace-invariant} and \eqref{eq:def-normalized-trace-invariant}. This is the canonical example, and we shall make use of the graphical notation (like in \cref{fig:Tensormoment}) for tensor trace invariants as often as possible. 
\end{example}

One can construct tensor probability spaces out of (usual, non-tensor) non-commutative probability spaces using the (algebraic) tensor product. 

\begin{example}\label{ex:tensor-product-of-ncps}
More generally, given $r$ non-commutative probability spaces $(\mathcal A^{(s)}, \phi^{(s)})$, $s \in [r]$, define their \emph{(algebraic) tensor product}
    $$ \left( \mathcal A := \bigotimes_{s=1}^r \mathcal A^{(s)}, \phi := \bigotimes_{s=1}^r \phi^{(s)}, (\phi_{\underline{\alpha}})_{\underline{\alpha} \in \bigsqcup_{p=1}^\infty S_p^r} \right),$$
    where $\phi_{\underline{\alpha}}$ is defined on simple tensors by
    $$\phi_{\underline{\alpha}}\left( \bigotimes_{s=1}^r x^{(s)}_1, \ldots, \bigotimes_{s=1}^r x^{(s)}_p \right) := \prod_{s=1}^r \phi^{(s)}_{\alpha_s} \left( x^{(s)}_1, \ldots, x^{(s)}_p \right)$$
    and extended by linearity to the whole tensor space. We remark that all the tensor products here are \textit{algebraic}, and therefore, elements in the algebra $\A$ are linear combinations of finite number of elements of the form $\bigotimes_{s=1}^r x^{(s)}$. 
\end{example}

One can use the example above to construct commutative tensor probability spaces consisting of $r$-tuples of random variables out of commutative probability spaces of the form $\A^{(s)}=(L^{\infty-}(\mathbb P), \mathbb E)$.

\medskip

Let us introduce now the important notion of convergence in tensor distribution, similar to the usual convergence in distribution for non-commutative random variables \cite[Definition 8.1]{nica2006lectures} or to the convergence in traffic-distribution \cite[Definition 1.2 and Theorem 1.8]{male2020traffic}.

\begin{definition}
    Consider a sequence of families $\W_N= (x_{N,j})_{j \in J} \subseteq \mathcal A^{(N)}$ from some non-commutative probability spaces $(\A^{(N)},\varphi^{(N)})$ that are endowed with multilinear functionals $(\phi^{(N)}_{\underline{\alpha}})_{\underline{\alpha} \in \bigsqcup_{p=1}^\infty S_p^r}$. Similarly, consider a family $\W=(x_j)_{j\in J} \subseteq \mathcal A$, where again $(\A,\varphi)$ is a non-commutative probability space endowed with multilinear functionals $\phi_{\underline{\alpha}}$; here we do not require that $\phi^{(N)}_{\underline{\alpha}}$ or $\phi_{\underline{\alpha}}$ satisfy all the assumptions \cref{def:tensor-ncps-consistency}-\cref{def:tensor-ncps-substitution} from \cref{def:tensor-ncps}, but we assume the relations
        $$\varphi^{(N)}=\varphi^{(N)}_{\underline{\id_1}}, \quad \varphi=\varphi_{\underline{\id_1}}.$$
    We say that the sequence $\W_N$ \emph{converges in tensor distribution} to $\W$ if, for all $p \geq 1$, for all $\underline{\alpha} \in S_p^r$, and for all $j:[p] \to J$, we have 
    $$\lim_{N \to \infty} \phi^{(N)}_{\underline{\alpha}}\big( x_{N,j(1)}, x_{N,j(2)}, \ldots, x_{N,j(p)} \big) = \phi_{\underline{\alpha}}( x_{j(1)}, x_{j(2)}, \ldots, x_{j(p)}).$$
We simply write $\W_N\xrightarrow{\otimes\text{-distr}} \W$.
\end{definition}

\medskip

Next, we would like to consider the case of \textit{random matrices}. In the following, we take $r$ sequences $d_s=d_{s,N}$ ($s=1,\ldots, r$) of positive integers which increase to $\infty$ as $N\to \infty$. Let $D=D_N:=d_{1,N}\cdots d_{r,N}$, and let us identify $\M{D}\cong \M{d_1}\otimes\cdots \otimes \M{d_r}$.
Note that one cannot simply extend the usual setting from \cite[Examples 1.4.(3)]{nica2006lectures} by taking tensor products because, even in the $r=1$ case, the multilinear functionals
    $$\mathbb E \circ \tr_{\underline{\alpha}}:\Big(\big({\bigotimes}_{s=1}^r \mathcal{M}_{d_s}\big)(L^{\infty-}\big(\mathbb{P})\big)\Big)^p\to \Comp, \quad p\geq 1, \;\; \underline{\alpha}\in S_p^r,$$
do not satisfy in general the multiplicativity condition from \cref{def:tensor-ncps}-\cref{def:tensor-ncps-multiplicativity}. To address this issue in the case $r=1$, we introduced the \textit{factorization property} in \cref{def-FactProp-BddMoment}. We now extend the conditions from \cref{def-FactProp-BddMoment} to their tensor analogs as follows.

\begin{definition} \label{def-Tensor-FactProp-BddMoment}
Let $\W_N$ be a family of $D_N\times D_N$ random matrices for each $N\geq 1$.
\begin{enumerate}
    \item $\W_N$ is said to satisfy the \emph{tensor factorization property} if, for all $X_1,\ldots, X_p,Y_1,\ldots, Y_q\in \W_N$, $\underline{\alpha}\in S_p^r$, and $\underline{\beta}\in S_q^r$, we have as $N\to \infty$,
    \begin{equation} \label{eq-condition-TensorFact}
        \E[\tr_{\underline{\alpha}\sqcup \underline{\beta}}(X_1,\ldots, X_p,Y_1,\ldots, Y_q)]= \E[\tr_{\underline{\alpha}}(X_1,\ldots, X_p)]\E[\tr_{\underline{\beta}}(Y_1,\ldots, Y_q)]+o(1)
    \end{equation}

    \item $\W_N$ is said to satisfy the \emph{bounded tensor moments property} if
    \begin{equation} \label{eq-condition-TensorBdd}
        \sup_N \big|\E[\tr_{\underline{\sigma}}(X_1,\ldots, X_p)]\big|<\infty, \quad X_1,\ldots, X_p\in \mathcal{W}_N, \quad p\geq 1, \;\;\underline{\sigma}\in S_p^r .
    \end{equation}
\end{enumerate}
\end{definition}

When 
$r=1$, the tensor factorization property and the bounded tensor moments property coincide with the factorization property and the bounded moments property, respectively. In general, properties involving tensors impose strictly stronger assumptions than those without tensor product structures. For example, the convergence in distribution does not necessarily imply the bounded moments property \cref{eq-condition-Bdd}; however, the (stronger) convergence in \textit{tensor} distribution always implies the bounded \textit{tensor} moments property \cref{eq-condition-TensorBdd}.

Note that if $\W_N$, endowed with the functionals $\mathbb E \circ \tr_{\underline{\alpha}}$, converges in tensor distribution and if the corresponding limit $\W$ is taken from an $r$-partite tensor probability space $(\A,\varphi,(\varphi_{\underline{\alpha}}))$ as in \cref{def:tensor-ncps}, then $\W_N$ necessarily satisfies tensor factorization property \cref{eq-condition-TensorFact} by the multiplicativity of $\varphi_{\underline{\alpha}}$. The converse statement also can be shown: given tensor factorization property, random matrices always converge in tensor distribution to limit elements in some tensor probability space. This result is the direct analogue of \cite[Lemma 4.11]{male2020traffic}. 

\begin{lemma} \label{lem-TensorLimitRealization}
Suppose a family of $D_N\times D_N$ random matrices $\mathcal{W}_N=(X_{N,j})_{j\in J}$ satisfies the following conditions:
\begin{enumerate}
    \item the limit $\lim_{N\to \infty}\E[\tr_{\underline{\alpha}}(X_1,\ldots, X_p)]$ exists for all $\underline{\alpha}\in (S_p)^r$ and $X_1,\ldots, X_p\in \mathcal{W}_N$;

    \item the \emph{tensor factorization property} \cref{eq-condition-TensorFact}.
\end{enumerate}
Then, there exists a family $\mathcal{W}=(x_j)_{j\in J}$ in some $r$-partite algebraic tensor probability space $(\A,\varphi,(\varphi_{\underline{\alpha}}))$ such that $\mathcal{W}_N$ converges in tensor distribution to $\mathcal{W}$ as $N \to \infty$.
\end{lemma}
\begin{proof}
Let $\A=\Comp\la (x_j)_{j \in J} \ra$ be the set of non-commutative polynomials in the formal variables $x_j$. For every $p\geq 1$ and $\underline{\alpha}\in (S_p)^r$, we can define the functional $\varphi_{\underline{\alpha}}:\A^p\to \Comp$ by
\begin{equation}
    \varphi_{\underline{\alpha}}(P_1,\ldots, P_p):=\lim_{N\to \infty} \E\left[\tr_{\underline{\alpha}}(P_1(\mathcal{W}_N), \ldots, P_1(\mathcal{W}_N)\right].
\end{equation}
Then we can check that $\varphi_{\underline{\alpha}}$ is well-defined multilinear functional, $(\A,\varphi,(\varphi_{\underline{\alpha}}))$ is an $r$-partite non-commutative probability space where $\varphi(a):=\varphi_{\underline{\id_1}}(a)$, and $\W_N\xrightarrow{\otimes\text{-distr}} \W:=(x_j)_{j\in J}$ as $N\to \infty$.
\end{proof}

As in \cref{prop-FactorConvInProb}, the tensor factorization property has a close relation with the convergence in probability.

\begin{proposition} \label{prop-TensorFactConvProb}
Let $\W_N$ be a family of $D_N\times D_N$ random matrices, and suppose all the limits
    $$\lim_{N\to \infty}\E \big[\tr_{\underline{\alpha}}(X_1,\ldots, X_p)\big]$$
exist for $p\geq 1$, $\underline{\alpha}\in (S_p)^r$, and 
$X_1,\ldots, X_p\in \W_N\cup \W_N^*$. Then $\W_N\cup \W_N^*$ satisfies the tensor factorization property \cref{eq-condition-TensorFact} if and only if the same convergences hold in probability.
\end{proposition}
\begin{proof}
If $\W_N\cup \W_N^*$ satisfies the tensor factorization property, then as $N\to \infty$,
    $${\rm var}\big(\tr_{\underline{\alpha}}(\underline{X})\big)=\E\big[\tr_{\underline{\alpha}\sqcup \underline{\alpha}^{-1}}(\underline{X}, \underline{X}^*)\big]-\E\big[\tr_{\underline{\alpha}}(\underline{X})\big] \E\big[\tr_{\underline{\alpha}^{-1}}(\underline{X}^*)\big]\to 0,$$
where $\underline{X}=(X_1,\ldots, X_p)$ and $\underline{X}^*=(X_1^*,\ldots,X_p^*)$. Therefore, $\displaystyle \tr_{\underline{\alpha}}(\underline{X})\to \lim_{N\to \infty}\E[\tr_{\underline{\alpha}}(\underline{X})]$ in probability as $N\to \infty$.

Conversely, assume the convergence in probability, i.e., $\displaystyle \tr_{\underline{\alpha}}(\underline{X})\to \lim_{N\to \infty}\E[\tr_{\underline{\alpha}}(\underline{X})]$ in probability for every $p\geq 1$ and $\underline{\alpha}\in (S_p)^r$. Then for $\underline{\alpha}\in (S_p)^r$, $\underline{\beta}\in (S_q)^r$, and $X_1,\ldots, X_p,Y_1,\ldots, Y_q\in \W_N\cup \W_N^*$, set $Z_N=\tr_{\underline{\alpha}\sqcup \underline{\beta}}(X_1,\ldots, X_p,Y_1,\ldots, Y_q)$. Then we have
    $$|Z_N|^2=\tr_{\underline{\alpha}\sqcup \underline{\beta}\sqcup \underline{\alpha}^{-1}\sqcup \underline{\beta}^{-1}}(\underline{X},\underline{Y},\underline{X}^*,\underline{Y}^*),$$
and therefore $\sup_N\E[|Z_N|^2]<\infty$. In other words, the sequence $(Z_N)_{N\geq 1}$ is uniformly integrable. Hence, the Vitali convergence theorem implies that
    $$\lim_{N\to\infty}\E \big[\tr_{\underline{\alpha}\sqcup\underline{\beta}}(\underline{X},\underline{Y})\big]= \E\big[\Plim_{N\to \infty} \tr_{\underline{\alpha}}(\underline{X})\tr_{\underline{\beta}}(\underline{Y})\big] = \lim_{N\to\infty}\E \big[\tr_{\underline{\alpha}}(\underline{X})\big] \cdot \lim_{N\to\infty}\E \big[\tr_{\underline{\beta}}(\underline{Y})\big],$$
which establishes the desired tensor factorization property.
\end{proof}

For later use, let us briefly introduce the tensor analogue of orthogonal trace invariants. For a tuple of pairings $\underline{\pi}=(\pi_1,\ldots, \pi_r)\in \mathcal{P}_2(\pm p)^r$, define the \textit{local orthogonal trace invariant} $\Tr_{\underline{\pi}\vee\delta}:(\bigotimes_{s=1}^r\M{d_s})^p\to \Comp$ by
\begin{equation} \label{eq-LocOITraceInv}
    \Tr_{\underline{\pi}\vee\delta}(X^{(1)},\ldots, X^{(p)}):=\sum_{\text{all indices}} \left(\prod_{k=1}^p X^{(k)}_{i_1^{(k)}\cdots \, i_r^{(k)},i_1^{(-k)}\cdots \, i_r^{(-k)}}\right)\prod_{s=1}^r \left(\prod_{(k,k')\in \pi_s}\delta_{i_s^{(k)},i_s^{(k')}}\right)
\end{equation}
and its normalized value $\tr_{\underline{\pi}\vee\delta}:=\frac{1}{d_1^{\#(\pi_1\vee\delta)}\cdots \,d_r^{\#(\pi_r\vee\delta)}}\Tr_{\underline{\pi}\vee\delta}$. We can imagine the diagrams of $X^{(k)}$'s whose outputs and inputs are labeled as $k$ and $-k$, resp., and then we connect their tensor legs according to each $\pi_s$.

Similarly to the case of the orthogonal trace invariant $\Tr_{\pi\vee\delta}$, we have for $\pi_s=\eps_s\sigma_s\delta\sigma_s^{-1}\eps_s$, 
\begin{equation}
    \Tr_{\underline{\pi}\vee\delta} (X_1,\ldots, X_p)=\Tr_{\underline{\sigma}}(X_1^{\underline{f}(1)}\cdots X_p^{\underline{f}(p)}),
\end{equation}
where $X^{\underline{t}}:=(t_1\otimes\cdots \otimes t_r)(X)$ denotes the \textit{partial transpose} for $\underline{t}\in \{\id,\top\}^r$ and $\underline{f}=(f_1,\ldots, f_r):[p]\to \{\id,\top\}^r$ where $f_s(k)=\begin{cases}
    \id & \text{if $\eps_s(k)=k$,}\\
    \top & \text{if $\eps_s(k)=-k$,}
\end{cases}$ for $k=1,\ldots, p$. In particular, $\Tr_{\underline{\sigma\delta\sigma^{-1}}\vee\delta}=\Tr_{\underline{\sigma}}$ for $\underline{\sigma}\in (S_p)^r$. Furthermore, for every multilinear functional $\mathcal{L}:\big(\bigotimes_{s=1}^r \M{d_s}\big)^p\to \Comp$ which is invariant under local-orthogonal matrices
    $$\mathcal{L}(OX^{(1)}O^{\top},\ldots, OX^{(p)}O^{\top})=\mathcal{L}(X^{(1)},\ldots, X^{(p)}),\quad  \text{where } O:=\bigotimes_{s=1}^rO_s,\;\; O_s\in \mathcal{O}_{d_s},$$
$\mathcal{L}$ can be written as a linear combination of $\{\Tr_{\underline{\pi}\vee\delta}\}_{\underline{\pi}\in \mathcal{P}_2(\pm p)^r}$ thanks to the orthogonal Weingarten calculus \cref{thm:Weingarten}.

\section{Tensor free cumulants}\label{sec:tensor-free-cumulants}

The notion of \emph{freeness}, introduced by Voiculescu \cite{voiculescu1985symmetries,voiculescu1992free}, originates in operator algebra and captures the lack of relations between the group algebras of free groups. Later, Speicher developed a combinatorial approach to freeness \cite{speicher1994multiplicative,nica2006lectures} using the notion of \emph{free cumulants} that is based on the lattice of \emph{non-crossing partitions}. We refer the reader to \cref{sec:preliminary} for some background material on these notions. 

In this section, we generalize Speicher's free cumulants to the tensor case, and in the next section we introduce the corresponding notion of \emph{tensor freeness}. Let us emphasize that there have been many recent works investigating the same problem, and developing very similar notions of tensor cumulants \cite{KMW24,BB24,collins2024free}. {A common thread is the definition of moments or invariants based on \emph{combinatorial structures} such as trace invariants encoded by permutations or graph contractions, and the subsequent definition of cumulants through moment-cumulant formulas. 

While sharing these foundational principles, the specific definitions and contexts differ significantly. Kunisky, Moore, and Wein \cite{KMW24} address real symmetric tensors under orthogonal $\mathcal O(d)$ invariance, defining ``tensorial finite free cumulants'' by averaging distinct-index graph moments over the Haar measure. Bonnin and Bordenave \cite{BB24} introduce an abstract algebraic framework using combinatorial maps, defining tensor freeness and associated cumulants via M\"obius inversion on a poset of maps. Collins, Gurau, and Lionni \cite{collins2024free} focus on local-unitary (LU) invariant complex tensors, distinguishing between ``pure'' and ``mixed'' types, and define finite-size cumulants using Weingarten calculus. After taking the large $N$ limit, they obtain the asymptotic version of free cumulants which, at first order, coincide with our notion of tensor free cumulants defined below. The specific invariants (permutation-based traces, graph moments, map traces) and the nature of the cumulants thus vary depending on the framework and objectives of each paper. 

Our approach is closer in nature the original definition of free cumulants by Speicher \cite{speicher1994multiplicative}, see also \cite{nica2006lectures,mingo2017free}. We generalize free cumulants to the tensor case by firstly embedding non-crossing partitions as geodesic permutations and secondly considering $r$-tuples of such permutations to address multipartite matrices. We proceed then in a similar way as the other research discussed above by making use of the M\"obius inversion to relate tensor cumulants to tensor moment functionals and definte tensor free independence by the vanishing of mixed tensor free cumulants.}

\begin{definition}\label{def:free-tensor-cumulants}
    Given an $r$-partite tensor probability space $(\A,\varphi, (\varphi_{\underline{\alpha}}))$, we introduce the \emph{tensor free cumulants} $(\kappa_{\underline{\beta}})$ associated to the tensor moment functionals $(\phi_{\underline{\alpha}})$ by the following equivalent formulas:

    \begin{enumerate}    
        \item The \emph{tensor free moment-cumulant formula}: for all $p\geq 1$ and all $\underline{\alpha}\in (S_p)^r$ 
        \begin{equation}\label{eq:tensor-moment-free-cumulant}
    	   \varphi_{\underline{\alpha}}(x_1,\ldots, x_p):=\sum_{\underline{\beta} \in S_{NC}(\underline{\alpha})}  \kappa_{\underline{\beta}}(x_1,\ldots, x_p).
    \end{equation}
    \item The \emph{tensor free cumulant-moment formula}: for all $p\geq 1$ and all $\underline{\beta}\in (S_p)^r$ 
    \begin{equation}\label{eq:tensor-free-cumulant-moment}
        \kappa_{\underline{\beta}}(x_1,\ldots, x_p):=\sum_{\underline{\alpha} \in S_{NC}(\underline{\beta})}  \varphi_{\underline{\alpha}}(x_1,\ldots, x_p) \Mob(\underline{\alpha}^{-1}\underline{\beta}).
    \end{equation}
    \end{enumerate}
    Above, we recall that by $\underline{\alpha} \in S_{NC}(\underline{\beta})$ we mean that $\alpha_s \in S_{NC}(\beta_s)$ for all $s \in [r]$, and that we write
    $$\Mob(\underline{\alpha}^{-1}\underline{\beta}):= \prod_{s=1}^r \Mob (\alpha_s^{-1}\beta_s),$$
    see \cref{sec:preliminary-permutations,sec:preliminary-FreeProb}. {Furthermore, we simply denote by $\kappa_{\underline{\alpha}}(x):=\kappa_{\underline{\alpha}}(x,x,\ldots, x)$.}
\end{definition}
First, note that the equations \eqref{eq:tensor-moment-free-cumulant} and \eqref{eq:tensor-free-cumulant-moment} extend the free moment-cumulant formulas \eqref{eq-FreeMomentCumulant} \eqref{eq-FreeMomentCumulant2}. The equivalence between \eqref{eq:tensor-moment-free-cumulant} and \eqref{eq:tensor-free-cumulant-moment} is straightforward and follows from the usual (non-tensor) case by multiplying $r$-times the corresponding equations, see \cite[Proposition 11.4]{nica2006lectures}. This equivalence is a particular instance of the M\"obius inversion formulas in a lattice, see \cite[Lecture 10]{nica2006lectures}. In particular, tensor free cumulants are exactly the usual, non-tensor, free cumulants in the trivial case $r=1$.

{
\begin{remark}\label{rk:tensor-free-cumulants-permutation-invariance}
    Thanks to the permutation invariance and multiplicativity of the M\"{o}bius function:
    \begin{align*}
        \Mob(\sigma \alpha \sigma^{-1})=\Mob(\alpha)&, \quad \alpha, \sigma\in S_p,\\
        \Mob(\alpha\sqcup \beta)=\Mob(\alpha)\Mob(\beta)&,\quad \alpha\in S_p,\;\; \beta\in S_q,
    \end{align*}
    tensor free cumulants inherit the permutation invariance property and multiplicativity (\cref{def:tensor-ncps}-\cref{def:tensor-ncps-permutation-invariance,def:tensor-ncps-multiplicativity}) from the tensor moment functionals $(\phi_{\underline{\alpha}})$.

    On the other hand, let us comment that tensor free cumulants do not satisfy the consistency in the case $r\geq 2$: even if $\alpha_s\equiv \sigma$, we have $\kappa_{\underline{\alpha}} \neq \tilde{\kappa}_{\sigma}$ in general while $\varphi_{\underline{\alpha}}=\varphi_{\sigma}$. See the discussion below and \cref{prop-cumulant-from-tensorcumulant} for precise relations between $\kappa_{\underline{\alpha}}$ and $\tilde{\kappa}_{\sigma}$.
\end{remark}

}

Let us explore the definition of tensor free cumulants via some examples and compare it to the definition of free cumulants (\cref{sec:preliminary-FreeProb}). Let $(\A,\varphi, (\varphi_{\underline{\alpha}}))$ be an $r$-partite tensor probability space, and consider $(\tilde \kappa_\beta)$ the usual free cumulants associated to $(\A,\varphi)$. The case $p=1$ is trivial;
$$\kappa_{(1), \ldots, (1)}(x) = \varphi(x) = \tilde \kappa_{(1)}(x).$$
In the case $p=r=2$, we have four tensor free cumulants: 
\begin{align*}
	\kappa_{(1)(2),(1)(2)}(x,y) &= \varphi_{(1)(2),(1)(2)}(x,y) = \varphi(x)\varphi(y) = \tilde \kappa_{(1)(2)}(x,y)\\
	\kappa_{(1)(2),(1,2)}(x,y) &= \varphi_{(1)(2),(1,2)}(x,y) - \varphi_{(1)(2),(1)(2)}(x,y) = \varphi_{(1)(2),(1,2)}(x,y) - \phi(x)\phi(y)\\
	\kappa_{(1,2),(1)(2)}(x,y) &= \varphi_{(1,2),(1)(2)}(x,y) - \varphi_{(1)(2),(1)(2)}(x,y) = \varphi_{(1,2),(1)(2)}(x,y) - \phi(x)\phi(y)\\
	\kappa_{(1,2),(1,2)}(x,y) &= \varphi_{(1,2),(1,2)}(x,y) - \varphi_{(1,2),(1)(2)}(x,y) - \varphi_{(1)(2),(1,2)}(x,y) + \varphi_{(1)(2),(1)(2)}(x,y)\\
    &= \phi(xy) - \varphi_{(1,2),(1)(2)}(x,y) - \varphi_{(1)(2),(1,2)}(x,y)+\phi(x)\phi(y).	
\end{align*}
Compare the last formula with the one for (non-tensor) free cumulant
$$\tilde \kappa_{(1,2)}(x,y) = \phi(xy) - \phi(x)\phi(y).$$
{The above simple computations show how the tensor case is different and contains more information than the usual (non-tensor) setting.}

Free cumulants can be expressed in terms of tensor free cumulants as follows. We recall that, for a given permutation $\alpha \in S_p$, the set of geodesic permutations $S_{NC}(\alpha)$ is a \emph{lattice}, isomorphic to 
$$\bigtimes_{c \in \cyc(\alpha)} NC({\rm Card}(c)),$$
hence the join operation $\vee$ is well-defined on $S_{NC}(\alpha)$; see \cref{prop:lattice-structure,prop:lattice-structure2}.

\begin{proposition}\label{prop-cumulant-from-tensorcumulant} 
Let $(\mathcal A, \phi, (\phi_{\underline{\alpha}}))$ be a tensor probability space and consider the tensor free cumulants $(\kappa_{\underline{\alpha}})$ as well as the usual, non-tensor, free cumulants $\tilde \kappa_\alpha$ associated to the functional $\phi$. Then, for all $p \geq 1$, all $\alpha \in S_p$, and all $x_1, \ldots, x_p \in \A$, we have
\begin{equation} \label{eq-cum-tensorcum}
    \tilde \kappa_{\alpha}(x_1,\ldots, x_p)=\sum_{\substack{\beta_s\in S_{NC}(\alpha)\\ \beta_1\vee\cdots\vee\beta_r=\alpha}}\kappa_{\underline{\beta}}(x_1,\ldots, x_p).
\end{equation}
\end{proposition}
\begin{proof}
Let us simply denote by $\underline{x}=(x_1,\ldots, x_p)$. Since $\varphi_{\sigma}=\varphi_{\sigma,\ldots, \sigma}$, we can combine the free moment-cumulant relation \cref{eq-FreeMomentCumulant2} with its tensor analogue \cref{eq:tensor-moment-free-cumulant} to have
\begin{align*}
	 \tilde \kappa_{\alpha}(\underline x) &= \sum_{\sigma \in S_{NC}(\alpha)} \varphi_\sigma(\underline x) \Mob(\sigma^{-1} \alpha) = \sum_{\sigma \in S_{NC}(\alpha)} \Mob(\sigma^{-1}\alpha) \sum_{\underline \beta \in S_{NC}(\sigma)^r \subseteq S_{NC}(\alpha)^r} \kappa_{\underline \beta}(\underline x)\\
	 &= \sum_{\underline \beta \in S_{NC}(\alpha)^r} \kappa_{\underline \beta}(\underline x)  \sum_{\substack{\sigma \in S_{NC}(\alpha)\\\beta_s \leq \sigma \leq \alpha \, \, \forall s \in [r]}} \Mob(\sigma^{-1} \alpha)\\
	 &= \sum_{\underline \beta \in S_{NC}(\alpha)^r} \kappa_{\underline \beta}(\underline x)  \underbrace{\sum_{\beta_1\vee \cdots \vee \beta_r \leq \sigma \leq \alpha} \Mob(\sigma^{-1}\alpha)}_{=\mathbf 1_{\beta_1\vee \cdots \vee \beta_r = \alpha}} \\
	 &= \sum_{\substack{\underline \beta \in S_{NC}(\alpha)^r \\ \beta_1\vee \cdots \vee \beta_r = \alpha}} \kappa_{\underline \beta}(\underline x)
\end{align*}
\end{proof}
As an example, let us consider the case $p=2$ and $r=2$. We focus only on the connected free cumulant, that is we take $\alpha = \gamma_2 = (12)$. 
\begin{align*}
	\tilde \kappa_{(12)}(x,y) &= \varphi(xy) - \varphi(x)\varphi(y) = \varphi_{\substack{(12)\\(12)}}(x,y) - \varphi_{\substack{(1)(2)\\(1)(2)}}(x,y) \\
	&= \kappa_{\substack{(12)\\(12)}}(x,y) + \kappa_{\substack{(12)\\(1)(2)}}(x,y) + \kappa_{\substack{(1)(2)\\(12)}}(x,y) + \kappa_{\substack{(1)(2)\\(1)(2)}}(x,y) - \kappa_{\substack{(1)(2)\\(1)(2)}}(x,y) \\
	&= \kappa_{\substack{(12)\\(12)}}(x,y) + \kappa_{\substack{(12)\\(1)(2)}}(x,y) + \kappa_{\substack{(1)(2)\\(12)}}(x,y),
\end{align*}
which is consistent with the result above. 

In the next lemma, we show that in the case where tensor probability space is obtained via a tensor product construction as in \cref{ex:tensor-product-of-ncps}, the tensor free cumulants also factorize. 

{
\begin{lemma} \label{lem-TensorCumulantProd}
Consider a $p$-tuple $\underline{x}=(x_1, \ldots, x_p)$ of non-commutative random variables in an $r$-partite tensor probability space $(\A,\varphi, (\varphi_{\underline{\alpha}}))$. Suppose there exist $r$  usual (non-tensor) non-commutative probability spaces $(\mathcal A^{(s)}, \varphi^{(s)})$ and elements $y_i^{(s)}$, $i \in [p]$, $s \in [r]$, such that $\varphi_{\underline{\alpha}}(\underline{x}) = \prod_s \varphi^{(s)}_{\alpha_s}(\underline{y}^{(s)})$ for all $\underline{\alpha}\in S_p^r$. Then we also have
    $$\kappa_{\underline{\alpha}}(\underline{x}) = \prod_s \tilde{\kappa}^{(s)}_{\alpha_s}(\underline{y}^{(s)}), \quad \underline{\alpha}\in S_p^r$$
    where $\tilde{\kappa}^{(s)}$ are the (usual) free cumulant functionals associated to $\varphi^{(s)}$.

    In particular, if $\mathcal A$ is obtained via a tensor product construction $\mathcal A=\bigotimes_s \A^{(s)}$ as in \cref{ex:tensor-product-of-ncps}, with $x_i = \bigotimes_s x_i^{(s)}$,
    then also 
    \begin{equation} \label{eq-TensorFreeCumulantsFact}
        \kappa_{\underline{\alpha}}(\underline{x}) = \prod_s \tilde \kappa^{(s)}_{\alpha_s}(\underline{x}^{(s)}), \quad \underline{\alpha}\in S_p^r.
    \end{equation}
\end{lemma}
}
\begin{proof}
    The proof is an easy application of the formulas \eqref{eq:tensor-free-cumulant-moment} and \eqref{eq-FreeMomentCumulant2}  relating moments and free cumulants in the tensor and non-tensor cases, hence it is left to the reader.
\end{proof}

\begin{example} [Tensor Haar unitary] \label{ex-TensorHaarUnitaries}
\textit{Haar unitaries} are unitary elements in usual ($r=1$) non-commutative probability space such that their moments are given by 
    $$\forall p \in \mathbb Z \qquad \varphi(u^p) = \mathds 1_{p=0}.$$
    Equivalently, for all functions $\varepsilon:[p] \to \{\pm 1\}$, 
    $$\varphi_\alpha(u^{\varepsilon(1)}, \ldots, u^{\varepsilon(p)}) = \prod_{c \in \alpha} \mathds 1_{\sum_{i \in c} \varepsilon(i) = 0},$$
    meaning that the function $\varepsilon$ is balanced on the cycles of $\alpha$. In particular, permutations having at least one odd cycle yield zero moments. It was shown in \cite[Lecture 15]{nica2006lectures} that the only non-vanishing free cumulants of $u$ are 
    $$\kappa_{2p}(u, u^*, \ldots, u, u^*) = \kappa_{2p}(u^*, u, \ldots, u^*,u) = (-1)^{p-1} \Cat_{p-1}.$$
    More generally,
    $$\kappa_\alpha(u^{\varepsilon(1)}, \ldots, u^{\varepsilon(p)}) = \prod_{c \in \alpha} \mathds 1_{\varepsilon \text{ is alternating on $c$}},$$
    meaning that the function $\varepsilon$ restricted on $c$ is of the form $(+,-, \cdots, +, -)$ or a cyclic permutation thereof.
The distribution of a \textit{tensor Haar unitary element} $u_\otimes := u^{\otimes r}$ can be either recovered by taking the tensor product of the (usual) free cumulants above, or directly: 
    $$\varphi_{\underline{\alpha}}(u_\otimes^{\varepsilon(1)}, \ldots, u_\otimes^{\varepsilon(p)}) = \prod_{s=1}^r \prod_{c \in \alpha_s} \mathds 1_{\sum_{i \in c} \varepsilon(i) = 0}.$$

\end{example}

Another very important example is that of \emph{tensor semicircular elements} which appear as limits of the tensor free central limit theorem, see \cref{sec:tensor-free-CLT}.

\section{Tensor free independence}\label{sec:tensor-free-independence}

Having introduced tensor free cumulants, we can now define the crucial notion of \emph{tensor freeness}. Recall from \cref{prop-TensorMomentIndep} that a tuple of permutations $\underline{\alpha}\in S_p^r$ is said to be \textit{irreducible} if $\Pi(\alpha_1)\vee_{\mathcal{P}}\cdots \vee_{\mathcal{P}} \Pi(\alpha_r)=1_p$. This definition is equivalent to the condition that, for any two distinct elements $k\neq l\in [p]$, there exists $s\in [r]$ such that $k$ and $l$ belong to the same cycle of $\alpha_s$ (such a tuple $\underline{\alpha}$ is referred to as \emph{connected} in \cite{collins2024free}). Moreover, for a function $f:[p]\to I$, $\ker f\in \mathcal{P}(p)$ is the partition consisting of inverse images of $f$ (see \cref{eq-KernelPartition}).

\begin{definition}\label{def:tensor-freeness}
    Let $(\A,\varphi, (\varphi_{\underline{\alpha}}))$ be an $r$-partite tensor probability space. For a set $I$, unital sub-algebras $(\A_i)_{i\in I}$ of $\A$ are called \emph{tensor freely independent} (or \emph{tensor free})
    if every {mixed} tensor free cumulant vanishes, i.e. $\kappa_{\underline{\alpha}}(x_1,\ldots, x_p)=0$ whenever $\underline{\alpha}\in S_p^r$ is irreducible, $x_j\in \A_{f(j)}$, and $f(k)\neq f(l)$ for some $k,l \in [p]$. {Equivalently, for arbitrary (not necessarily irreducible) $\underline{\alpha} \in S_p^r$:
    $${\bigvee}_{\mathcal{P}} \Pi(\alpha_s) \nleq \ker f \implies \kappa_{\underline{\alpha}} (\underline{x}) = 0.$$}
    Furthermore, subsets $(\W_i)_{i\in I}$ of $\A$ are called \emph{tensor free} if the unital algebras $\A_i$ generated by $\W_i$ are tensor free. In particular, non-commutative random variables $(x_i)_{i\in I}\in \A$ are tensor free if the unital algebras they generate are tensor free.
\end{definition}

Note that the notion of free independence for tensors introduced in the definition above allows to compute \emph{mixed} tensor moments of a family of tensor freely independent random variables in terms of their individual, \emph{marginal}, tensor distribution (see \cref{prop-tenfreemoment}). In this sense, it serves the same purpose as the usual notion of freeness. Let us also point out that the notion of tensor free independence reduces to the usual, non-tensor, notion of free independence in the trivial case $r=1$.

\medskip

We first examine how tensor freeness arises from usual freeness in the case of the tensor product construction from \cref{ex:tensor-product-of-ncps}.

{
\begin{proposition}\label{prop:tensor-free-tensor-product}
    Consider $r$ non-commutative probability spaces $(\A^{(s)}, \phi^{(s)})$, $s \in [r]$, and perform the tensor product construction from \cref{ex:tensor-product-of-ncps} to obtain an $r$-partite tensor probability space 
    $$(\A,\varphi, (\varphi_{\underline{\alpha}})):=\bigotimes_{s=1}^r \left( \A^{(s)}, \phi^{(s)} \right).$$
    Let also, for every $s \in [r]$, $\A^{(s)}_1, \ldots, \A^{(s)}_L \subseteq \A^{(s)}$ be freely independent (in the usual, non-tensor sense) unital subalgebras, and consider their tensor products
    $$\forall i \in [L], \qquad \A_i := \bigotimes_{s=1}^r \A^{(s)}_i.$$
    Then, the unital subalgebras $\A_1, \ldots, \A_L \subseteq \A$ are \emph{tensor freely independent}. 
\end{proposition}
\begin{proof}
    Consider $p$ elements $x_1, \ldots, x_p \in \A$ such that $x_j \in \A_{f(j)}$ for a non-constant function $f:[p] \to [L]$, and let $\underline{\alpha} \in S_p^r$ be an {irreducible} $r$-tuple of permutations. Since $f$ is non-constant, $\bigvee_s \Pi(\alpha_s) = 1_p \nleq \ker f$, hence there must exist some $s_0 \in [r]$ such that $\Pi(\alpha_{s_0}) \nleq \ker f$. In particular, there exist $k,l \in [p]$ such that $f(k) \neq f(l)$ and $k$ and $l$ belong to the same cycle $c=(i_1\;\cdots\;i_n)$ of $\alpha_{s_0}$. Considering general elements
    $$\forall i \in [L], \qquad x_i = \sum_{j \in J_i} \bigotimes_{s=1}^r x^{(s)}_{j|i},$$
    we have that, for all choices of $j_1, \ldots, j_{n}$, $\tilde \kappa^{(s_0)}_c(x^{(s_0)}_{j_1|i_1}, \ldots, x^{(s_0)}_{j_{n}|i_n})$ is a \emph{mixed cumulant}, hence it is zero by the free independence assumption for the sub-algebras $\A^{(s_0)}_i$. We have thus 
    $$\kappa_{\underline{\alpha}}(x_1, \ldots, x_p) = \sum_{\underline{j} \in \bigtimes_{i=1}^p J_i} \prod_{s=1}^r \prod_{c_s \in \alpha_s} \tilde \kappa_{c_s}(\underline{x}^{(s)}_{\underline{j}}\big|_{c_s})=0,$$
    proving the claim of tensor free independence.
\end{proof}
}

We can now present a very simple example of tensor free random variables which are not free in the usual, non-tensor, sense. 

{
\begin{example} \label{ex-TensorFreeNonFree}
    For $r=2$, consider $k$ freely independent semicircular random variables $s_1, \ldots, s_k$ in some $*$-probability space $(\A,\phi)$. Then, the random variables $s_1 \otimes s_1, s_2 \otimes s_2, \ldots, s_k \otimes s_k \in (\A, \phi)^{\otimes 2}$ are:
    \begin{itemize}
        \item tensor freely independent
        \item not freely independent (in the usual sense).
    \end{itemize}

    The first point follows from \cref{prop:tensor-free-tensor-product}, while the second point has been established in several previous works \cite[Theorem 1.8]{collins2017freeness}, \cite[Corollary 1.2]{LSY24}; see also \cite[Proposition 5.9]{nica2016free} for the relation between the sum of these random variables and meanders.

   On the other hand, our notion of tensor free cumulants provides a simple argument to show non-freeness between $x=s_1\otimes s_1$ and $y=s_2\otimes s_2$. Indeed, \cref{prop-cumulant-from-tensorcumulant,lem-TensorCumulantProd} and the freeness between $s_1, s_2$ imply that
    \begin{align*}
        \tilde{\kappa}_4(x^2,y^2,x^2,y^2)&=\sum_{\substack{\beta_1,\beta_2\in S_{NC}(\gamma_4)\\ \beta_1\vee\beta_2=\gamma_4}}\kappa_{\beta_1,\beta_2}(x^2,y^2,x^2,y^2)\\
        &=\sum_{\substack{\beta_1,\beta_2\in S_{NC}(\gamma_4)\\ \beta_1\vee\beta_2=\gamma_4}} \tilde{\kappa}_{\beta_1}(s_1^2,s_2^2,s_1^2,s_2^2) \; \tilde{\kappa}_{\beta_2}(s_1^2,s_2^2,s_1^2,s_2^2)\\
        &=\sum_{\{\beta_1,\beta_2\}=\{(13)(2)(4),(1)(3)(24)\}} \tilde{\kappa}_{\beta_1}(s_1^2,s_2^2,s_1^2,s_2^2) \; \tilde{\kappa}_{\beta_2}(s_1^2,s_2^2,s_1^2,s_2^2)\\
        &=2\cdot \tilde{\kappa}_2(s_1^2) \, \tilde{\kappa}_1(s_2^2)^2 \,  \tilde{\kappa}_2(s_2^2) \, \tilde{\kappa}_1(s_1^2)^2 \neq 0.\\
    \end{align*}
\end{example}
}

The following result contains the tensor version of \cite[Remark 5.20]{nica2006lectures}; we leave the proof to the reader. 

\begin{proposition}
    The tensor freeness is \emph{symmetric} and \emph{associative}: for subalgebras $\A, \B, \mathcal{C}$ of a tensor probability space
    \begin{enumerate}
        \item If $\A$ is tensor free from $\B$, then $\B$ is tensor free from $\A$.
    
        \item If $\A$ and $\B$ are tensor free and if $\A\cup \B$ and $\mathcal{C}$ are tensor free, then $\A$, $\B$, and $\mathcal{C}$ are tensor free.
    \end{enumerate}
\end{proposition}

As advertised, tensor freeness gives a universal way to compute the \emph{joint} tensor distribution of tensor freely independent elements from their marginal tensor distribution. This can be either expressed from the vanishing property of tensor free cumulants, or at the level of the trace invariants, as follows. We refer the reader to \cref{cor-freemoment} for the similar result in the case of usual (non-tensor) freeness.

\begin{proposition}\label{prop-tenfreemoment}
Subalgebras $\A_1,\ldots, \A_L$ of an $r$-partite tensor probability space are tensor freely independent if and only if for $p \geq 1$, $\underline{\alpha}\in S_p^r$, for every function $f:[p]\to [L]$, and $x_j\in \A_{f(j)}$, $j \in [p]$:
\begin{equation} \label{eq-TensorFreeMoment}
    \varphi_{\underline{\alpha}}(x_1,\ldots, x_p) = \sum_{\substack{\underline{\beta}\in S_{NC}(\underline{\alpha}) \\ \bigvee_{\mathcal{P}}\Pi(\beta_s)\leq  \ker f}}\kappa_{\underline{\beta}}(x_1,\ldots, x_p)=\sum_{\substack{\underline{\beta}\in S_{NC}(\underline{\alpha}) \\ \bigvee_{\mathcal{P}}\Pi(\beta_s)\leq  \ker f}}\prod_{i\in f([p])} \kappa_{\underline{\beta}|_{f^{-1}(i)}}((x_j)_{j\in f^{-1}(i)}).
\end{equation}
\end{proposition} 
\begin{proof}
The tensor freeness implies the formula \cref{eq-TensorFreeMoment} by applying the tensor moment-free cumulant formula \cref{eq:tensor-moment-free-cumulant} and ignoring the mixed tensor free cumulants. Conversely, if \cref{eq-TensorFreeMoment} holds for every $\underline{\alpha}\in S_p^r$, then the inversion formula \cref{eq:tensor-free-cumulant-moment} implies that $\kappa_{\underline{\beta}}(x_1,\ldots, x_p)=0$ whenever $\bigvee_{\mathcal{P}}\Pi(\beta_s)\not\leq \ker f$.
\end{proof}

Although tensor freeness is not the same notion as freeness, they share several similarities. For example, if $x,y$ are two tensor free elements, then the formula \cref{eq-TensorFreeMoment} implies that
    $$\varphi(xy)=\varphi_{\underline{\gamma_2}}(x,y)=\kappa_{\underline{\id_2}}(x,y)=\varphi(x)\varphi(y),$$
since $\Pi(\beta_s)\leq \ker f=0_2$ implies $\beta_s\equiv \id_2$. On the other hand, the same formula $\varphi(xy)=\varphi(x)\varphi(y)$ holds if $x,y$ are \textit{freely independent}, following from \cref{eq-FreeMoment}. More generally, if $(\A_i)_{i\in [L]}$ are freely independent unital subalgebras of a non-commutative probability space $(\A,\varphi)$, $x_j\in \A_{f(j)}$ for $j\in [p]$, and if $\ker f$ is non-crossing, then \cref{eq-FreeMoment} and \cref{eq-FreeMomentCumulant} imply that
    $$\varphi(x_1\cdots x_p)=\sum_{\substack{\pi\in NC(p)\\ \pi\leq \ker f}}\tilde{\kappa}_{\pi}(x_1,\ldots, x_p)=\varphi_{\ker f}(x_1,\ldots, x_p)=\prod_{i\in f([p])}\varphi_{f^{-1}(i)}((x_j)_{j\in f^{-1}(i)}).$$
We show in the following that the same formula holds for \textit{tensor free} elements.

\begin{corollary} \label{cor-TensorFreeFactor}
Let $f$ and $x_1,\ldots x_p$ be as in \cref{prop-tenfreemoment}, and suppose that, for all $s\in [r]$, there exists $\sigma_s\in S_p$ such that $\Pi(\sigma_s)=\ker f$ and $\sigma_s,\alpha_s\in S_{NC}(\gamma^{(s)})$ for some full cycle $\gamma^{(s)}$. Then we have
    $$\varphi_{\underline{\alpha}}(x_1,\ldots, x_p)=\varphi_{\underline{\alpha}\wedge \underline{\sigma}}(x_1,\ldots, x_p)=\prod_{i\in f([p])}\varphi_{\underline{\alpha}\wedge \underline{\sigma}|_{f^{-1}(i)}}((x_j)_{j\in f^{-1}(i)}).$$
In particular, whenever $\ker f\in NC(p)$, we have
    $$\displaystyle \varphi(x_1\cdots x_p)=\varphi_{\ker f}(x_1,\ldots, x_p)=\prod_{i\in f([p])}\varphi_{f^{-1}(i)}((x_j)_{j\in f^{-1}(i)}).$$
\end{corollary}
\begin{proof}
Under the assumption, the condition $\bigvee_{\mathcal{P}} \Pi(\beta_s)\leq \ker f$ in \cref{eq-TensorFreeMoment} is equivalent to $\underline{\beta}\in S_{NC}(\underline{\sigma})$. The condition $\underline{\alpha},\underline{\sigma}\in S_{NC}(\underline{\gamma})$ further implies $S_{NC}(\underline{\alpha})\cap S_{NC}(\underline{\sigma})=S_{NC}(\underline{\alpha}\wedge \underline{\sigma})$, so we have
\begin{align*}
    \varphi_{\underline{\alpha}}(x_1,\ldots, x_p)     &= \sum_{\underline{\beta}\in S_{NC}(\underline{\alpha}\wedge \underline{\sigma})} \kappa_{\beta}(x_1,\ldots, x_p) = \varphi_{\underline{\alpha}\wedge \underline{\sigma}} (x_1,\ldots, x_p) = \prod_{i\in f([p])}\varphi_{\underline{\alpha}\wedge \underline{\sigma}|_{f^{-1}(i)}}((x_j)_{j\in f^{-1}(i)}).
\end{align*}
The last assertion follows by taking $\underline{\alpha}=\underline{\gamma_p}$.
\end{proof}

One of the nice properties of usual, non-tensor, free cumulants is the additivity: if $x_1,\ldots, x_k$ are freely independent, then
\begin{equation} \label{eq-CumulantsAdd}
    \tilde{\kappa}_p(x_1+\cdots + x_k)=\tilde{\kappa}_p(x_1)+\cdots + \tilde{\kappa}_p(x_k),\qquad \forall \, p\geq 1,
\end{equation}
see \cite[Proposition 12.3]{nica2006lectures}. We have an analogous property for tensor freely independent elements; {see \cite[Proposition 4.12]{KMW24}, \cite[Section 3]{BB24}, and \cite[Proposition 4.9]{collins2024free} for the same property in slightly different settings.}

\begin{proposition} \label{prop-TensorCumulantAdditivity}
If $x_1,\ldots, x_k$ are tensor freely independent elements of a tensor probability space, and if $\underline{\alpha}\in S_p^r$ is \emph{irreducible}, then
    $$\kappa_{\underline{\alpha}}(x_1+\cdots + x_k)=\sum_{i=1}^k \kappa_{\underline{\alpha}}(x_i).$$
\end{proposition}
\begin{proof}
    The result follows from the multilinearity and the vanishing of mixed tensor free cumulants properties: 
    \begin{align*}
        \kappa_{\underline{\alpha}}(x_1+\cdots + x_k)&=\sum_{f:[p] \to [k]} \kappa_{\underline{\alpha}}(x_{f(1)}, \ldots, x_{f(p)}) = \sum_{\substack{f:[p] \to [k]\\f \text{ constant}}} \kappa_{\underline{\alpha}}(x_{f(1)}, \ldots, x_{f(p)}) =\sum_{i=1}^k \kappa_{\underline{\alpha}}(x_i).
    \end{align*}
\end{proof}

The following gives the analogue of \cite[Theorem 11.12]{nica2006lectures} in which the behavior of the free cumulants under \emph{grouping} is described.

\begin{proposition} \label{prop-TensorCumulantProd}
Consider $p$ random variables $x_1, \ldots, x_p \in \A$ from a tensor probability space. Group them as follows: for $q\leq p$ and $0=i(0)<\cdots <i(q)=p$ define
$$X_j := x_{i(j-1)+1}\cdots x_{i(j)} \qquad \text{ for } j \in [q].$$
For $\underline{\alpha}\in S_q^r$, define $\underline{\hat{0}}$ and $\underline{\hat{\alpha}}\in S_p^r$ by 
\begin{align*}
    \hat{0}_s\equiv \hat{0} &:=(1\,\cdots \,i(1))\cdots (i(j-1)+1\,\cdots \,i(j)) \cdots (i(q-1)+1\,\cdots \,p),\\
    \hat{\alpha}_s &:=\hat{c}_1\cdots \hat{c}_l,
\end{align*}
where $\alpha_s=c_1\cdots c_l$ is the cycle decomposition of $\alpha_s$ and if $c_k=(j_1\;\cdots\; j_n)$, then
    $$\hat{c}_k:=(i(j_1-1)+1\;\cdots\; i(j_1)\;\;i(j_2-1)+1\;\cdots\; i(j_2)\;\cdots\; i(j_n-1)+1\;\cdots\; i(j_n)). $$
Then
    $$\kappa_{\underline{\alpha}}(X_1,\ldots, X_q)=\sum_{\substack{\underline{\beta} \in S_{NC}(\underline{\hat{\alpha}}) \\ \beta_s\vee \hat{0}=\hat{\alpha}_s \; \forall s}}\kappa_{\underline{\beta}}(x_1,\ldots, x_p).$$
\end{proposition}
\begin{proof}
According to \cite[Remarks 11.11]{nica2006lectures}, the map $\sigma\in S_q \mapsto \hat{\sigma}\in S_p$ defines a lattice isomorphism between $(S_{NC}(\alpha),\leq)$ and $(\{\beta\in S_p: \hat{0}\leq \beta\leq \hat{\alpha}\}, \leq)$ for all $\alpha\in S_q$. Therefore, for any $\underline{\alpha}\in S_q^r$, we have by M\"{o}bius inversion,
\begin{align*}
\kappa_{\underline{\alpha}}(X_1,\ldots, X_q)&=\sum_{\underline{\sigma}\,\leq \,\underline{\alpha}} \varphi_{\underline{\sigma}}(X_1,\ldots, X_q) \Mob(\underline{\sigma}^{-1}\underline{\alpha})\\
&=\sum_{\underline{\hat{0}}\,\leq\, \underline{\beta}\,\leq\, \underline{\hat{\alpha}}} \varphi_{\underline{\beta}}(x_1,\ldots, x_q) \Mob(\underline{\beta}^{-1} \underline{\hat{\alpha}})\\
&=\sum_{\substack{\underline{\beta} \in S_{NC}(\underline{\hat{\alpha}}) \\ \beta_s\vee \hat{0}=\hat{\alpha}_s \; \forall s}}\kappa_{\underline{\beta}}(x_1,\ldots, x_p).
\end{align*}
Note that we have used the substitution property \cref{def:tensor-ncps}-\cref{def:tensor-ncps-substitution} in the second equality above.
\end{proof}

We shall now prove that the unit element $1$ in a tensor probability space is tensor freely independent from the whole algebra $\A$; this property is the generalization of \cite[Lemma 5.17]{nica2006lectures} to the tensor case. Recall from \cref{rk:unitality} that the tensor moments of $1 \in \A$ are all 1; its tensor free cumulants can be easily computed: 
$$\kappa_{\underline{\alpha}}(1)=\prod_{s=1}^r\mathds 1_{\alpha_s = \id_p} = \mathds 1_{\underline{\alpha} = \underline{\id_p}}.$$

\begin{proposition} \label{prop-1-TensorFree}
Let $(\A, \phi, (\phi_{\underline{\alpha}}))$ be a tensor probability space, and consider $p\geq 2$ random variables $x_1,\ldots, x_p\in \A$ such that $x_j=1$ for some $j\in [p]$. Then for every $\underline{\alpha}\in S_p^r$ which connects $j$ with other elements, i.e.~there exists $s \in [r]$ such that $\alpha_s(j) \neq j$, we have $\kappa_{\underline{\alpha}}(x_1,\ldots, x_p)=0$. In particular, $1$ is tensor freely independent from the whole algebra $\A$.  
\end{proposition}
\begin{proof}
By the permutation invariance of $\kappa_{\underline{\alpha}}$ (see \cref{rk:tensor-free-cumulants-permutation-invariance}), we may assume that $x_p=1$. We shall show the statement by induction on the number $m=\sum_{s=1}^r |\alpha_s|$. 
If $m=1$, then the condition says that $\alpha_s=\id_p$ for all but one permutation $\alpha_{s_0}$ which is a transposition of the form $(j\,p)$. Therefore, the unitality property from \cref{def:tensor-ncps}-\cref{def:tensor-ncps-unitality} implies that
\begin{align*}
    \kappa_{\underline{\alpha}}(x_1,\ldots, x_{p-1},1)&=\varphi_{\underline{\alpha}}(x_1,\ldots, x_{p-1},1)-\varphi_{\underline{\id_p}}(x_1,\ldots, x_{p-1},1)\\
    &=\varphi_{\underline{\id_{p-1}}}(x_1,\ldots, x_{p-1})-\varphi_{\underline{\id_{p-1}}}(x_1,\ldots, x_{p-1})=0.
\end{align*}
Now if $m\geq 2$, then by the induction hypothesis, we have $\kappa_{\underline{\beta}}(x_1,\ldots, x_{p-1},1)=0$ for all $\underline{\beta}\lneq \underline{\alpha}$ such that $\underline{\beta}$ connects $p$ with some other elements. Therefore, the tensor moment-free cumulant relation \cref{eq:tensor-moment-free-cumulant} implies that
\begin{align*}
\varphi_{\underline{\alpha}}(x_1,\ldots, x_{p-1},1)&=\sum_{\underline{\beta}\,\leq \, \underline{\alpha}}\kappa_{\underline{\beta}}(x_1,\ldots, x_{p-1},1)\\
&=\kappa_{\underline{\alpha}}(x_1,\ldots, x_{p-1},1)+\sum_{{\underline{\beta}\,\lneq \, \underline{\alpha}}}\kappa_{\underline{\beta}}(x_1,\ldots, x_{p-1},1)\\
&=\kappa_{\underline{\alpha}}(x_1,\ldots, x_{p-1},1)+\sum_{\underline{\beta}\,\leq \, \underline{\alpha}'\sqcup \underline{(p)}}\kappa_{\underline{\beta}}(x_1,\ldots, x_{p-1},1)\\
&=\kappa_{\underline{\alpha}}(x_1,\ldots, x_{p-1},1)+\varphi_{\underline{\alpha}'}(x_1,\ldots, x_{p-1})
\end{align*}
where $\underline{\alpha}'\in S_{p-1}^r$ is obtained by erasing $p$ in every permutation $\alpha_s$. By unitality, we have $\varphi_{\underline{\alpha}}(x_1,\ldots, x_{p-1},1)=\varphi_{\underline{\alpha}'}(x_1,\ldots, x_{p-1})$, and therefore, $\kappa_{\underline{\alpha}}(x_1,\ldots, x_{p-1},1)=0$.

The final claim follows from the vanishing of mixed tensor free cumulants characterization of tensor freeness. 
\end{proof}

The next result shows that in order to show the tensor freeness of the algebras generated by some family of subsets, one can check the vanishing on mixed tensor free cumulants only on elements of the given subsets. 

\begin{theorem} \label{thm-TensorFreeSubsets}
Let $(\A, \phi, (\phi_{\underline{\alpha}}))$ be an $r$-partite tensor probability space, and consider a family of subsets $(\W_i)_{i \in [L]} \subseteq \A$. Then, the following statements are equivalent:
\begin{enumerate}
    \item $(\W_i)_{i \in [L]}$ are tensor freely independent (in the sense that the algebras they generate are tensor freely independent, see \cref{def:tensor-freeness});
    \item The vanishing property of mixed tensor free cumulants holds for elements of the $\W_i$'s: $\kappa_{\underline{\alpha}}(x_1,\ldots, x_p)=0$ whenever $\underline{\alpha}\in S_p^r$ is irreducible, and $x_j\in \W_{f(j)}$ for a non-constant function $f:[p] \to [L]$.
\end{enumerate}
{In particular, elements $x_1,\ldots, x_L\in\A$ are tensor freely independent if and only if 
    $$\kappa_{\underline{\alpha}}(x_{f(1)},\ldots, x_{f(p)})=0$$
whenever $\underline{\alpha}\in (S_p)^r$ is irreducible and $f:[p]\to [L]$ is a non-constant function.}
\end{theorem}
\begin{proof}
We may assume that $1\in \W_i$ for each $i\in [L]$ since $1$ is tensor free from $\A$ (\cref{prop-1-TensorFree}), and $\W_i$ and $\W_i\cup \{1\}$ generate the same unital subalgebra.

The direction (1) $\Rightarrow$ (2) is clear by definition. For the converse, assume the condition (2), and let us show that the same property holds for all elements of $\A_i$, the unital algebra generated by $\W_i$. By multilinearity of tensor free cumulants, we may take the elements of $\A_i$ from products of elements of $\W_i$. Therefore, if we set $0=i(0)<\cdots <i(q)=p$, $X_j=x_{i(j-1)+1}\cdots x_{i(j)}$, and $x_{i(j-1)+1},\ldots, x_{i(j)}\in \W_{f(j)}$ for all $j\in [q]$, then we have
\begin{equation} \label{eq-SubsetTensorFree}
    \kappa_{\underline{\alpha}}(X_1,\ldots, X_q)=\sum_{\forall s \in [r] \, : \, \beta_s\vee \hat{0}=\hat{\alpha}_s}\kappa_{\underline{\beta}}(x_1,\ldots, x_p)
\end{equation}
for $\underline{\alpha}\in S_q^r$, where $\hat{0}$ and $\hat{\alpha}_s$ are defined as in \cref{prop-TensorCumulantProd}. Now suppose that $\underline{\alpha}$ is irreducible and a permutation $\alpha_{s_0}$ is such that $\alpha_{s_0}(j) = j'$ with $f(j)\neq f(j')$. Then the conditions $\beta_{s_0}\vee \hat{0}=\hat{\alpha}_{s_0}$ forces that $\beta_{s_0}$ connects two elements $i,i'\in [p]$ such that $x_i\in \W_{f(j)}$ and $x_{i'}\in \W_{f(j')}$. Therefore, the condition (2) implies that $\kappa_{\underline{\beta}}(x_1,\ldots, x_p)=0$ for every $\underline{\beta}$, and hence
$\kappa_{\underline{\alpha}}(X_1,\ldots, X_q)=0$, which shows (1).
\end{proof}

{
Combining \cref{thm-TensorFreeSubsets} with \cref{prop-tenfreemoment} directly gives the following tensor freeness criterion, which will be very useful for showing the asymptotic tensor freeness of random matrices later.

\begin{corollary} \label{cor-TensorFreeSubsetMoment}
Subsets $\W_1,\ldots, \W_L$ of an $r$-partite tensor probability space are tensor freely independent if and only if, for every function $f:[p]\to [L]$ and $x_j\in \W_{f(j)}$ ($j \in [p]$), we have:
\begin{equation} \label{eq-TensorFreeSubsetMoment}
\varphi_{\underline{\alpha}}(x_1,\ldots, x_p) =\sum_{\substack{\underline{\beta}\in S_{NC}(\underline{\alpha}) \\ \bigvee_{\mathcal{P}}\Pi(\beta_s)\leq  \ker f}}\prod_{i\in f([p])} \kappa_{\underline{\beta}|_{f^{-1}(i)}}((x_j)_{j\in f^{-1}(i)}), \quad p \geq 1,\;\; \underline{\alpha}\in S_p^r.
\end{equation}
\end{corollary}

}

\subsection{Moment conditions for tensor freeness}

Recall that the usual notion of freeness was originally defined without using free cumulants; namely, every alternating product of centered elements remains centered. Since tensor freeness extends freeness and even coincides with it in the case $r=1$, it is natural to ask whether similar (and possibly equivalent) conditions can be established for tensor freeness. One straightforward approach is to further expand each $\kappa_{\underline{\beta}}$ in \cref{eq-TensorFreeMoment}, using \cref{eq:tensor-free-cumulant-moment}, to obtain that
\begin{align*}
\varphi_{\underline{\alpha}}(x_1,\ldots, x_p) &= \sum_{\substack{\underline{\beta}\in S_{NC}(\underline{\alpha}) \\ \bigvee_{\mathcal{P}}\Pi(\beta_s)\leq  \ker f}} \sum_{\underline{\alpha}'\in S_{NC}(\underline{\beta})}\Mob({\underline{\alpha}'}^{-1}\underline{\beta})\, \varphi_{\underline{\alpha}'}(x_1,\ldots, x_p)\\
&= \sum_{\substack{\underline{\alpha}'\in S_{NC}(\underline{\alpha}) \\ \bigvee_{\mathcal{P}}\Pi(\alpha_s')\leq  \ker f}} \bigg( \sum_{\substack{\underline{\alpha}'\leq \underline{\beta} \leq \underline{\alpha} \\ \bigvee_{\mathcal{P}}\Pi(\beta_s)\leq  \ker f}} \Mob({\underline{\alpha}'}^{-1}\underline{\beta}) \bigg) \varphi_{\underline{\alpha}'}(x_1,\ldots, x_p)
\end{align*}
for $x_j\in \A_{f(j)}$ ($j\in [p]$) where $(\A_i)_{i\in [L]}$ are tensor free subalgebras. From this, one can conclude that:
\begin{center}
    $\varphi_{\underline{\alpha}'}(x_1,\ldots, x_p)=0$ for all $\underline{\alpha}'\in S_{NC}(\underline{\alpha})$ such that $\bigvee_{\mathcal{P}}\Pi(\alpha_s')\leq \ker f$ $\implies \varphi_{\underline{\alpha}}(x_1,\ldots, x_p)=0$.
\end{center}

\medskip

In the following, we improve the above and give necessary moment conditions for tensor freeness, inspired by \cite[Theorem 5]{BB24}. First, for a permutation $\alpha\in S_p$ and a partition $\pi\in \mathcal{P}(p)$, let us define $\alpha^{\pi}$ by the maximal permutation $\beta\in (S_{NC}(\alpha),\leq)$ satisfying the following two conditions:
\begin{enumerate}
    \item  $\Pi(\beta)\leq \pi$,

    \item every cycle $c$ of $\beta$ is \textit{cyclic} on a cycle of $\alpha$: if $c=(i_1\,\cdots\,i_n)$, then there is a cycle $c'$ of $\alpha$ which is of the form $c'=(i_1\,\cdots \,i_m)$ for $m\geq n$.
\end{enumerate}
Intuitively, for each cycle of $\alpha$, we can group all (cyclically) consecutive elements that belong to the same block of 
$\pi$ to form the cycles  of $\alpha^{\pi}$. For example, if $\alpha=(1\,7\,3)(2\,5\,4\,6)\in S_7$ and $\pi=\{\{1,7\},\{2,4,5\},\{3,6\}\} \in \mathcal{P}(7)$, then
\begin{align}
    \alpha^{\pi} &=(1\,7)(3)(2\,5\,4)(6), \label{eq-AlphaOverPi} \\
    {\gamma_7}^{\pi} &=(7\,1\,2\,3\,4\,5\,6)^{\pi}=(7\,1)(2)(3)(4\,5)(6) \nonumber.    
\end{align}

\begin{proposition} \label{prop-TensorFreeVanishingMoment}
Let $\A_1,\ldots, \A_L$ be tensor freely independent subalgebras of $\A$. For $\underline{\alpha}\in (S_p)^r$ and a function $f:[p]\to [L]$, let us define the set
    $$\mathcal{S}_{\underline{\alpha},f}:=\{\underline{\beta}\in (S_p)^r: \underline{{\alpha}}^{\ker f}\leq \underline{\beta}\leq \underline{\alpha}\;\; \text{ and } \;\; \bigvee \Pi(\beta_s)\leq \ker f\} \subseteq S_{NC}(\underline{\alpha}).$$
Note that $\underline{{\alpha}}^{\ker f}=(\alpha_1^{\ker f},\ldots, \alpha_r^{\ker f})\in \mathcal{S}_{\underline{\alpha},f}$. Then we have
\begin{equation} \label{eq-VanishingMoment}
    \varphi_{\underline{\alpha}}(x_1,\ldots, x_p)=0 \text{ whenever $x_j\in \A_{f(j)}$ for $j\in [p]$ and $\varphi_{\underline{\beta}}(x_1,\ldots, x_p)=0$ for all } \underline{\beta}\in \mathcal{S}_{\underline{\alpha},f}
\end{equation}
\end{proposition}
\begin{proof}

Let us enumerate $\mathcal{S}_{\underline{\alpha},f}=\{\underline{\beta}^{(1)}, \underline{\beta}^{(2)}, \ldots, \underline{\beta}^{(n)}\}$. Then we can observe two crucial properties of $S_{\underline{\alpha},f}$ as follows: 
\begin{enumerate}
    \item The set $\mathcal{S}_{\underline{\alpha},f}$ is closed under the operation $\wedge$. Moreover, $\underline{\beta}^{(1)}:=\underline{\alpha}^{\ker f}$ is minimal in $\mathcal{S}_{\underline{\alpha},f}$.
    
    \item We have: $\underline{\beta}\in S_{NC}(\underline{\alpha})$ and $\bigvee_{\mathcal{P}}\Pi(\beta_s)\leq \ker f$ if and only if $\underline{\beta} \in S_{NC}(\underline{\beta}^{(i)})$ for some $i=1,\ldots, n$.
\end{enumerate}
Now let us assume that $\varphi_{\underline{\beta}^{(i)}}(x_1,\ldots, x_p)=0$ for all $i=1, \ldots, n$. Then we have
\begin{align*}
    \varphi_{\underline{\alpha}}(x_1,\ldots, x_p) &= \sum_{\substack{\underline{\beta}\in S_{NC}(\underline{\alpha}) \\ \bigvee_{\mathcal{P}}\Pi(\beta_s)\leq  \ker f}} \kappa_{\underline{\beta}}(x_1,\ldots, x_p) \\
    &= \sum_{\underline{\beta}\in S_{NC}(\underline{\alpha})} \kappa_{\underline{\beta}}(x_1,\ldots, x_p) \cdot \mathds{1}_{\bigcup_{i=1}^{n} S_{NC}(\underline{\beta}^{(i)})}(\underline{\beta})\\
    &= \sum_{\underline{\beta}\in S_{NC}(\underline{\alpha})} \kappa_{\underline{\beta}}(x_1,\ldots, x_p) \sum_{l=1}^n (-1)^{l-1} \sum_{I\subseteq [n],\; {\rm Card}(I)=l} \mathds{1}_{\bigcap_{i\in I} S_{NC}(\underline{\beta}^{(i)})}(\underline{\beta})\\
    &= \sum_{l=1}^n (-1)^{l-1} \sum_{I\subseteq [n],\; {\rm Card}(I)=l} \sum_{\underline{\beta}\in \bigcap_{i\in I} S_{NC}(\underline{\beta}^{(i)})} \kappa_{\underline{\beta}}(x_1,\ldots, x_p) \\
    &= \sum_{l=1}^n (-1)^{l-1} \sum_{I\subseteq [n],\; {\rm Card}(I)=l} \varphi_{\bigwedge_{i\in I}\underline{\beta}^{(i)}}(x_1,\ldots, x_p).
\end{align*}
Here we started from \cref{eq-TensorFreeMoment} and used the property (2) of $\mathcal{S}_{\underline{\alpha},f}$ in the second equality. In the third equality, we applied the \textit{inclusion–exclusion principle}
    $$ \mathds{1}_{\bigcup_{i=1}^n A_i}=\sum_{l=1}^n (-1)^{l-1}\sum_{I\subseteq [n],\,{\rm Card}(I)=l} \mathds{1}_{\bigcap_{i\in I}A_i} $$
for any finite sets $A_1,\ldots, A_n$, which is equivalent to $\prod_{i=1}^n \big(\mathds{1}_{\bigcup_j A_j}-\mathds{1}_{A_i}\big)=0$. Since $\bigcap_{i\in I}S_{NC}(\underline{\beta}^{(i)})=S_{NC}(\bigwedge_{i\in I}\underline{\beta}^{(i)})$ and $\bigwedge_{i\in I}\underline{\beta}^{(i)}\in \mathcal{S}_{\underline{\alpha},f}$ (property (1)), the last equality shows that $\varphi_{\underline{\alpha}}(x_1,\ldots, x_p)=0$.
\end{proof}

\begin{example} \label{ex-VanishingMoment}
Let us discuss several examples where \cref{prop-TensorFreeVanishingMoment} can be applied.
\begin{enumerate}
    \item If $\{x_1,x_2\}$ and $\{y_1,y_2\}$ are tensor free elements in a \textit{bipartite} tensor probability space (i.e. $r=2$), then we have 
        $$\varphi(x_1y_1x_2y_2)=\varphi_{\gamma_4,\gamma_4}(x_1,y_1,x_2,y_2)=0$$
    whenever $ \varphi(x_1)=\varphi(y_2)=0$ and $\varphi_{\substack{(1\,2)\\(1)(2)}}(x_1,x_2)=\varphi_{\substack{(1\,2)\\(1)(2)}}(y_1,y_2)=0$. Indeed, the latter condition implies $\varphi_{\beta_1,\beta_2}(x_1,y_1,x_2,y_2)=0$ whenever $\beta_1,\beta_2\in \{\id_4,(1\,3)(2)(4),(2\,4)(1)(3)\}$ (note that $\underline{\gamma_4}^{\ker f}=\underline{\id_4}$ and hence $\mathcal{S}_{\underline{\gamma_4},f}=\{\id_4,(13)(2)(4),(24)(1)(3)\}^2$ in this case). 

    \item In the case $r=1$, tensor freeness coincides with usual freeness, and we obtain that for freely independent subalgebras $(\A_i)_{i\in [L]}$ and  $x_j\in \A_{f(j)}$,
    \begin{center}
        $\varphi(x_1\cdots x_p)=0$ whenever $\varphi_{\pi}(x_1,\ldots, x_p)=0$ for all $\pi\in NC(p)$ with $\Pi(\gamma_p^{\ker f})\leq \pi\leq \ker f.$
    \end{center}
    In particular, if $x_j$'s are alternating and centered, i.e., $f(j)\neq f(j+1)$ and $\varphi(x_j)=0$ for all $j$, then $\varphi_{\pi}(x_1,\ldots, x_p)=0$ for every non-crossing partition $\pi$ with $\pi\leq \ker f$ since such $\pi$ contains a singleton as noted in \cite[Theorem 11.16]{nica2006lectures}.
\end{enumerate}
\end{example}

{We present below a simple corollary of \cref{prop-TensorFreeVanishingMoment}, see also \cite[Lemma 3.4]{LSY23} or \cite[Lemma 4.1]{LSY24} for a similar result.} 
\begin{corollary}\label{cor:vanishing-centered-tensor-moments}
    Consider tensor freely independent subalgebras $(\A_i)_{i \in [L]}$ inside some tensor probability space $(\A, \phi, (\phi_{\underline{\alpha}}))$, and \emph{centered} elements $x_1, \ldots, x_p \in \A$, with $x_j \in \A_{f(j)}$ for some function $f:[p] \to I$. Then, for all $\underline{\alpha} \in S_p^r$ such that
    \begin{equation*}
        \begin{rcases*}
            \forall s \in [r] \quad &$\beta_s \in S_{NC}(\alpha_s)$\\
            \forall s \in [r] \quad &$\Pi(\beta_s) \leq_P \ker f$        
        \end{rcases*}
     \implies \bigvee_{s \in [r]}^P \Pi(\beta_s) \text{ has a singleton},
    \end{equation*}
    the corresponding tensor moment is null:
    $$\phi_{\underline{\alpha}}(x_1, \ldots, x_p) = 0.$$
In particular, we have $\phi_{\underline{\alpha}}(x_1, \ldots, x_p) = 0$ whenever there exists $k\in [p]$ such that $f(k)\neq f(k')$ for all $k'\neq k$.
\end{corollary}
\begin{proof}
    In the setting of \cref{prop-TensorFreeVanishingMoment}, note that any $\underline{\beta} \in S_{\underline{\alpha},f}$ has a common singleton hence $\phi_{\underline{\beta}}(x_1, \ldots, x_p) = 0$ because the elements are centered. One can thus apply \cref{prop-TensorFreeVanishingMoment} and the conclusion follows.
\end{proof}

Let us consider the converse of the statement above: is it possible to prove tensor freeness starting from the vanishing tensor moment relations given in \cref{prop-TensorFreeVanishingMoment} and \cref{cor:vanishing-centered-tensor-moments}? We show next that the answer to this question is negative. 

First, consider two sub-algebras $\A_1, \A_2$ such that all the tensor moments as in the result above are vanishing. Let $x_{1,2} \in \A_1$ and $y_{1,2} \in \A_2$ be some \emph{centered} elements. We would like to check whether the mixed tensor free cumulant
$$\kappa_{\substack{(1)(2)(3)(4)\\(1243)}}(x_1,y_1,x_2,y_2)$$
is vanishing, using only the conditions in \cref{cor:vanishing-centered-tensor-moments}. Write
$$\kappa_{\substack{(1)(2)(3)(4)\\(1243)}}(x_1,y_1,x_2,y_2) = \sum_{\beta_2 \leq (1243)} \phi_{\substack{(1)(2)(3)(4)\\\beta_2}}(x_1,y_1,x_2,y_2) \Mob(\beta_2, (1243)).$$
Let us consider the conditions for a general term in the sum above to be vanishing according to \cref{cor:vanishing-centered-tensor-moments}:
$$\sigma_2 \leq \beta_2 \text{ and } \Pi(\sigma_2) \leq \ker f = \big\{\{1,3\},\{2,4\}\big\} \implies \sigma_2 \text{ has a fixed point}.$$
Note that the partition $\ker f$, with the ordering given by $(1243)$ is non-crossing, so 
$$\sigma_2 \leq \beta_2 \text{ and } \Pi(\sigma_2) \leq \ker f = \big\{\{1,3\},\{2,4\}\big\} \iff \sigma_2 \leq \beta_2 \wedge_{NC((1243))} \ker f.$$
Hence, all the terms in the sum above vanish, except for the ones such that $\beta_2 \wedge_{NC((1243))} \ker f$ does not have a singleton. In turn, this is equivalent to 
$$\beta_2 \wedge_{NC((1243))} \ker f = \ker f \iff \beta_2 \in \{(13)(24), (1243)\}.$$
In conclusion, 
\begin{align*}
    \kappa_{\substack{(1)(2)(3)(4)\\(1243)}}(x_1,y_1,x_2,y_2) &= \phi_{\substack{(1)(2)(3)(4)\\(1243)}}(x_1,y_1,x_2,y_2) + \phi_{\substack{(1)(2)(3)(4)\\(13)(24)}}(x_1,y_1,x_2,y_2)
    \cdot (-1)\\
    &= \phi_{\substack{(1)(2)(3)(4)\\(1243)}}(x_1,y_1,x_2,y_2) - \phi_{\substack{(1)(2)\\(12)}}(x_1,x_2)\phi_{\substack{(1)(2)\\(12)}}(y_1,y_2),
\end{align*}
none of which vanishes as per \cref{cor:vanishing-centered-tensor-moments}. In conclusion, the notion of tensor freeness defined via vanishing mixed tensor free cumulants implies the following relation: 
\begin{equation} \label{eq-ExVanishingMoment}
    \phi_{\substack{(1)(2)(3)(4)\\(1243)}}(x_1,y_1,x_2,y_2) = \phi_{\substack{(1)(2)\\(12)}}(x_1,x_2)\phi_{\substack{(1)(2)\\(12)}}(y_1,y_2)
\end{equation}
which is not captured by the vanishing of centered tensor moments in \cref{cor:vanishing-centered-tensor-moments}. 

Note however that the example above can be captured by the (more general) conditions in \cref{prop-TensorFreeVanishingMoment}. Indeed, if $\underline{\alpha}=(\id_4,(1243))$ and $\ker f=\{\{1,3\},\{2,4\}\}$, then $\underline{\alpha}^{\ker f}=(\id_4, (13)(24))$ and $\mathcal{S}_{\underline{\alpha},f}=\{\underline{\alpha}^{\ker f}\}$, so \cref{eq-VanishingMoment} says that
\begin{center}
    If $x_1,x_2\in \A_1$, $y_1,y_2\in \A_2$, and $\phi_{\substack{(1)(2)\\(12)}}(x_1,x_2)\phi_{\substack{(1)(2)\\(12)}}(y_1,y_2)=0$, then $\phi_{\substack{(1)(2)(3)(4)\\(1243)}}(x_1,y_1,x_2,y_2)=0$.
\end{center}
Now for any $x_1,x_2\in \A_1$ and $y_1,y_2\in \A_2$, we can choose $\lambda_1,\lambda_2\in \Comp$ such that
    $$\phi_{\substack{(1)(2)\\(12)}}(x_1-\lambda_1 1_\A,x_2-\lambda_2 1_\A)=\phi_{\substack{(1)(2)\\(12)}}(x_1,x_2)-\lambda_1\varphi(x_2)-\lambda_2\varphi(x_1)+\lambda_1\lambda_2=0.$$
This implies that $\phi_{\substack{(1)(2)(3)(4)\\(1243)}}(x_1-\lambda_1 1_\A,y_1,x_2-\lambda_2 1_\A,y_2)=0$, and hence we obtain \cref{eq-ExVanishingMoment}. In particular, we have
    $$\kappa_{\id_4,(1243)}(x_1,y_1,x_2,y_2)=\kappa_{\id_4,(1243)}(x_1-\varphi(x_1),y_1-\varphi(y_1),x_2-\varphi(x_2),y_2-\varphi(y_2))=0.$$

Finally, let us argue that the conditions in \cref{prop-TensorFreeVanishingMoment} may still be insufficient to fully capture tensor freeness in the general case.

\begin{lemma} \label{lem-BipartiteMoment}
Let $\A_1$ and $\A_2$ be unital subalgebras of a bipartite tensor probability space $(\A,\varphi,(\varphi_{\alpha_1,\alpha_2}))$ such that 
\begin{enumerate}
    \item up to order $p\leq 3$, every mixed tensor free cumulants $(\kappa_{\alpha_1,\alpha_2})_{\alpha_{1,2}\in S_p}$ between $\A_1$ and $\A_2$ vanish; 

    \item $\kappa_{\alpha_1,\alpha_2}(x_1,y_1,x_2,y_2)=0$ whenever $x_1,x_2\in \A_1$, $y_1,y_2\in \A_2$, and $\alpha_1,\alpha_2\in S_{NC}(\gamma_4)$ such that $\Pi(\alpha_1)\vee \Pi(\alpha_2)\not\leq \{\{1,3\},\{2,4\}\}$ and $\underline{\alpha}\neq (\gamma_4,\gamma_4)$.
\end{enumerate}
Then for all $x_1,x_2\in \A_1$ and $y_1,y_2\in \A_2$, we have
\begin{align*} 
    \kappa_{\gamma_4,\gamma_4}(x_1,y_1,x_2,y_2) &=\varphi({x}_1 {y}_1x_2y_2)-\phi_{\substack{(1)(2)\\(12)}}({x}_1,x_2)\phi_{\substack{(12)\\(1)(2)}}({y}_1,y_2) - \phi_{\substack{(12)\\(1)(2)}}({x}_1,x_2)\phi_{\substack{(1)(2)\\(12)}}({y}_1,y_2) \\
    & -\varphi(x_1) \varphi(x_2) \kappa_{\substack{(12)\\(12)}}(y_1,y_2) - \varphi(y_1) \varphi(y_2) \kappa_{\substack{(12)\\(12)}}(x_1,x_2)+\varphi(x_1)\varphi(x_2) \varphi(y_1) \varphi(y_2).
\end{align*}
\end{lemma}
\begin{proof}
Let $\ring{a}:=a-\varphi(a)1_\A$ for $a\in \A$. Since $1_\A$ is tensor free from $\A_1$ and $\A_2$, the condition (2) implies that
\begin{align*}
    \kappa_{\gamma_4,\gamma_4}(x_1,y_1,x_2,y_2) &= \kappa_{\gamma_4,\gamma_4}(\ring{x}_1,\ring{y}_1,{x}_2,{y}_2)\\
    &= \varphi(\ring{x}_1\ring{y}_1 {x}_2 {y}_2)-\sum_{\substack{\alpha_1,\alpha_2\in S_{NC}(\gamma_4)\\ (\alpha_1,\alpha_2)\neq (\gamma_4,\gamma_4)}} \kappa_{\alpha_1,\alpha_2}(\ring{x}_1, \ring{y}_1, {x}_2, {y}_2)\\
    &= \varphi(\ring{x}_1\ring{y}_1 {x}_2 {y}_2)-\kappa_{\substack{(13)(2)(4)\\ (1)(3)(24)}}(\ring{x}_1, \ring{y}_1, {x}_2, {y}_2)-\kappa_{\substack{(1)(3)(24)\\ (13)(2)(4)}}(\ring{x}_1, \ring{y}_1, {x}_2, {y}_2)\\
    &= \varphi(\ring{x}_1\ring{y}_1 {x}_2 {y}_2)-\kappa_{\substack{(12)\\ (1)(2)}}(\ring{x}_1,{x}_2) \kappa_{\substack{(1)(2)\\ (12)}}(\ring{y}_1,{y}_2)- \kappa_{\substack{(1)(2)\\ (12)}}(\ring{x}_1,{x}_2) \kappa_{\substack{(12)\\ (1)(2)}}(\ring{y}_1,{y}_2)\\
    &= \varphi(\ring{x}_1\ring{y}_1 {x}_2 {y}_2)-\varphi_{\substack{(12)\\ (1)(2)}}(\ring{x}_1,{x}_2) \varphi_{\substack{(1)(2)\\ (12)}}(\ring{y}_1,{y}_2)- \varphi_{\substack{(1)(2)\\ (12)}}(\ring{x}_1,{x}_2) \varphi_{\substack{(12)\\ (1)(2)}}(\ring{y}_1,{y}_2).
\end{align*}
On the other hand, the condition (1) implies that
\begin{align*}
    \varphi(\ring{x}_1 \ring{y}_1 x_2 y_2) &=\varphi(x_1y_1x_2y_2)-\varphi(x_1)\varphi(y_1x_2y_2)-\varphi(y_1)\varphi(x_1x_2y_2)+\varphi(x_1)\varphi(y_1)\varphi(x_2 y_2)\\
    &= \varphi(x_1y_1x_2y_2)-\varphi(x_1) \varphi(x_2) \varphi(y_1y_2)-\varphi(y_1)\varphi(y_2)\varphi(x_1x_2)+\varphi(x_1)\varphi(y_1)\varphi(x_2)\varphi(y_2).
\end{align*}

Now combining with the observation that $\varphi_{\underline{\alpha}}(\ring{a},b)=\varphi_{\underline{\alpha}}(a,b)-\varphi(a)\varphi(b)$ for all $a,b\in \A$ and $\underline{\alpha}\in (S_2)^2$, we have
\begin{align*}
    \kappa_{\gamma_4,\gamma_4}(x_1,y_1,x_2,y_2) =&\, \varphi({x}_1 {y}_1x_2y_2)-\phi_{\substack{(1)(2)\\(12)}}({x}_1,x_2)\phi_{\substack{(12)\\(1)(2)}}({y}_1,y_2) - \phi_{\substack{(12)\\(1)(2)}}({x}_1,x_2)\phi_{\substack{(1)(2)\\(12)}}({y}_1,y_2) \\
    & -\varphi(x_1) \varphi(x_2) \big( \varphi(y_1y_2) - \phi_{\substack{(12)\\(1)(2)}}({y}_1,y_2) - \phi_{\substack{(1)(2)\\(12)}}({y}_1,y_2) \big)\\
    & - \varphi(y_1) \varphi(y_2) \big( \varphi(x_1x_2) - \phi_{\substack{(12)\\(1)(2)}}({x}_1,x_2) - \phi_{\substack{(1)(2)\\(12)}}({x}_1,x_2) \big) \\
    & - \varphi(x_1)\varphi(x_2)\varphi(y_1)\varphi(y_2) \\
    =&\, \varphi({x}_1 {y}_1x_2y_2)-\phi_{\substack{(1)(2)\\(12)}}({x}_1,x_2)\phi_{\substack{(12)\\(1)(2)}}({y}_1,y_2) - \phi_{\substack{(12)\\(1)(2)}}({x}_1,x_2)\phi_{\substack{(1)(2)\\(12)}}({y}_1,y_2) \\
    & -\varphi(x_1) \varphi(x_2) \kappa_{\substack{(12)\\(12)}}(y_1,y_2) - \varphi(y_1) \varphi(y_2) \kappa_{\substack{(12)\\(12)}}(x_1,x_2)+\varphi(x_1)\varphi(x_2) \varphi(y_1) \varphi(y_2).
\end{align*}
\end{proof}

\begin{remark}
Let $\A_1$ and $\A_2$ be unital subalgebras of a bipartite tensor probability space $(\A,\varphi,(\varphi_{\alpha_1,\alpha_2}))$.
\begin{enumerate}
    \item Both conditions (1) and (2) in \cref{lem-BipartiteMoment} are satisfied if all the tensor moment conditions \cref{eq-VanishingMoment} hold for $\A_1$ and $\A_2$. For example, if $x\in \A_1$ and $y\in \A_2$, then we have
    \begin{align*}
        \kappa_{\substack{(12)\\(1)(2)}}(x,y &)=\kappa_{\substack{(12)\\(1)(2)}}(x-\varphi(x)1_\A,y)=\varphi_{\substack{(12)\\(1)(2)}}(x-\varphi(x)1_\A,y)-\varphi_{\substack{(1)(2)\\(1)(2)}}(x-\varphi(x)1_\A,y)=0
    \end{align*}
    from \cref{eq-VanishingMoment}. Similarly, we have $\kappa_{\substack{(1)(2)\\(12)}}(x,y)=0$ and
        $$\kappa_{\substack{(12)\\(12)}}(x,y)=\kappa_{\substack{(12)\\(12)}}(x-\varphi(x)1_A,y)=\varphi_{\substack{(12)\\(12)}}(x-\varphi(x)1_A,y)-\kappa_{\substack{(12)\\(1)(2)}}(x,y)-\kappa_{\substack{(1)(2)\\(12)}}(x,y)=0.$$
    One can check the remaining conditions using inductive arguments, which is left to the reader.

    \item Inducing $\kappa_{\gamma_4,\gamma_4}(x,y,x,y)=0$, for all $x\in \A_1$ and $y\in A_2$, from \cref{eq-VanishingMoment}, may require more assumptions on $\A_1$ and $\A_2$ (or on $\A$). For example, we may set $x_i=x-\lambda_i 1_\A$ and $y_i=y- \mu_i 1_\A$, for $i=1,2$ and $\lambda_{1,2},\mu_{1,2}\in \Comp$, and apply \cref{lem-BipartiteMoment} to obtain that
    \begin{align*}
        \kappa_{\gamma_4,\gamma_4}(x,y,x,y) &=\varphi({x}_1 {y}_1x_2y_2)-\phi_{\substack{(1)(2)\\(12)}}({x}_1,x_2)\phi_{\substack{(12)\\(1)(2)}}({y}_1,y_2) - \phi_{\substack{(12)\\(1)(2)}}({x}_1,x_2)\phi_{\substack{(1)(2)\\(12)}}({y}_1,y_2) \\
        & -\varphi(x_1) \varphi(x_2) \kappa_{\substack{(12)\\(12)}}(y) - \varphi(y_1) \varphi(y_2) \kappa_{\substack{(12)\\(12)}}(x)+\varphi(x_1)\varphi(x_2) \varphi(y_1) \varphi(y_2).
    \end{align*}
    In particular, one has $\kappa_{\gamma_4,\gamma_4}(x,y,x,y)=0$ if $ \varphi(x_1)=\varphi(y_2)=0$ and $\varphi_{\substack{(1\,2)\\(1)(2)}}(x_1,x_2)=\varphi_{\substack{(1\,2)\\(1)(2)}}(y_1,y_2)=0$, by \cref{ex-VanishingMoment} (1). However, this approach is not valid: one should take $\lambda_1=\varphi(x)$ in order to get $\varphi(x_1)=0$, but in that case
        $$\varphi_{\substack{(12)\\(1)(2)}}(x_1,x_2)=\varphi_{\substack{(12)\\(1)(2)}}(x)-\varphi(x)^2$$
    regardless of $\lambda_2\in \Comp$. This means that we cannot always choose $\lambda_1,\lambda_2$ such that $\varphi(x_1)=0$ and $\varphi_{\substack{(12)\\(1)(2)}}(x_1,x_2)=0$. In conclusion, in this specific situation, one cannot derive the vanishing of the mixed cumulant $\kappa_{\gamma_4,\gamma_4}(x,y,x,y)$ from the moment conditions of \cref{prop-TensorFreeVanishingMoment}, using the method from \cref{lem-BipartiteMoment}.
\end{enumerate}
\end{remark}

\bigskip

\subsection{Classical independence from tensor freeness}

As a final remark, let us point out that the classical notion of independence can be captured by tensor freeness in a very natural way. Classical independence is also known as tensor independence in the literature, in connection to the tensor product of classical probability spaces. In this work, to avoid confusion with our notions of tensor free independence, we use the term ``classical independence''. Recall that unital subalgebras $(\A_i)_{i\in I}$ of a noncommutative probability space $(\A,\varphi)$  are called \emph{classically independent} if the subalgebras commute with each other and
    $$\varphi(a_1\cdots a_p)=\varphi(a_1)\cdots \varphi(a_p)$$
whenever $a_1,\ldots, a_p$ are taken from mutually distinct algebras $\A_{i_1}, \ldots, \A_{i_p}$, respectively. Note that the classical independence precisely describes the probabilistic independence of random variables in the $*$-probability space $(L^{\infty-}(\mathbb{P}),\E)$.

\begin{proposition} \label{prop-TensorIndep}
Let $(\A,\varphi)$ be a noncommutative probability space with a tracial state $\varphi$ and $\A_1,\ldots, \A_L$ be classically independent subalgebras of $\A$. Define $\tilde{\A}_i$ as an embedding of $\A_i$ into $i$-th tensor component of ${\A}^{\otimes L}$, i.e.,
    $$\tilde{\A}_i:=1_{\A}\otimes \cdots \otimes \A_i\otimes \cdots \otimes 1_{\A}.$$
Then $(\tilde{\A}_i)_{i\in [L]}$ are tensor free subalgebras of the $L$-partite tensor probability space $\big(\A^{\otimes L},\varphi^{\otimes L}, (\bigotimes_{s=1}^L \varphi_{\alpha_s})\big)$.
{
Furthermore, $(\tilde{\A}_i)_{i\in [L]}$ has the same joint distribution with $(\A_i)_{i\in [L]}$ in the following sense:
    $$\text{For } {x}_j\in {\A}_{i_{j}} \text{ with $j\in [p]$},  \quad \phi^{\otimes L}(\tilde{x}_1, \ldots, \tilde{x}_p) = \phi(x_1, \ldots, x_p),$$
where each $\tilde{x}_j$ denotes the embedding of $x_j$ according to the identification $\A_{i_j}\cong \tilde{A}_{i_j}$.
}

\end{proposition}
\begin{proof}
Tensor freeness of $(\tilde{\A}_i)_{i\in [L]}$ follows from \cref{prop:tensor-free-tensor-product}, and the remaining assertions are also clear by definition.
\end{proof}

By combining with several properties of tensor free cumulants, we can recover the free cumulants of the classically independent variables derived in \cite[Theorem 1.2]{BD14}.

\begin{corollary}
Let $x_1,\ldots, x_L$ be a family of classically independent elements in a noncommutative probability space $(\A,\varphi)$. Then one has for $p\geq 1$,
    $$\tilde{\kappa}_p(x_1\cdots x_L)=\sum_{\substack{\pi_1,\ldots, \pi_L\in NC(p),\\ \bigvee_{NC}\pi_i=1_p}}\prod_{i=1}^L \tilde{\kappa}_{\pi_i}(x_i).$$
\end{corollary}
\begin{proof}
We may assume that $\varphi$ is tracial by replacing $\A$ with the unital algebra generated by $x_1,\ldots, x_L$ if necessary. As in \cref{prop-TensorIndep}, let $\tilde{x}_i:=1_{\A}\otimes \cdots \otimes x_i \otimes \cdots \otimes 1_{\A}$ be the embedding of $x_i$ into $i$-th tensor component of $\A^{\otimes L}$. Then $(\tilde{x}_i)_{i\in [L]}$ is a tensor free family in $\big(\A^{\otimes L},\varphi^{\otimes L}, (\bigotimes_{s=1}^r \varphi_{\alpha_s})\big)$ and has the same joint distribution with $(x_i)_{i\in [L]}$. Therefore, \cref{prop-cumulant-from-tensorcumulant,lem-TensorCumulantProd} imply that
\begin{align*}
    \kappa_p(x_1 \cdots x_L)&= \kappa_p (\tilde{x}_1 \cdots \tilde{x}_L)
    =\sum_{\substack{\underline{\alpha}\in S_{NC}(\gamma_p)^r,\\ \bigvee_{NC}\alpha_i=\gamma_p}}\kappa_{\underline{\alpha}}(x_1\otimes \cdots \otimes x_L)
    =\sum_{\substack{\pi_1,\ldots, \pi_L\in NC(p),\\ \bigvee_{NC}\pi_i=1_p}}\prod_{i=1}^L \tilde{\kappa}_{\pi_i}(x_i),
\end{align*}
where we used the correspondence $S_{NC}(\gamma_p)\cong NC(p)$ in the last equality.
\end{proof}

\section{Asymptotic tensor freeness of globally invariant random matrices} \label{sec-UItensorfree}

In the next three sections, we present a wide range of examples of random matrices exhibiting asymptotic behaviors that can be captured by the notion of \emph{tensor free independence} introduced in \cref{sec:tensor-free-independence}. Let us keep the same notations from \cref{sec-TensorNCPS}: we take $r$ sequences $d_{s,N}$ ($s=1,\ldots, r$) of positive integers which increase to $\infty$ as $N\to \infty$, and let $D_N:=d_{1,N}\cdots d_{r,N}$. Recall that a family $\W_N$ of $D_N\times D_N$ random matrices \textit{converges in tensor distribution} (in expectation) if for every $p\geq 1$ and every $p$-tuple of permutations $\underline{\alpha}\in (S_p)^r$, the limit 
    $$\lim_{N\to \infty} \E\big[\tr_{\underline{\alpha}}(X_{1,N},\ldots, X_{p,N})\big]=\lim_{N\to \infty} \frac{1}{d_{1,N}^{\#\alpha_1}\cdots d_{r,N}^{\#\alpha_s}} \E\big[\Tr_{\underline{\alpha}}(X_{1,N},\ldots, X_{p,N})\big]$$
exists for $X_{1,N},\ldots, X_{p,N}\in \W_N$. We say that families $\W_N^{(1)},\ldots, \W_N^{(L)}$ of $D_N\times D_N$ random matrices are \textit{asymptotically tensor freely independent} if there exists tensor free families $\W_1,\ldots, \W_L$ in an $r$-partite tensor probability space $(\A,\varphi,(\varphi_{\underline{\alpha}}))$ such that $\W_N^{(1)},\ldots, \W_N^{(L)}\to \W_1,\ldots, \W_L$ jointly in tensor distribution. Note that in this case, the family $\bigcup_{i=1}^L \W_N^{(i)}$ necessarily satisfies the tensor factorization property \cref{eq-condition-TensorFact} as in the discussion before \cref{lem-TensorLimitRealization}.

\medskip

Prior to our general results (\cref{thm-locui-tensorfree,thm-LocOI-tensorfree,thm-indepTranspose}) on LUI and LOI matrices, in this section, we first consider simpler models of random matrices: multipartite random matrices which are \textit{globally} unitary / orthogonal invariant. We show that within these classes, convergence in distribution is equivalent to convergence in \textit{tensor} distribution and asymptotic freeness is equivalent to asymptotic \textit{tensor} freeness. As a corollary, independent UI / OI random matrices automatically shows asymptotic tensor free independence, thanks to \cref{thm-UIasympfree,thm-OIAsympFree}. We remark that the similar analogue has been shown in {\cite[Theorem 1.1]{CDM24}}: for UI matrices, convergence in $*$-distribution is equivalent to convergence in \textit{traffic} distribution, and asymptotic $*$-freeness is equivalent to asymptotic \textit{traffic} independence.

\medskip

From now, we frequently drop out $N$ and simply write $X_1,\ldots, X_p\in \W_N$ (or $\W$) and $D=d_1\cdots d_r$, etc. Let us begin by providing the formula of joint tensor distribution between UI / OI random matrices, using the Weingarten calculi.

\begin{lemma} \label{lem-UIOITensorMoments}
Let $\W$ be a family of $D\times D$ random matrices, and let $X_1,\ldots, X_p\in \W$.
\begin{enumerate}
    \item If $\W$ is UI, then for $\underline{\alpha}\in (S_p)^r$,
    \begin{equation} \label{eq-UITensorInv}
        \E\left[\Tr_{\underline{\alpha}}(X_1,\ldots, X_p)\right]=\sum_{\sigma,\tau\in S_p} \E\left[\Tr_{\tau}(X_1,\ldots, X_p)\right] \left(\prod_{s=1}^r d_s^{\#(\sigma^{-1}\alpha_s)}\right) \Wg_{D}^{(U)}(\tau^{-1}\sigma)
    \end{equation}

    \item If $\W$ is OI, then for $\underline{\alpha}\in (S_p)^r$,
    \begin{equation} \label{eq-OITensorInv}
        \E\left[\Tr_{\underline{\alpha}}(X_1,\ldots, X_p)\right]=\sum_{\pi,\rho\in \mathcal{P}_2(\pm p)} \E\left[\Tr_{\rho\vee\delta}(X_1,\ldots, X_p)\right] \left(\prod_{s=1}^r d_s^{\#(\pi \vee \alpha_s\delta\alpha_s^{-1})}\right) \Wg_{D}^{(O)}(\rho,\pi)
    \end{equation}
\end{enumerate}
\end{lemma}
\begin{proof}
(1) Let $U$ be a $D\times D$ Haar random unitary matrix independent from $\W$. Then the unitary invariance of $\W$ says that the joint probability distribution of $(UX_1 U^*,\ldots, UX_p U^*)$ is equal to $(X_1,\ldots, X_p)$. In particular, we have
    $$\E[\Tr_{\underline{\alpha}}(X_1,\ldots,X_p)]= \E[\Tr_{\underline{\alpha}}(U X_1 U^*,\ldots,U X_p U^*)]=\E_{\W}\Big[\E_U\big[\Tr_{\underline{\alpha}}(U X_1 U^*,\ldots,U X_p U^*)\big| \W\big]\Big].$$
We now apply the (graphical) unitary Weingarten calculus \cref{thm:Weingarten} to the inner expectation to obtain (see \cref{fig:Wg-unitarily-invariant}): almost surely,
    $$\E_U\big[\Tr_{\underline{\alpha}}(U X_1 U^*,\ldots,U X_p U^*)\big| \W\big]=\sum_{\sigma,\tau\in S_p} \Tr_{\tau}(X_1,\ldots, X_p)\left(\prod_{s=1}^r d_s^{\#(\sigma^{-1}\alpha_s)}\right) \Wg_{D}^{(U)}(\tau^{-1}\sigma).$$

\begin{figure}[!htb]
    \centering
    \includegraphics[width=0.85 \linewidth]{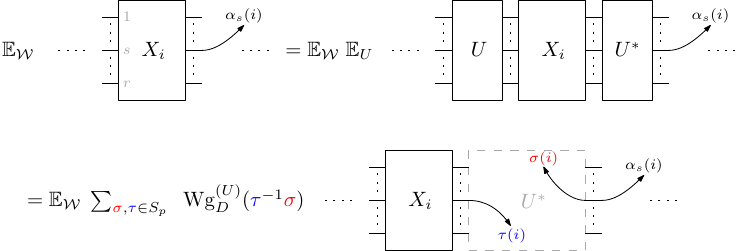}
    \caption{Using the unitary invariance of the family $\mathcal W$, one can replace each variable $X_i$ by $UX_iU^*$. Performing the Weingarten integration with respect to $U$ results in a sum over diagrams indexed by two permutations.}
    \label{fig:Wg-unitarily-invariant}
\end{figure}
    
Therefore, taking the expectation $\E_\W$ gives \cref{eq-UITensorInv}.

(2) The proof is analogous to (1), the only difference residing in the application of the graphical orthogonal Weingarten formula from \cref{thm-GraphWeingarten} instead of the unitary one, see \cref{fig:Wg-orthogonal-invariant}.

\begin{figure}[!htb]
    \centering
    \includegraphics[width=0.95 \linewidth]{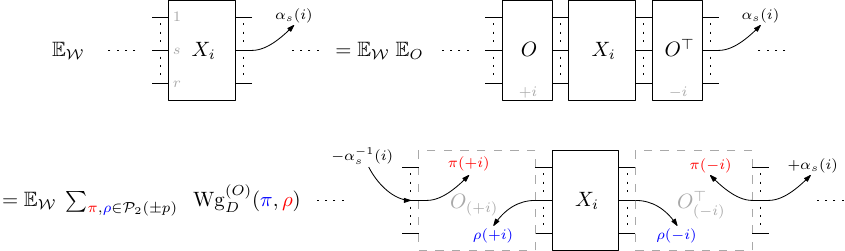}
    \caption{Using the orthogonal invariance of the family $\mathcal W$, one can replace each variable $X_i$ by $OX_iO^\top$. Performing the Weingarten integration with respect to $O$ results in a sum over diagrams indexed by two pairings that act on $2p$ points, $[\pm p]$.}
    \label{fig:Wg-orthogonal-invariant}
\end{figure}

\end{proof}

Let us begin first to investigate the asymptotic tensor distribution limit of UI random matrices.

\begin{proposition} \label{prop-ui-tensormoment}
Let $\mathcal{W}_N$ be a family of $D_N\times D_N$ UI matrices satisfying the factorization property (\cref{eq-condition-Fact}), and suppose $\W_N$ converges in distribution as $N\to \infty$. Then $\W_N$ satisfies the tensor factorization property (\cref{eq-condition-TensorFact}) and converges in tensor distribution. Furthermore, if $\W$ is the corresponding tensor distribution limit, then for $\underline{\alpha}\in (S_p)^r$ and $x_1,\ldots, x_p\in \mathcal{W}$, we have

\begin{align}
    \varphi_{\underline{\alpha}}(x_1,\ldots, x_p)&=\sum_{\sigma\in \bigcap_{s=1}^r S_{NC}(\alpha_s)} \tilde{\kappa}_{\sigma}(x_1,\ldots, x_p), \label{eq-ui-tensormoment}\\
    \kappa_{\underline{\alpha}}(x_1,\ldots, x_p)&=\begin{cases}
        \tilde{\kappa}_{\sigma}(x_1,\ldots, x_p) & \text{if $\alpha_s\equiv \sigma$ for some $\sigma\in S_p$,} \\
        0 & \text{otherwise.} \label{eq-ui-tensorcumulant}
    \end{cases}
\end{align}
\end{proposition}
\begin{proof}
Let $X_j=X_{j,N}\in \W_N$ ($j=1,\ldots, p$), and let $(x_1,\ldots,x_p)\in (\A,\varphi)$ be the corresponding joint distribution limit of $(X_1,\ldots, X_p)$. Then for $\underline{\alpha}\in (S_p)^r$ we can apply \cref{eq-UITensorInv} and the estimation of the unitary Weingarten function (\cref{eq-WeinAsymp1}) to have
\begin{align*}
    \E[\tr_{\underline{\alpha}}(X_1,\ldots, X_p)]
    &= \sum_{\sigma,\tau\in S_p} \E\left[\tr_{\tau}(X_1,\ldots, X_p)\right]  \left(\prod_{s=1}^r d_s^{\#\tau-\#\alpha_s}\,d_s^{\#(\sigma^{-1}\alpha_s)}\right) D_N^{-p-|\tau^{-1}\sigma|}(\Mob(\tau^{-1}\sigma)+o(1))\\
    &= \sum_{\sigma,\tau\in S_p} \E\left[\tr_{\tau}(X_1,\ldots, X_p)\right] \left(\prod_{s=1}^r d_s^{-(|\tau|+|\tau^{-1}\sigma|+|\sigma^{-1}\alpha_s|-|\alpha_s|)}\right) (\Mob(\tau^{-1}\sigma)+o(1))
\end{align*}
(recall that $D_N=d_1\cdots d_r$). Note that for each $s\in [r]$, the inequality $|\tau|+|\tau^{-1}\sigma|+|\sigma^{-1}\alpha_s|-|\alpha_s|\geq 0$ always holds, and it attains 0 if and only if $\tau\leq \sigma\leq \alpha_s$. Furthermore, the factorization property implies that $\E\left[\tr_{\tau}(X_1,\ldots, X_p)\right]\to \varphi_{\tau}(x_1,\ldots, x_p)$ for arbitary $\tau\in S_p$, so we have
\begin{align*}
    \lim_{N\to \infty}\E\big[\tr_{\underline{\alpha}}(X_1,\ldots, X_p)\big]&=\sum_{\sigma,\tau\in S_p}\left(\prod_{s=1}^r\mathds{1}_{\tau \leq \sigma \leq \alpha_s}\right) \varphi_{\tau}(x_1,\ldots, x_p)\Mob(\tau^{-1}\sigma) \\
    &= \sum_{\sigma\in \bigcap_{s=1}^r S_{NC}(\alpha_s)} \sum_{\tau\leq \sigma}\varphi_{\tau}(x_1,\ldots, x_p)\Mob(\tau^{-1}\sigma)\\
    &=\sum_{\sigma\in \bigcap_{s=1}^r S_{NC}(\alpha_s)} \tilde{\kappa}_{\sigma}(x_1,\ldots, x_p).
\end{align*}
Now if we define the family of complex numbers $(\kappa^{(0)}_{\underline{\alpha}})_{\underline{\alpha}\in (S_p)^r}$ by the same as in \cref{eq-ui-tensorcumulant}, then the above computation is equivalent to
    $$\lim_{N\to \infty}\E[\tr_{\underline{\alpha}}(X_1,\ldots, X_p)]= \sum_{\underline{\sigma} \in S_{NC}(\underline{\alpha})} \kappa^{(0)}_{\underline{\sigma}}(x_1,\ldots, x_p) \text{ for all } \underline{\alpha}\in (S_p)^r.$$
This suggests that $\W_N$ converges in tensor distribution, and if we call $(x_1,\ldots, x_p)$ again the corresponding tensor distribution limit of $(X_1,\ldots, X_p)$, then $\kappa_{\underline{\alpha}}(x_1,\ldots, x_p)=\kappa^{(0)}_{\underline{\alpha}}(x_1,\ldots, x_p)$ by M\"{o}bius inversion. Finally, the tensor factorization property follows from the multiplicativity of $\tilde{\kappa}_{\sigma}$ and the discussion above \cref{lem-TensorLimitRealization}.
\end{proof}

\begin{theorem} \label{thm-ui-tensorfree}
Let $\W_1,\ldots, \W_L$ be families of $D_N\times D_N$ random matrix ensembles such that $\bigcup_{i=1}^L \W_N^{(i)}$ is UI and satisfies the factorization property \cref{eq-condition-Fact}. Then $(\W_N^{(i)})_{i\in [L]}$ are asymptotically free if and only if asymptotically tensor free.
\end{theorem}
\begin{proof}
Suppose $(\W_N^{(i)})_{i\in [L]}$ are asymptotically free. Then by \cref{prop-ui-tensormoment}, the family $\bigcup_{i=1}^L \W_N^{(i)}$ converges in tensor distribution. Let us denote by $\widetilde{W}=\bigcup_{i=1}^L \W^{(i)}$ the corresponding tensor limit, where $\W_N^{(i)} \xrightarrow{\otimes\text{-distr}} \W^{(i)}$. If $x_j\in \W^{(f(j))}$ where $f(j)\neq f(k)$ for some $j$ and $k$, and if $\underline{\alpha}\in (S_p)^r$ is irreducible, then by \cref{eq-ui-tensorcumulant}, $\kappa_{\underline{\alpha}}(x_1,\ldots, x_p)=0$ unless $\alpha_s\equiv \sigma$ for some $\sigma\in S_p$. However, if $\underline{\alpha}=(\sigma,\ldots, \sigma)$, then the irreducibility of $\underline{\alpha}$ implies that $\sigma$ is a full cycle, and the freeness between $\W^{(f(j))}$ and $\W^{(f(k))}$ again implies $\kappa_{\underline{\alpha}}(x_1,\ldots, x_p)=\tilde{\kappa}_{\sigma}(x_1,\ldots, x_p)=0$. This shows the tensor freeness of $(\W^{(i)})_{i\in [L]}$.

Conversely, assume the asymptotic tensor freeness of $(\W_N^{(i)})_{i\in [L]}$ as $N\to\infty$ so that $\bigcup_{i=1}^L \W_N^{(i)} \xrightarrow{\otimes\text{-distr}} \widetilde{W}=\bigcup_{i=1}^L\W^{(i)}$ and $(\W^{(i)})_{i\in [L]}$ are tensor free. Then the definition guarantees that $\bigcup_{i=1}^L \W_N^{(i)} \to \widetilde{W}$ in (usual) distribution. Now if $x_j\in \W^{(f(j))}$ where $f(j)\neq f(k)$ for some $j$ and $k$, then \cref{eq-ui-tensorcumulant} again implies that
    $$\tilde{\kappa}_p(x_1,\ldots, x_p)=\kappa_{\gamma_p,\ldots, \gamma_p}(x_1,\ldots, x_p)=0.$$
Therefore, the families $(\W^{(i)})_{i\in [L]}$ are free.
\end{proof}

Note that we do not assume the independence property between the families in \cref{thm-ui-tensorfree}. Since independent UI random matrices exhibit the asymptotic freeness (\cref{thm-UIasympfree}), \cref{thm-ui-tensorfree} combined with \cref{prop-IndepUI} gives the following asymptotic tensor freeness of UI random matrices.

\begin{corollary} \label{cor-UITensorFree}
Let $(\W_N^{(i)})_{i\in [L]}$ be families of independent $D_N\times D_N$ UI random matrix ensembles. If each family has first order limit, i.e. satisfying the factorization property \cref{eq-condition-1} and convergent in distribution, then $(\W_N^{(i)})_{i\in [L]}$ are asymptotically tensor free.
\end{corollary}

In the later section, we obtain much stronger results by considering partial transposes of UI random matrices, see \cref{thm-indepTranspose}.

We now derive the OI versions of the above results. The arguments are largely similar, except that we use the asymptotic estimate from orthogonal Weingarten calculus (\cref{eq-WeinAsymp2}) and the inequalities involving pair partitions (\cref{lem-pairings}).

\begin{proposition} \label{prop-oi-tensormoment}
Let $\W_N$ be a family of $D_N\times D_N$ OI random matrices.
\begin{enumerate}
    \item Suppose $\W_N$ satisfies the factorization property \cref{eq-condition-Fact} and $\W_N\cup \W_N^{\top}$ satisfies the bounded moments condition \cref{eq-condition-Bdd}.  Then $\W_N$ converges in distribution if and only if $\W_N$ converges in tensor distribution. Furthermore, if $\W$ is the corresponding tensor distribution limit, then the same formulas \cref{eq-ui-tensormoment,eq-ui-tensorcumulant} hold for $\underline{\alpha}\in (S_p)^r$ and $x_1,\ldots, x_p\in \W$.

    \item Suppose $\W_N \cup \W_N^{\top}$ satisfies the factorization property. Then $\W_N \cup \W_N^{\top}$ converges in distribution if and only if $\W_N \cup \W_N^{\top}$ converges in tensor distribution. Furthermore, if $\widetilde{\W}$ is the corresponding tensor limit distribution, then the same formulas \cref{eq-ui-tensormoment,eq-ui-tensorcumulant} hold for $\underline{\alpha}\in (S_p)^r$ and $x_1,\ldots, x_p\in \widetilde{W}$.
\end{enumerate}
\end{proposition}
\begin{proof}
We may proof only the assertion (1); the assertion (2) follows from (1) for families $\widetilde{W}_N=\W_N\cup \W_N^{\top}$. For (1), we repeat the arguments in \cref{prop-ui-tensormoment} using \cref{eq-OITensorInv}. Then we have by \cref{lem-PairingSup,lem-pairings},
\small
\begin{align*}
    &\E[\tr_{\underline{\alpha}}(X_1,\ldots, X_p)]\\
    &= \sum_{\pi,\rho\in \mathcal{P}_2(\pm p)} \E\left[\tr_{\rho\vee\delta}(X_1,\ldots, X_p)\right]  \left(\prod_{s=1}^r d_s^{\#(\rho\vee\delta)-\#\alpha_s}\,d_s^{\#(\pi\vee \alpha_s\delta\alpha^{-1})}\right) D_N^{-p-|\rho\pi|/2}(\Mob(\rho\vee\pi)+o(1))\\
    &= \sum_{\pi,\rho\in \mathcal{P}_2(\pm p)} \E\left[\tr_{\rho\vee\delta}(X_1,\ldots, X_p)\right] \left(\prod_{s=1}^r d_s^{- \frac{1}{2} \big(|\rho\delta|+|(\rho\delta)^{-1}(\pi\delta)|+|(\pi\delta)^{-1}(\alpha_s\delta\alpha_s^{-1}\delta)|-|\alpha_s\delta\alpha_s^{-1}\delta|\big)}\right) (\Mob(\rho\vee\pi)+o(1))
\end{align*}
\normalsize
for $X_1,\ldots, X_p\in \W_N$ and $\underline{\alpha}\in (S_p)^r$. Note that the bounded moments property and \cref{eq-OITraceInvTransp} says that $\sup_N\big|\E[\tr_{\rho\vee\delta}(X_1,\ldots, X_p)]\big|<\infty$ for every $\rho\in \mathcal{P}_2(\pm p)$. Furthermore, the exponents of $d_s$ is always non-positive and attains zero if and only if $\rho\delta \leq \pi\delta \leq \alpha_s\delta\alpha_s^{-1}\delta$ for all $s$. In this case, we can uniquely take $\tau,\sigma\in S_p$ such that $\pi=\sigma\delta\sigma^{-1}$, $\rho=\tau\delta\tau^{-1}$, $\sigma\in \bigcap_{s=1}^r S_{NC}(\alpha_s)$, and $\tau\in S_{NC}(\sigma)$ by \cref{lem-pairings} (2) and (3).
Consequently, applying \cref{eq-OITraceInvTransp,eq-MobPairing} gives that.
\begin{align*}
    \lim_{N\to \infty}\E[\tr_{\underline{\alpha}}(X_1,\ldots, X_p)]
    &= \sum_{\tau,\sigma\in S_p}\left(\prod_{s=1}^r\mathds{1}_{\tau \leq \sigma \leq \alpha_s}\right) \varphi_{\tau}(x_1,\ldots, x_p)\Mob(\tau^{-1}\sigma) \\
    &=\sum_{\sigma\in \bigcap_{s=1}^r S_{NC}(\alpha_s)} \tilde{\kappa}_{\sigma}(x_1,\ldots, x_p).
\end{align*}
as in the proof of \cref{prop-ui-tensormoment}.
\end{proof}

Note that the bounded moments property of $\W_N\cup \W_N^{\top}$ in \cref{prop-oi-tensormoment} is essential: if $X_d=2^d (E_{12}\otimes I_d)\in \M{d}^{\otimes 2}$, and if $O$ is a $d^2\times d^2$ Haar orthogonal matrix, then $OX_dO^{\top}$ satisfies the factorization property and converges in distribution to $0$ as $d\to \infty$. However, \cref{eq-OITensorInv,eq-OWg-dim2} and the facts $\Tr(X_d)=\Tr(X_d^2)=0$ imply that
\begin{align*}
    \E\big[\tr_{\substack{(1\,2)\\{(1)(2)}}}(O X_d O^{\top})\big]&=d^{-3}\sum_{\pi, \rho \in \mathcal{P}_2(\pm 2)} \Tr_{\rho\vee \delta}(X_d) d^{\#(\pi\vee\delta)+\#(\pi\vee\gamma_2\delta\gamma_2)}\Wg_{d^2}^{(O)}(\pi,\rho)\\
    &= d^{-3} \Tr(X_d X_d^{\top})\Big(2 d^3 \cdot \frac{-1}{d^2(d^2+2)(d^2-1)}+ d^2\cdot \frac{d^2+1}{d^2(d^2+2)(d^2-1)}\Big)\\
    &=2^{2d}d^{-2} \cdot \frac{d-1}{(d+1)(d^2+2)},
\end{align*}
which is not convergent as $d\to \infty$.

Now repeating the same reasoning, we obtain the analogues of \cref{thm-ui-tensorfree,cor-UITensorFree}.

\begin{theorem} \label{thm-oi-tensorfree}
Let $(\W_N^{(i)})_{i\in [L]}$ be families of $D_N\times D_N$ random matrices such that the family $\widetilde{\W}_N=\bigcup_{i\in [L]} \W_N^{(i)}$ is OI and  satisfies the factorization property \cref{eq-condition-Fact}. If $\widetilde{\W}_N\cup(\widetilde{\W}_N)^{\top}$ further satisfies the bounded moments condition \cref{eq-condition-Bdd}, then $(\W_N^{(i)})_{i\in [L]}$ are asymptotically free if and only if asymptotically tensor free.
\end{theorem}

\begin{corollary}
\label{cor-OITensorFree}
Let $(\W_N^{(i)})_{i\in [L]}$ be families of independent $D_N\times D_N$ OI random matrix ensembles where each family $\W_N^{(i)}$ satisfies the factorization property and convergent in distribution, and $\W_N^{(i)}\cup (\W_N^{(i)})^{\top}$ satisfies the boundedness condition \cref{eq-condition-Bdd}. Then $(\W_N^{(i)})_{i\in [L]}$ are asymptotically tensor free.
\end{corollary}

\begin{remark} \label{rmk-UITensorAlmostSure}
One can obtain the almost sure convergence of tensor moments under stronger assumptions. For example, suppose $d_{s,N}\equiv N$ and $\W_N$ is a UI family of $N^r\times N^r$ random matrices which converges in distribution and  satisfies the stronger factorization property
    $$\E[\tr_{\sigma \sqcup \tau}(X_1,\ldots, X_p,Y_1,\ldots, Y_q)]=\E[\tr_{\sigma}(X_1,\ldots, X_p)]\E[\tr_{\tau}(Y_1,\ldots, Y_q)]+O(N^{-2})$$
for $\sigma\in S_p$, $\tau\in S_q$, and  $X_1,\ldots, X_p,Y_1,\ldots, Y_q\in \W_N$,
which is the case when $\W_N$ satisfies the bounded cumulants property \cref{eq-condition-BCP}. Then one has the stronger tensor factorization property:
\begin{equation} \label{eq-StrTensorFact}
    \E[\tr_{\underline{\alpha}\sqcup \underline{\beta}}(X_1,\ldots, X_p,Y_1,\ldots, Y_q)]=\E[\tr_{\underline{\alpha}}(X_1,\ldots, X_p)]\E[\tr_{\underline{\beta}}(Y_1,\ldots, Y_q)]+O(N^{-2})
\end{equation}
for $\underline{\alpha}\in (S_p)^r$ and $\underline{\beta}\in (S_q)^r$, by slightly modifying the proof of \cref{prop-ui-tensormoment}. In particular, if $\W_N$ is closed under adjoint $*$, then almost surely, every tensor moment $\tr_{\underline{\alpha}}(\underline{X})$ converges to the same limit $\lim_N \E[\tr_{\underline{\alpha}}(\underline{X})]$ as $N\to \infty$. The same arguments apply to OI random matrices in \cref{prop-oi-tensormoment}.
\end{remark}

\section{Asymptotic tensor freeness of independent locally invariant random matrices} \label{sec-LocUITensorFree}

In this section, we extend the freeness theorems of UI and OI matrices (\cref{thm-UIasympfree,thm-OIAsympFree}) to obtain asymptotic tensor freeness of independent families of local-unitary / orthogonal invariant random matrices.
As in \cref{sec-UItensorfree}, we take $D=D_N:=d_{1,N}\cdots d_{r,N}$ and consider the regime $d_s=d_{s,N}\to \infty$ for all $s=1,\ldots, r$ as $N\to \infty$. Furthermore, let us denote by $\Wg^{(U)}(\underline{\tau},\underline{\sigma}):=\prod_{s=1}^r \Wg_{d_s}^{(U)}(\tau_s^{-1}\sigma_s)$ for $\underline{\sigma},\underline{\tau}\in S_p^r$ and $\Wg^{(O)}(\underline{\rho},\underline{\pi}):=\prod_{s=1}^r \Wg_{d_s}^{(O)}(\rho_s,\pi_s)$ for $\underline{\pi},\underline{\rho}\in \mathcal{P}_2(\pm p)^r$ for simplicity.

\begin{lemma} \label{lem-LocInvTensorMoments}
Let $\W^{(1)},\ldots, \W^{(L)}$ be independent families of $D\times D$ random matrices, and suppose $X_j\in \W^{(f(j))}$ for a function $f:[p]\to [L]$.
\begin{enumerate}
    \item If each $\W^{(i)}$ is LUI, then for $\underline{\alpha} \in (S_p)^r$, we have
    \small
    \begin{align} 
        &\E[\Tr_{\underline{\alpha}}(X_1,\ldots, X_p)] \nonumber \\ 
        &=\sum_{\substack{\underline{\sigma}^{(i)}, \underline{\tau}^{(i)} \in S(f^{-1}(i))^r\\\forall \, i\in f([p])}} \E\bigg[\prod_{i\in f([p])}\Tr_{\underline{\tau}^{(i)}}((X_j)_{j\in f^{-1}(i)})\bigg] \bigg( \prod_{s=1}^r d_s^{\#(\alpha_s^{-1}(\sqcup_i \sigma_s^{(i)}))} \bigg)\prod_{i\in f([p])} {\rm Wg}^{(U)}({\underline{\tau}^{(i)}},\underline{\sigma}^{(i)}). \label{eq-LocUITensorMoments}
    \end{align}
    \normalsize

    \item If each $\W^{(i)}$ is LOI, then for $\underline{\alpha} \in (S_p)^r$, we have
    \small
    \begin{align} 
        &\E[\Tr_{\underline{\alpha}}(X_1,\ldots, X_p)] \nonumber \\ 
        &=\sum_{\substack{\underline{\pi}^{(i)}, \underline{\rho}^{(i)} \in \mathcal{P}_2(\pm f^{-1}(i))^r\\\forall \, i\in f([p])}} \E\bigg[\prod_{i\in f([p])}\Tr_{\underline{\rho}\vee\delta^{(i)}}((X_j)_{j\in f^{-1}(i)})\bigg] \bigg(\prod_{s=1}^r d_s^{\#(\alpha_s\delta\alpha_s^{-1}\vee (\sqcup_i \pi_s^{(i)}))}\bigg)\prod_{i\in f([p])} {\rm Wg}^{(O)}({\underline{\rho}^{(i)}}, \underline{\pi}^{(i)}), \label{eq-LocOITensorMoments}
    \end{align}
    \normalsize
    where $\pm f^{-1}(i):=f^{-1}(i)\sqcup \{-j:j\in f^{-1}(i)\}$, $\delta^{(i)}:=\delta\big|_{\pm f^{-1}(i)}=\prod_{j\in f^{-1}(i)}(j, -j)$.
\end{enumerate}
\end{lemma}
\begin{proof}
We may assume that $f$ is surjective, i.e.~$f([p])=[L]$, by discarding unused variables. Let $\left(U_i=\bigotimes_{s=1}^r U_s^{(i)}\right)_{i\in [L]}$ be an independent family of the Haar random local unitaries, i.e.~$\{U_s^{(i)}\}_{(i,s) \in [L] \times [r]}$ are independent, Haar-distributed random unitary matrices. Since $(X_1,\ldots, X_p)$ has the same joint probability law as $(U_{f(1)}X_1 U_{f(1)}^*,\ldots, U_{f(p)} X_p U_{f(p)}^*)$, we have
\begin{align*}
    \E[\Tr_{\underline{\alpha}}(X_1,\ldots, X_p)]&= \E[\Tr_{\underline{\alpha}}(U_{f(1)}X_1U_{f(1)}^*, \ldots, U_{f(p)}X_p U_{f(p)}^*)]\\
    &=\E_{\bigcup_i \W^{(i)}} \Big[\E_{U_1,\ldots, U_L}\big[\Tr_{\underline{\alpha}}(U_{f(1)}X_1U_{f(1)}^*, \ldots, U_{f(p)}X_p U_{f(p)}^*)\big| {\bigcup}_{i} \W^{(i)}\big] \Big].
\end{align*}
Now as in the proof of \cref{eq-UITensorInv}, we can apply unitary Weingarten calculi for $(U_s^{(i)})_{s\in [r],i\in [L]}$ to obtain
\begin{align*} 
    &\E_{U_1,\ldots, U_L}\big[\Tr_{\underline{\alpha}}(U_{f(1)}X_1U_{f(1)}^*, \ldots, U_{f(p)}X_p U_{f(p)}^*)\big| {\bigcup}_{i} \W^{(i)}\big] \\ 
    &=\sum_{\substack{\underline{\sigma}^{(i)}, \underline{\tau}^{(i)} \in S(f^{-1}(i))^r\\\forall \, i\in [L]}} \prod_{i=1}^L \Tr_{\underline{\tau}^{(i)}}((X_j)_{j\in f^{-1}(i)}) \bigg( \prod_{s=1}^r d_s^{\#(\alpha_s^{-1}(\sqcup_i \sigma_s^{(i)}))} \bigg)\prod_{i=1}^L {\rm Wg}^{(U)}({\underline{\tau}^{(i)}},\underline{\sigma}^{(i)}),
\end{align*}
see \cref{fig:Wg-local-unitarily-invariant}.
Therefore, \cref{eq-LocUITensorMoments} is obtained. The formula \cref{eq-LocOITensorMoments} follows verbatim.
\end{proof}

\begin{figure}[!htb]
    \centering
    \includegraphics[width=0.9\linewidth]{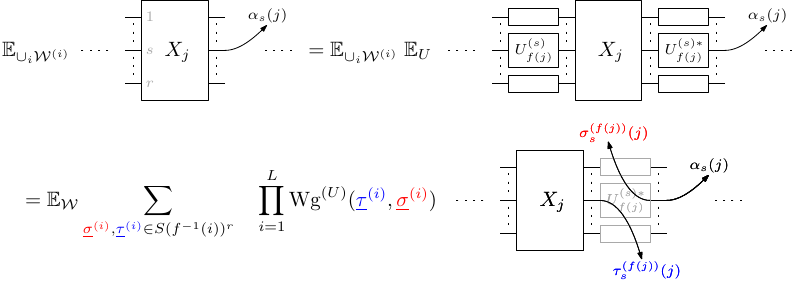}
    \caption{Using the local unitary invariance of the families $\mathcal W^{(i)}$, one can replace each variable $X_j \in \mathcal W^{(f(j))}$ by $U_{f(j)} X_j U_{f(j)}^*$, where the $U_i$'s are tensor products of independent Haar-distributed random unitary matrices that are also independent from the $X$'s. Performing the Weingarten integration with respect to all the Haar unitary matrices results in a sum over diagrams indexed by several tuples of permutations.}
    \label{fig:Wg-local-unitarily-invariant}
\end{figure}

The main theorem of this section is following.

\begin{theorem} \label{thm-locui-tensorfree}
Let $\W_N^{(1)},\ldots ,\W_N^{(L)}$ be independent families of $D_N\times D_N$ random matrices such that:
\begin{enumerate}
    \item Each family, except possibly one, is LUI.

    \item Each family {converges in tensor distribution}, i.e., the limit
        $$\lim_{N\to \infty} \E\left[\tr_{\underline{\alpha}}(X_1^{(i)},\ldots, X_p^{(i)})\right]=\lim_{N\to \infty} \frac{1}{d_1^{\#\alpha_1}\cdots d_r^{\#\alpha_r}}\E\left[\Tr_{\underline{\alpha}}(X_1^{(i)},\ldots, X_p^{(i)})\right]$$
    exists for all $i=1,\ldots, L$, $X_j^{(i)}\in \W_N^{(i)}$, and $\underline{\alpha}\in (S_p)^r$.

    \item Each family satisfies the tensor factorization property \cref{eq-condition-TensorFact}.
\end{enumerate}
Then $\W_N^{(1)},\ldots \W_N^{(L)}$ are {asymptotically tensor freely independent} as $N\to \infty$.
In particular, the family $\bigcup_{i=1}^L \W_N^{(i)}$ satisfies the tensor factorization property \cref{eq-condition-TensorFact}.
\end{theorem}

\begin{proof}
We may assume that every family $\W_N^{(i)}$ is LUI. Indeed, $\bigcup_{i}\W_N^{(i)}$ has the same joint tensor distribution with $\bigcup_{i}(U\W_N^{(i)}U^*)$ where $U=\bigotimes_{s=1}^r U_s$ is a Haar random local-unitary matrix independent from $\bigcup_{i}\W_N^{(i)}$, thanks to the local-unitary invariance of $\tr_{\underline{\alpha}}$'s and the local-unitary analogue of \cref{prop-IndepUI}. For $\underline{\alpha}\in (S_p)^r$, $f:[p]\to [L]$, and $X_j\in \W_N^{(f(j))}$ for $j\in [p]$, suppose that $(X_j)_{j\in f^{-1}(i)}\to (x_j)_{j\in f^{-1}(i)}$ in tensor distribution for each $i\in f([p])$. Then we can apply \cref{eq-LocUITensorMoments} and the estimate of M\"{o}bius function (\cref{eq-WeinAsymp1}) to have
\small
\begin{align*}
    &\E[\tr_{\underline{\alpha}}(X_1,\ldots, X_p)]\\
    &= \sum_{\sigma_s^{(i)}, \tau_s^{(i)}\in S(f^{-1}(i))} \E\bigg[\prod_{i\in f([p])}\tr_{\underline{\tau}^{(i)}}((X_j)_{j\in f^{-1}(i)})\bigg] \prod_{s=1}^r \big(d_s^{\#(\alpha_s^{-1}(\sqcup_i \sigma_s^{(i)}))+\#(\sqcup_i \tau_s^{(i)})-\#\alpha_s}\big)\\
    &\qquad\qquad\qquad\qquad\qquad\qquad\qquad \times \prod_{i\in f([p])} \Big(\prod_{s=1}^r d_s^{-|f^{-1}(i)|-|{\tau_s^{(i)}}^{-1}\sigma_s^{(i)}|} \Big) \Big(\Mob\big({{\underline{\tau}^{(i)}}}^{-1}\underline{\sigma}^{(i)}\big)+ o(1)\Big)\\
    &= \sum_{\sigma_s^{(i)}, \tau_s^{(i)}\in S(f^{-1}(i))} \bigg(\prod_{i\in f([p])}\E\left[\tr_{\underline{\tau}^{(i)}}((X_j)_{j\in f^{-1}(i)})\right]\bigg) \\
    &\qquad\qquad\qquad\qquad \times \prod_{s=1}^r d_s^{-(|\alpha_s^{-1}(\sqcup_i \sigma_s^{(i)})|+|\sqcup_i ({\tau_s^{(i)}}^{-1}\sigma_s^{(i)})|+|\sqcup_i \tau_s^{(i)}|-|\alpha_s|)} \bigg(\prod_{i\in f([p])} \Mob\big({{\underline{\tau}^{(i)}}}^{-1}\underline{\sigma}^{(i)}\big)+o(1) \bigg)\\
    & \to \sum_{\substack{\sigma_s^{(i)}\in S(f^{-1}(i)) \\ \sqcup_i \underline{\sigma}^{(i)}\in S_{NC}(\underline{\alpha})}} \prod_{i} \sum_{\underline{\tau}^{(i)}\in S_{NC}(\underline{\sigma}^{(i)})}  \varphi_{\underline{\tau}^{(i)}}((x_j)_{j\in f^{-1}(i)}) \Mob\big({{\underline{\tau}^{(i)}}}^{-1}\underline{\sigma}^{(i)}\big) \\
    &= \sum_{\substack{\sigma_s^{(i)}\in S(f^{-1}(i)) \\ \sqcup_i \underline{\sigma}^{(i)}\in S_{NC}(\underline{\alpha})}} \prod_i \kappa_{\underline{\sigma}^{(i)}}((x_j)_{f^{-1}(i)})=\sum_{\substack{\underline{\beta}\in S_{NC}(\underline{\alpha}) \\ \bigvee_{s=1}^r\Pi(\beta_s)\leq {\ker}f}}\prod_{i\in f([p])} \kappa_{\underline{\beta}\big|_{f^{-1}(i)}}((x_j)_{j\in f^{-1}(i)}),
\end{align*}
\normalsize
Here we need the independence of $\W_N^{(i)}$'s and their factorization property to guarantee that
    $$\E\bigg[\prod_{i\in f([p])}\tr_{\underline{\tau}^{(i)}}((X_j)_{j\in f^{-1}(i)})\bigg]=\prod_{i\in f([p])}\E\left[\tr_{\underline{\tau}^{(i)}}((X_j)_{j\in f^{-1}(i)})\right]\to \prod_{i}\varphi_{\underline{\tau}^{(i)}}((x_j)_{j\in f^{-1}(i)})$$
for arbitrary $(\underline{\tau}^{(i)})_i$ (note that each $\underline{\tau}^{(i)}$ does not connect any two matrices from different families), and thus all the terms, except for the case $\sqcup_i\,\underline{\tau}^{(i)} \leq \sqcup_i\, \underline{\sigma}^{(i)} \leq  \underline{\alpha}$, converge to $0$ as $N\to \infty$. Moreover, we can identify $\beta_s=\sqcup_i \sigma_s^{(i)}$ in the last equality. This induces the conclusion by \cref{cor-TensorFreeSubsetMoment}.
\end{proof}

\begin{remark}
\cref{thm-locui-tensorfree} indeed generalizes \cref{thm-UIasympfree} since, for unitary invariant random matrices, the asymptotic freeness becomes the same notion with asymptotic tensor freeness (\cref{thm-ui-tensorfree}).
\end{remark}

If $U=\bigotimes_{s=1}^r U_s$ is a Haar random local-unitary matrix, then it is straightforward to show that $\{U,U^*\}$ converges in tensor distribution to $\{u_{\otimes}, u_{\otimes}^*\}$, where $u_{\otimes}=u^{\otimes r}$ is a tensor Haar unitary described in \cref{ex-TensorHaarUnitaries}. Since $\{U,U^*\}$ is LUI, \cref{thm-locui-tensorfree} directly gives the following tensor freeness result.

\begin{corollary} \label{cor-TensorHaarTensorfree}
Suppose $\W_N$ is a family of $D_N\times D_N$ random matrices which satisfies the tensor factorization property \cref{eq-condition-TensorFact} and converges in tensor distribution. Let $U_1,\ldots, U_L$ are independent $D_N\times D_N$ Haar random local-unitaries independent from $\W_N$. Then $\W_N, \{U_1,U_1^*\}, \ldots, \{U_L,U_L^*\}$ jointly converges in tensor distribution to tensor free families $\W,\{u_{\otimes}^{(1)},{u_{\otimes}^{(1)\,*}}\}, \ldots, \{u_{\otimes}^{(L)},{u_{\otimes}^{(L)\,*}}\}$ where $\W_N \xrightarrow{\otimes\text{-distr}} \W$ and each $u_{\otimes}^{(i)}$ is a tensor product Haar unitary element.
\end{corollary}

\begin{remark} \label{rmk-LocUIAlmostSure}
Let $d_{s,N}\equiv N$ and suppose that $\W_N=(X_{j,N})_{j\in J}$ is closed under the adjoint $X\mapsto X^*$ and  $\W_N\xrightarrow{\otimes\text{-distr}} \W=(x_j)_{j\in J}$ \emph{almost surely}, i.e., for every $\underline{\alpha}\in (S_p)^r$ and $j_1,\ldots, j_p\in J$, we have
    $$\tr_{\underline{\alpha}}(X_{j_1,N},\ldots, X_{j_p,N})\to \varphi_{\underline{\alpha}}(x_{j_1},\ldots, x_{j_p})$$
almost surely as $N\to \infty$. Note that here we impose no conditions on expectation values. Then almost surely as $N\to \infty$,
    $$\W_N, \{U_1,U_1^*\}, \ldots, \{U_L,U_L^*\} \xrightarrow{\otimes\text{-distr}} \W,\{u_{\otimes}^{(1)},{u_{\otimes}^{(1)\,*}}\}, \ldots, \{u_{\otimes}^{(L)},{u_{\otimes}^{(L)\,*}}\}.$$
Indeed, we may assume that $\W_N$ is deterministic after conditioning on the event $\{\om:\W_N(\om)\xrightarrow{\otimes\text{-distr}}\W\}$ (we refer to \cite[Sec 3]{Mal12} for the formal argument). Then the proof of \cref{thm-locui-tensorfree} shows that the family $\widetilde{\W}_N= \W_N\cup \Big(\bigcup_{i=1}^L \{U_i,U_i^*\}\Big)$ satisfies the condition \cref{eq-StrTensorFact}, thanks to the bounded cumulants property of Haar random unitary matrices. Therefore, every joint tensor moment of $\widetilde{\W}_N$ converges almost surely.

\end{remark}

In order to obtain the LOI analogue of \cref{thm-locui-tensorfree}, we further need conditions on partial transposes of random matrices. For $D=d_1\cdots d_r$ and a $D\times D$ matrix $X\in \M{d_1}\otimes\cdots \otimes \M{d_r}$, let us denote its \textit{partial transposes} by
\begin{equation} \label{eq-PartialTransp}
    X^{\underline{t}}:=(\Ec_{t_1}\otimes \cdots \otimes \Ec_{t_r})(X)
\end{equation}
for $\underline{t}=(t_1,\ldots, t_r)\in \{\pm 1\}^r$, where $\Ec_1=\id$ and $\Ec_{-1}=\top$.  For example, we have $X^{(1,\ldots, 1)}=X$ and $X^{(-1,\ldots, -1)}=X^{\top}$. Note that we abuse the notation by using the same symbols $\id$ and  $\top$ to denote the identity and the transpose maps, respectively, on the matrix algebras even when they are of different sizes. Moreover, for a family $\W$ of $D\times D$ matrices, let us denote by $\W^{\underline{t}}:=\{X^{\underline{t}}:X\in \W\}$ for $\underline{t}\in \{\pm 1\}^r$, and let us simply denote by $\W^{\top}:=\W^{(-1,\ldots, -1)}$ as before. 

\begin{theorem} \label{thm-LocOI-tensorfree}
Let $\W_N^{(1)},\ldots ,\W_N^{(L)}$ be independent families of $D_N\times D_N$ random matrices such that each family (except possibly one) is LOI.

\begin{enumerate}
    \item If each family $\W_N^{(i)}$ satisfies the tensor factorization property \cref{eq-condition-TensorFact}, converges in tensor distribution, and if the family of its partial transposes $\bigcup_{\underline{t}\in \{\pm 1\}^r}(\W_N^{(i)})^{\underline{t}}$ 
    satisfies the bounded tensor moment condition \cref{eq-condition-TensorBdd},
    then $\W_N^{(1)},\ldots \W_N^{(L)}$ are asymptotically tensor free as $N\to \infty$.
    
    \item If each family $\bigcup_{\underline{t}\in \{\pm 1\}^r}(\W_N^{(i)})^{\underline{t}}$ satisfies the tensor factorization property \cref{eq-condition-TensorFact} and {converges in tensor distribution}, then $\bigcup_{\underline{t}\in \{\pm 1\}^r}(\W_N^{(1)})^{\underline{t}},\ldots , \bigcup_{\underline{t}\in \{\pm 1\}^r}(\W_N^{(L)})^{\underline{t}}$ are asymptotically tensor free as $N\to \infty$.
\end{enumerate}

\end{theorem}

\begin{proof}
Similarly to \cref{prop-oi-tensormoment}, we may prove only the assertion (1). Moreover, similarly as in \cref{thm-locui-tensorfree}, we may assume that every family is LOI. Then we can apply \cref{eq-LocOITensorMoments} and the estimate of orthogonal Weingarten function (\cref{eq-WeinAsymp2}) to have
\begin{align*}
    &\E[\tr_{\underline{\alpha}}(X_1,\ldots, X_p)]\\
    &= \sum_{\pi_s^{(i)}, \rho_s^{(i)} \in \mathcal{P}_2(\pm f^{-1}(i))} \bigg(\prod_{i\in f([p])}\E\left[\tr_{\underline{\rho}\vee\delta^{(i)}}((X_j)_{j\in f^{-1}(i)})\right]\bigg)\, \Big( \prod_{i\in f([p])} \Mob\big({{\underline{\rho}^{(i)}}}\vee \underline{\pi}^{(i)}\big)+ o(1)\Big) \\
    &\qquad\qquad\qquad\qquad \times \prod_{s=1}^r \Big(d_s^{\#(\delta\vee(\sqcup_i \rho_s^{(i)}))-\#\alpha_s}\, d_s^{\#(\alpha_s\delta\alpha_s^{-1}\vee (\sqcup_i \pi_s^{(i)}))}\, \prod_{i\in f([p])}  d_s^{-|f^{-1}(i)|-|{\rho_s^{(i)}} \pi_s^{(i)}|/2} \Big).
\end{align*}
\normalsize
After simply writing $\pi_s:=\sqcup_i \pi_s^{(i)}$ and $\rho_s:=\sqcup_i \rho_s^{(i)}$, and applying \cref{lem-PairingSup,lem-pairings}, the whole exponent of $d_s$ becomes
\begin{align*}
    & \#(\delta \vee \rho_s)-\#\alpha + \#(\alpha_s\delta\alpha_s^{-1}\vee \pi_s)-p-|\rho_s\pi_s|/2\\
    &=\frac{1}{2}(2p-|\rho_s\delta|)-(p-|\alpha_s|)+\frac{1}{2}(2p-|\alpha_s\delta\alpha_s^{-1}\delta)\cdot (\pi_s\delta)^{-1}|-p-\frac{1}{2}|(\pi_s\delta)\cdot (\rho_s\delta)^{-1}|\\
    &=-\frac{1}{2}\Big(|(\alpha_s\delta\alpha_s^{-1}\delta)\cdot (\pi_s\delta)^{-1}|+ |(\pi_s\delta)\cdot (\rho_s\delta)^{-1}|+ |\rho_s\delta|-|\alpha_s\delta\alpha_s^{-1}\delta| \Big) .
\end{align*}
In particular, the above is always non-positive and attains zero if and only if there exist $\sigma_s=\bigsqcup_{i}\sigma_s^{(i)}, \tau_s=\bigsqcup_{i}\tau_s^{(i)} \in S_p$ such that $\pi_s\delta=\sigma_s\delta\sigma_s^{-1}\delta$, $\rho_s\delta=\tau_s\delta\tau_s^{-1}\delta$, and $\tau_s \leq \sigma_s \leq \alpha_s$, by \cref{lem-pairings}. In this case, we also have $\tr_{\underline{\rho}\vee\delta^{(i)}}=\tr_{\underline{\tau}^{(i)}}$ and $\Mob(\rho_s\vee \pi_s)=\Mob(\tau_s^{-1}\sigma_s)$. Consequently, we have
\begin{align*}
    \lim_{N\to \infty} \E \left[\tr_{\underline{\alpha}}(X_1,\ldots, X_p)\right]&=\sum_{\substack{\sigma_s^{(i)}\in S(f^{-1}(i)) \\ \sqcup_i \sigma_s^{(i)} \leq \alpha_s}} \prod_{i} \sum_{\tau_s^{(i)}\in S_{NC}(\sigma_s^{(i)})}  \varphi_{\underline{\tau}^{(i)}}((x_j)_{j\in f^{-1}(i)}) \Mob\big({{\underline{\tau}^{(i)}}}^{-1}\underline{\sigma}^{(i)}\big)\\
    &= \sum_{\substack{\sigma_s^{(i)}\in S(f^{-1}(i)) \\ \sqcup_i \sigma_s^{(i)} \leq \alpha_s}} \prod_i \kappa_{\underline{\sigma}^{(i)}}((x_j)_{f^{-1}(i)}),
\end{align*}
similarly as in the proof of \cref{thm-locui-tensorfree}. Note that we need the boundedness condition for whole faimily $\bigcup_{\underline{t}}(\W_N^{(i)})^{\underline{t}}$ to guarantee that all the terms in the sum except for the case
    $${\rho_s\delta} \leq {\pi_s \delta} \leq {\alpha_s \delta\alpha_s^{-1}\delta} \quad \forall\, s\in [r]$$
converges to $0$ as $N\to \infty$.
\end{proof}

In the following section, we will show that local-unitary invariant random matrices having limit tensor distribution automatically have a joint limit tensor distribution between all partial transposes (\cref{thm-LocUITranspose}). Therefore, we can say that \cref{thm-LocOI-tensorfree} generalizes \cref{thm-locui-tensorfree}.

Let $O=\bigotimes_{s=1}^r O_s$ be a Haar random local-orthogonal matrix. Since $\{O_s,O_s^{\top}\}\to \{u,u^*\}=\{u,u^{-1}\}$ jointly in distribution for each $s\in [r]$ where $u$ is a Haar unitary element, one has that $\{O^{\underline{t}}\}_{\underline{t}\in \{\pm 1\}^r}\to \{u_{\otimes}^{\underline{t}}\}_{\underline{t}\in \{\pm 1\}^r}$ in tensor distribution, where $u_{\otimes}^{\underline{t}}:=u^{t_1}\otimes \cdots \otimes u^{t_r}$ (we also refer to \cref{prop-LuiTranspose} for tensor distribution of each $u^{\underline{t}}_{\otimes}$). Similarly as in \cref{cor-TensorHaarTensorfree,rmk-LocUIAlmostSure}, we obtain the following tensor freeness result for Haar random local-orthogonal matrices.

\begin{corollary}
Suppose $\W_N$ is a family of $D_N\times D_N$ random matrices which satisfies the tensor factorization property \cref{eq-condition-TensorFact} and converges in tensor distribution, and suppose that $\bigcup_{\underline{t}\in \{\pm 1\}^r}\W_N^{\underline{t}}$ satisfies the bounded tensor moments property \cref{eq-condition-TensorBdd}. Then the $(L+1)$ families
\begin{equation} \label{eq-TensorHaarOrthogonal}
    \W_N, \{O_1^{\underline{t}}:\underline{t}\in \{\pm 1\}^r\}, \ldots, \{O_L^{\underline{t}}:\underline{t}\in \{\pm 1\}^r\}
\end{equation}
are asymptotically tensor free as $N\to\infty$, where $O_1,\ldots, O_L$ are independent Haar random local-orthogonal matrices. Furthermore, if $\W_N$ converges in tensor distribution almost surely, then the same convergence in tensor distribution of \cref{eq-TensorHaarOrthogonal} holds almost surely.
\end{corollary}

\section{Asymptotic tensor freeness of non-independent random matrices} \label{sec-TensorFreeNonIndep}

One natural question from \cref{thm-locui-tensorfree} is whether it is possible to find examples of random matrices that are \textit{not independent} but still exhibit asymptotic tensor free independence. Actually, we can give an affirmative answer using a rather ``trivial example'': if $X_N$ is a $N\times N$ random matrix which satisfies the factorization property and converges in distribution to $x\in (\A,\varphi)$, then we have
    $$X_N\otimes I_N,\; I_N\otimes X_N \xrightarrow{\otimes\text{-distr}} x\otimes 1_{\A},\; 1_{\A}\otimes x\in \big(\A^{\otimes 2},\varphi^{\otimes 2}, (\varphi_{\alpha_1}\otimes \varphi_{\alpha_2})\big),$$
where $x\otimes 1_{\A}$ and $1_{\A}\otimes x$ are tensor freely independent (and also classically independent) by \cref{prop:tensor-free-tensor-product}.

In this section, we additionally provide ``non-trivial'' examples in two aspects where asymptotic freeness has been observed in literature. In \cref{sec-PTTensorfree}, we show that \emph{partial transposes} of random matrices become tensor free independent in the limit of large dimensions, and in \cref{sec-EmbeddingTensorFree}, we show the asymptotic tensor freeness of \emph{tensor embeddings} of random matrices.

\subsection{Partial transposes of random matrices} \label{sec-PTTensorfree}

The asymptotic behavior of random matrices and their (partial) transposes has been extensively studied over the past decade. In particular, much of the research has focused on \textit{Wishart matrices}, as they naturally model random quantum states \cite{nechita2007asymptotics} and because their partial transposes have a close connection to entanglement in quantum information theory \cite{Per96,HHH96}. The first studies in this direction were conducted by Aubrun \cite{Aub12} and Banica and Nechita \cite{BN13}, who examined the limit distributions of the partial transposes of balanced and unbalanced Wishart matrices, respectively. On the other hand, \cite{MP16} showed that a UI random matrix and its transpose exhibit asymptotic free independence, providing one of the first examples of asymptotic freeness without (probabilistic) independence. This result was further extended in \cite{MP19,MP24}, which showed that for any bipartite UI random matrix $X$, all its partial transposes
    $$X, \quad X^{(1,-1)}=(\id\otimes \top)(X),\quad X^{(-1,1)}=(\top\otimes \id)(X), \quad X^{(-1,-1)}=X^{\top}$$
are asymptotically free (recall \cref{eq-PartialTransp} for our notation on partial transposes). More recently, \cite{PY24} established that any (multipartite) $N^r\times N^r$ Wishart matrix $W$, all its partial transposes $(W^{\underline{t}})_{\underline{t}\in \{\pm 1\}^r}$ are asymptotically free. 

In this section, we give the complete description on the first order (tensor) limit behaviors of (local-) UI / OI random matrices and their partial transposes, benefiting from the ideas of tensor trace invariants and combinatorics of pairings introduced in \cref{sec:preliminary-permutations}. Our results extend all the aforementioned findings and offer further generalizations of the results in \cref{sec-UItensorfree,sec-LocUITensorFree}.

\medskip

Let us begin with the (marginal) tensor distribution limit of partial transposes of random matrices. {We display in \cref{fig:trace-invariant-PT} a trace invariant of a family of three matrices $X_1, X_2,X_3$, where the first matrix is partially transposed on the first tensor index, while the last one is partially transposed on the second tensor index.}

\begin{figure}[htb!]
    \centering
    \includegraphics[width=0.5\linewidth]{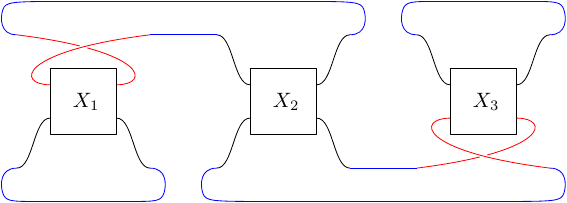}
    \caption{The trace invariant $\Tr_{\textcolor{blue}{(12)(3),(1)(23)}}(X_1^{(\textcolor{red}{-1},1)}, X_2, X_3^{(1,\textcolor{red}{-1})})$ corresponding to the permutations $\textcolor{blue}{\alpha_1 = (12)(3)}$ and $\textcolor{blue}{\alpha_2 = (1)(23)}$, where the matrix $X_1$ is partially transposed on the \textcolor{red}{first index} and the matrix $X_3$ is partially transposed on the \textcolor{red}{second index}.}
    \label{fig:trace-invariant-PT}
\end{figure}

\begin{proposition} \label{prop-LuiTranspose}
Suppose $\W_N$ is a family of $D_N\times D_N$ random matrices which has the tensor factorization property \cref{eq-condition-TensorFact} and converges in tensor distribution to $\W$. Then for every $\underline{t} \in \{\pm 1\}^r$, the family $\W_N^{\underline{t}}$ converges in tensor distribution to $\W^{\underline{t}}:=\{x^{\underline{t}}: x\in \W\}$. Here the tensor distribution of $\W^{\underline{t}}$ is defined as
    $$\varphi_{\underline{\alpha}}(x_1^{\underline{t}}, \ldots, x_p^{\underline{t}}):=\varphi_{\underline{\alpha}^t}(x_1,\ldots, x_p), \quad \underline{\alpha}\in (S_p)^r,$$
where $\underline{\alpha}^t:=(\alpha_1^{t_1},\ldots \alpha_s^{t_s})$. Moreover, we have the corresponding tensor free cumulants
\begin{equation} \label{eq-TranspTensorCumulants}
    \kappa_{\underline{\alpha}}(x_1^{\underline{t}},\ldots x_p^{\underline{t}})=\kappa_{\underline{\alpha}^t}(x_1,\ldots, x_p),\quad \underline{\alpha}\in (S_p)^r.
\end{equation}
\end{proposition}
\begin{proof}
The first assertion simply follows from the observation that, for $X_1,\ldots, X_p\in \W_N$ and $\underline{\alpha}\in (S_p)^r$, we have
    $$\E\left[\tr_{\underline{\alpha}}(X_1^{\underline{t}},\ldots, X_p^{\underline{t}})\right]=\E\left[\tr_{\underline{\alpha}^{t}}(X_1,\ldots, X_p)\right] \to \varphi_{\underline{\alpha}^{t}}(x_1,\ldots, x_p)=\varphi_{\underline{\alpha}}(x_1^{\underline{t}},\ldots, x_p^{\underline{t}})$$
as $N\to \infty$, where $(x_1,\ldots, x_p)$ is the joint tensor distribution limit of $(X_1,\ldots, X_p)$. For the tensor free cumulant, we apply \cref{eq:tensor-free-cumulant-moment} to obtain
\begin{align*}
    \kappa_{\underline{\alpha}}(x_1^{\underline{t}},\ldots, x_p^{\underline{t}}) &=\sum_{\underline{\beta}\,\leq\, \underline{\alpha}}\varphi_{\underline{\beta}}(x_1,\ldots, x_p) \Mob(\underline{\beta}^{-1}\underline{\alpha}) \\
    &=\sum_{\underline{\beta}^t\,\leq\, \underline{\alpha}^t}\varphi_{\underline{\beta}^t}(x_1,\ldots, x_p) \Mob((\underline{\beta}^t)^{-1}\underline{\alpha}^t)\\
    &=\sum_{\underline{\beta}\,\leq\, \underline{\alpha}^t}\varphi_{\underline{\beta}}(x_1,\ldots, x_p) \Mob((\underline{\beta})^{-1}\underline{\alpha}^t) =\kappa_{\underline{\alpha}^t}(x_1,\ldots, x_p).
\end{align*}
Note that we used the observations that $\tau\leq \sigma \iff \tau^{-1}\leq \sigma^{-1}$ and $\Mob(\tau^{-1}\sigma)=\Mob(\tau\sigma^{-1})$.
\end{proof}

The above proposition, combined with \cref{prop-ui-tensormoment,prop-cumulant-from-tensorcumulant}, gives a simple way to recover and generalize the limit distribution of partially transposed random matrices in \cite{Aub12, MP19, MP24}. Recall that a \textit{semicircular element} with mean $\lambda\in \mathbb{R}$ and variance $\sigma^2$ is a self-adjoint element $s_{\lambda,\sigma}$ in a $*$-probability space $(\A,\varphi)$ whose moments are described as
    $$\varphi(s_{\lambda,\sigma}^p)=\int_{\lambda-2\sigma}^{\lambda+2\sigma}t^p \frac{\sqrt{4\sigma^2-(t-\lambda)^2}}{2\pi\sigma}{\rm d}t,\quad p\geq 1.$$
Specifically, $s_{\lambda,\sigma}=\lambda1_{\A}+\sigma s_{0,1}$ where $s_{0,1}$ is a \textit{standard} semicircular element (i.e. having mean $0$ and variance $1$). Another way to characterize semicircular elements is to look at their (usual) free cumulants:
    $$\tilde{\kappa}_1(s_{\lambda,\sigma})=\varphi(s_{\lambda,\sigma})=\lambda, \qquad \tilde{\kappa}_2(s_{\lambda,\sigma})=\varphi(s_{\lambda,\sigma}^2)-\varphi(s_{\lambda,\sigma})^2=\sigma^2, \qquad \tilde{\kappa}_p(s_{\lambda,\sigma})=0 \;\;\forall\, p\geq 3.$$

\begin{corollary} \label{cor-UITranspose}
Suppose $\W_N$ is a family of $D_N\times D_N$ random matrices which is unitary invariant, converges in distribution and has the factorization property \cref{eq-condition-Fact}. Then for every $\underline{t} \in \{\pm 1\}^r$, the family $\W_N^{\underline{t}}$ converges in tensor distribution whose limit has tensor free cumulants
\begin{equation} \label{eq-UITranspCumulant}
    \kappa_{\underline{\alpha}}(x_1^{\underline{t}},\ldots, x_p^{\underline{t}})=\begin{cases}
        \tilde{\kappa}_{\sigma}(x_1,\ldots, x_p) & \text{if $\alpha_s^{t_s}\equiv \sigma$}, \\
        0 & \text{otherwise},
    \end{cases}
    \quad \underline{\alpha}\in (S_p)^r.
\end{equation}
If $\underline{t}\notin \{(1,\ldots, 1),(-1,\ldots, -1)\}$, then their free cumulants can be further described by
\begin{equation} \label{eq-UITranspCumulant2}
    \tilde{\kappa}_{p}(x_1^{\underline{t}},\ldots, x_p^{\underline{t}})=\begin{cases}
        \tilde{\kappa}_{1}(x_1)=\varphi(x_1) & \text{if $p=1$}, \\
        \tilde{\kappa}_{2}(x_1,x_2)=\varphi(x_1x_2)-\varphi(x_1)\varphi(x_2) & \text{if $p=2$}, \\
        0 & \text{otherwise},
    \end{cases}
\end{equation}

In particular, if $X_N$ is a unitary invariant $D_N\times D_N$ Hermitian random matrices satisfying \cref{eq-condition-Fact} and if $X_N$ converges in distribution to $x\in (\A,\varphi)$, then $X_N^{\underline{t}}\to s_{\kappa_1,\kappa_2}$ in distribution whenever $\underline{t}\notin \{(1,\ldots, 1),(-1,\ldots, -1)\}$, where $s_{\kappa_1,\kappa_2}$ is a semicircular element with mean $\kappa_1=\tilde{\kappa}_1(x)=\varphi(x)$ and variance $\kappa_2=\tilde{\kappa}_2(x)=\varphi(x^2)-\varphi(x)^2$.
\end{corollary} 
\begin{proof}
The first assertion is a direct consequence of \cref{prop-LuiTranspose,prop-ui-tensormoment}. For the second assertion, \cref{eq-ui-tensormoment} with $\alpha_s \equiv \gamma_p^{t_s}$ imply that
    $$\varphi(x_1^{\underline{t}}\cdots x_p^{\underline{t}}) = \varphi_{\underline{\gamma_p}^{t}}(x_1, \ldots, x_p)= \sum_{\sigma\in S_{NC}(\gamma_p)\cap S_{NC}(\gamma_p^{-1})} \tilde{\kappa}_{\sigma}(x_1,\ldots, x_p)=\sum_{\sigma\in NC_{1,2}(p)}\tilde{\kappa}_{\sigma}(x_1,\ldots, x_p).$$
Note that the condition $\underline{t}\notin \{(1,\ldots, 1),(-1,\ldots, -1)\}$ is needed in the second equality, and the correspondence $S_{NC}(\gamma_p)\cap S_{NC}(\gamma_p^{-1})\cong NC_{1,2}(p)$ is discussed in literature, e.g. \cite[Lemma 5.1]{MP19} and \cite[Appendix A.1]{NP24}. Therefore, the M\"{o}bius inversion induces \cref{eq-UITranspCumulant2}.
\end{proof}

\begin{remark}
By examining the proof of \cref{prop-ui-tensormoment},  we can show the stronger statement: Suppose $\W_N$ is a family of $D_N\times D_N$ random matrices which is unitary invariant and has the factorization property \cref{eq-condition-Fact} and the bounded moments property \cref{eq-condition-Bdd}. Suppose further that $\W_N$ has the convergent joint moment limits up to degree $2$, i.e., the limits
    $$\lim_{N\to \infty}\tr(X_N),\quad \lim_{N\to \infty}\tr(X_{N,1}X_{N,2})$$
exist for all $X_N,X_{N,1},X_{N,2}\in \W_N$. Then for every $\underline{t} \in \{\pm 1\}^r\setminus \{(1,\ldots, 1), (-1,\ldots, -1)\}$, the family $\W_N^{\underline{t}}$ converges in (non-tensor) distribution with the same limit as in \cref{cor-UITranspose}. This was also observed in \cite{MP24}, where the stronger condition of the bounded cumulant property (\cref{eq-condition-BCP}) was assumed for $\W_N$, leading to almost sure convergence.
\end{remark}

Now in order to investigate the \emph{joint} tensor limit distribution between all partial transposes, we associate a permutation $\eps=\eps_{f}\in \bigsqcup_{k=1}^p S(\{k,-k\})$ for a function $f:[p]\to \{\pm 1\}$ defined by 
    $$\eps(k):=f(|k|)\,k, \quad k\in [\pm p].$$
In particular, $\eps=\id_{\pm p}$ when $f\equiv 1$ and $\eps=\delta=(1\, -1)\cdots (p\, -p)$ when $f\equiv -1$. Note that the $\eps_{f}$ commute with each other and $\eps_{f}^{-1}=\eps_{f}$.

\begin{lemma} \label{lem-PartialTransExpectation}
Suppose $\{X_1,\ldots, X_p\}$ is a family of $D\times D$ random matrices which is LUI. Then for a {function} $\underline{f}=(f_1,\ldots, f_r):[p]\to \{\pm 1\}^r$ and $\underline{\alpha}\in (S_p)^r$, we have
\small
\begin{equation} \label{eq-LuiPartialTransExpectation}
    \E\left[\Tr_{\underline{\alpha}}(X_1^{\underline{f}(1)},\ldots, X_p^{\underline{f}(p)})\right]=\sum_{\underline{\sigma},\underline{\tau}\in (S_p)^r} \E \left[\Tr_{\underline{\tau}}(X_1,\ldots, X_p) \right]\prod_{s=1}^r \left(d_s^{\# (\eps_s\alpha_s\delta \alpha_s^{-1}\eps_s \vee \sigma_s \delta \sigma_s^{-1})}\right) \Wg^{(U)}(\underline{\tau},\underline{\sigma})
\end{equation}
\normalsize
where each $\eps_s=\eps_{f_s}\in S_{\pm p}$ is defined according to the function $f_s:[p]\to \{\pm 1\}$. Furthermore, if $\{X_1,\ldots, X_p\}$ is UI, then
\small
\begin{equation} \label{eq-UIPartialTransExpectation}
    \E\left[\Tr_{\underline{\alpha}}(X_1^{\underline{f}(1)},\ldots, X_p^{\underline{f}(p)})\right]=\sum_{{\sigma},{\tau}\in S_p} \E \left[\Tr_{{\tau}}(X_1,\ldots, X_p) \right] \left(\prod_{s=1}^r d_s^{\# (\eps_s\alpha_s\delta \alpha_s^{-1}\eps_s \vee \sigma \delta \sigma^{-1})}\right) \Wg_{D}^{(U)}(\tau^{-1}\sigma)
\end{equation}
\normalsize
\end{lemma}
\begin{proof}
We only show \cref{eq-UIPartialTransExpectation}; the proof of \cref{eq-LuiPartialTransExpectation} is analogous. If $\W=\{X_1,\ldots, X_p\}$ is UI, then we can write
\begin{align*}
    \E\big[\Tr_{\underline{\alpha}}(X_1^{\underline{f}(1)},\ldots, X_p^{\underline{f}(p)})\big] &= \E \Big[\Tr_{\underline{\alpha}}\big((U X_1 U^*)^{\underline{f}(1)},\ldots, (U X_p U^*)^{\underline{f}(p)}\big)\Big] \\
    &=\E_{\W}\Big[\E_U\big[\Tr_{\underline{\alpha}}\big((U X_1 U^*)^{\underline{f}(1)},\ldots, (U X_p U^*)^{\underline{f}(p)}\big| \W\big]\Big],
\end{align*}
where $U$ is a Haar random unitary matrix of size $D$ independent from $\W$. In order to compute the integral $\E_U\big[\Tr_{\underline{\alpha}}\big((U X_1 U^*)^{\underline{f}(1)},\ldots, (U X_p U^*)^{\underline{f}(p)}\big| \W\big]$ using graphical Weingarten calculus, let us label $i$ and $-i$ for $U$ and $U^*$ boxes, respectively, corresponding to each $X_i$. The action of $\tau\in S_p$ from \cref{thm-GraphWeingarten} is the same as in \cref{eq-UITensorInv}: for each $s\in [r]$ and $i\in [p]$, we connect the $s$-th tensor component input of $X_i$ (label $-i$) with $s$-th tensor component output of $X_{\tau(i)}$ (label $\tau(i)$), so the total contribution becomes $\Tr_{\tau}(X_1,\ldots, X_p)$. 

On the other hand, for $\sigma\in S_p$ and $s\in [r]$, we obtain a loop (of dimension $d_s$) for each block of a partition  $\eps_s\alpha_s\delta \alpha_s^{-1}\eps_s \vee \sigma \delta \sigma^{-1}\in \mathcal{P}(\pm p)$. Indeed, we first connect the label $-i$ to $\sigma(i)$, for each $i\in [p]$, and induce the pairing $(-1,\sigma(1))\cdots (-p,\sigma(p))=\sigma\delta\sigma^{-1}$ which comprises the ``half'' of the loops. For the other half, we look at another pairing $\eps_s\alpha_s\delta\alpha_s^{-1}\eps_s$: starting from $\alpha_s\delta\alpha_s^{-1}$ which corresponds to the case without partial transpose, we swap the labels between $i$ and $-i$ for each $i\in [p]$ with $f_s(i)=-1$, inducing a new pairing $\eps_s\alpha_s\delta\alpha_s^{-1}\eps_s$ from the definition of $\eps_s=\eps_{f_s}$. We also refer to \cite[Theorem 3.7]{MP19} for a (non-graphical) formal argument for this.

We represent in \cref{fig:LUI-PT} the ``loop factor'' (i.e.~the factor corresponding to the product of powers of $d_s$ from \cref{eq-LuiPartialTransExpectation}) for the expansion of $\Tr_{(12)(3),(1)(23)}(X_1^{(-1,1)}, X_2, X_3^{(1,-1)})$ (see also \cref{fig:trace-invariant-PT}) corresponding to $\sigma=(1)(23)$.

\begin{figure}[htb]
    \centering
    \includegraphics[width=0.6\linewidth]{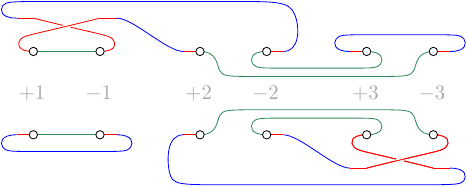}
    \caption{The ``loop factor'' from the Weingarten expansion of $\Tr_{\textcolor{blue}{(12)(3),(1)(23)}}(X_1^{\textcolor{red}{(-1,1)}}, X_2^{\textcolor{red}{(1,1)}}, X_3^{\textcolor{red}{(1,-1)}})$ corresponding to the permutation $\textcolor{seagreen}{\sigma=(1)(23)}$. The \textcolor{red}{partial transpositions $f_s$} and \textcolor{blue}{permutations $\alpha_s$} give raise to the pairings $\epsilon_1 \alpha_1\delta\alpha_1^{-1}\epsilon_1 = (1\, 2)(-1 \, -2)(3\, -3)$ and $\epsilon_2 \alpha_2\delta\alpha_2^{-1}\epsilon_2 = (1\, -1)(2 \, 3)(-2\, -3)$. while the \textcolor{seagreen}{permutation $\sigma$} gives raise to the pairing $\sigma\delta\sigma^{-1} = (1\, -1)(2\, -3)(3\, -2)$.}
    \label{fig:LUI-PT}
\end{figure}
\end{proof}

Recall from \cref{lem-pairings} that for $\eps\in \bigsqcup_{k=1}^p S(\{k,-k\})$, $\eps$ and $\delta$ commutes with each other, $(\eps\sigma\eps)(\delta(\eps\sigma\eps)^{-1}\delta)=\eps\sigma\delta\sigma^{-1}\delta\eps$, and $|(\eps\sigma\eps)(\delta(\eps\sigma\eps)^{-1}\delta)|=2|\sigma|$. The following lemma, proved in \cite[Lemma 4.3, Lemma 4.4, and Proposition 4.7]{MP19}, characterize the condition that the permutation $(\eps\sigma\eps)(\delta(\eps\sigma\eps)^{-1}\delta)$ lies on the geodesic between $\id_{\pm p}$ and $\alpha\delta\alpha^{-1}\delta$ for $\alpha\in S_p$ in the case $\eps=\eps_f$.

\begin{lemma} \label{lem-PairingGeodesic}
Let $\eps=\eps_{f}\in S_{\pm p}$ be the permutation associated to a function $f:[p]\to \{\pm 1\}$.
\begin{enumerate}
    \item Suppose $\sigma\in S_p$ satisfies $\Pi(\sigma)\leq \ker f=\{f^{-1}(1),f^{-1}(-1)\}$, i.e., $f$ is constant on the cycles of $\sigma$. Then there exists a unique permutation $\sigma_{f}\in S_p$ such that $\eps\sigma\delta\sigma^{-1}\delta\eps=\sigma_{f}\delta\sigma_{f}^{-1}\delta$. Specifically, if $\sigma=c_1\ldots, c_l$ is the cycle decomposition of $\sigma$, then $\sigma_{f}=c_1^{\lambda_1}\cdots c_l^{\lambda_l}$ where $\lambda_k\in \{\pm 1\}$ is the (constant) value of $f$ on the cycle $c_k$. 
    \item For $\sigma,\alpha\in S_p$, $\eps\sigma\delta\sigma^{-1}\delta\eps\in S_{NC}(\alpha\delta\alpha^{-1}\delta)$ if and only if $\Pi(\sigma)\leq \ker f$ and $\sigma_{f}\in S_{NC}(\alpha)$.
\end{enumerate}
\end{lemma}

Now we can state and prove the \textit{joint} asymptotic behavior between (local-) UI random matrices and their partial transposes.

\begin{theorem} \label{thm-LocUITranspose}
Let $\W_N$ be a family of $D_N\times D_N$ random matrices.
\begin{enumerate}
    \item If $\W_N$ is local-unitary invariant, has the tensor factorization property \cref{eq-condition-TensorFact}, and converges in tensor distribution, then the family of all partial transposes 
        $$\bigcup_{\underline{t}\in \{\pm 1 \}^r} \W_N^{\underline{t}}$$
    satisfies the tensor factorization property and has the joint tensor distribution limit as $N\to \infty$. Furthermore, the families $\W_N$ and $\W_N^{\top}$ are asymptotically tensor free.
    
    \item If $\W_N$ is unitary invariant, has the factorization property \cref{eq-condition-Fact}, and converges in distribution, then all the $2^r$ families $(\W_N^{\underline{t}})_{\underline{t}\in \{\pm 1 \}^r}$ are both asymptotically free and asymptotically tensor free as $N\to \infty$.
\end{enumerate}
\end{theorem}
\begin{proof}
Let us denote by $X_j^{\underline{f}(j)}$ ($j=1,\ldots, p$) the $p$-elements from $\bigcup_{\underline{t}\in \{\pm 1\}^r}\W_N^{\underline{t}}$, where $\underline{f}=(f_1,\ldots,f_r):[p]\to \{\pm 1\}^r$ is a partial transpose function.

(1) Continuing from \cref{eq-LuiPartialTransExpectation} we have
\begin{align*}
    \E\left[\tr_{\underline{\alpha}}(X_1^{\underline{f}(1)},\ldots, X_p^{\underline{f}(p)})\right]&=\sum_{\underline{\sigma},\underline{\tau}\in (S_p)^r} \E \left[\tr_{\underline{\tau}}(X_1,\ldots, X_p) \right] \left(\Mob(\underline{\tau}^{-1} \underline{\sigma}) + o(1) \right)\\
    &\qquad\qquad \times \prod_{s=1}^r \left(d_s^{\#(\tau_s)-\#(\alpha_s)}\;d_s^{\# (\eps_s\alpha_s\delta \alpha_s^{-1}\eps_s \vee \sigma_s \delta \sigma_s^{-1})} \;d_s^{-p-|\tau_s^{-1}\sigma_s|})\right)
\end{align*}

In order to estimate the exponents of $d_s$'s, we first apply \cref{lem-PairingSup} and the relation $\eps_s=\delta\eps_s\delta$ to observe that
    $$\# (\eps_s\alpha_s\delta \alpha_s^{-1}\eps_s \vee \sigma_s \delta \sigma_s^{-1})=\frac{1}{2}\#(\eps_s \alpha_s \delta \alpha_s^{-1}\delta \eps_s \delta\sigma_s\delta\sigma_s^{-1})=\frac{1}{2} \#((\alpha_s \delta \alpha_s^{-1}\delta)\cdot (\eps_s \sigma_s\delta\sigma_s^{-1}\delta\eps_s)^{-1}).$$
Note that the permutation above is defined in $S_{\pm p}$ although $\sigma_s,\sigma_s\in S_p$. Therefore, the whole exponent of $d_s$ becomes
\begin{align*}
    &\#(\tau_s)-\#(\alpha_s) + \# (\eps_s\alpha_s\delta \alpha_s^{-1}\eps_s \vee \sigma_s \delta \sigma_s^{-1}) -p-|\tau_s^{-1}\sigma_s|\\
    &= (p-|\tau_s|)-(p-|\alpha_s|)+\frac{1}{2}(2p-|(\alpha_s \delta \alpha_s^{-1}\delta)\cdot (\eps_s \sigma_s\delta\sigma_s^{-1}\delta\eps_s)^{-1}|)-p-|\tau_s^{-1}\sigma_s|\\
    &=-\frac{1}{2}\left((|(\alpha_s \delta \alpha_s^{-1}\delta)\cdot (\eps_s \sigma_s\delta\sigma_s^{-1}\delta\eps_s)^{-1}|+|\eps_s \sigma_s\delta\sigma_s^{-1}\delta\eps_s|-|\alpha_s\delta\alpha_s^{-1}\delta|)\right) -(|\tau_s^{-1}\sigma_s|+|\tau_s|-|\sigma_s|)\\
\end{align*}
where the last equality comes from $|\alpha_s\delta\alpha_s^{-1}\delta|=2|\alpha_s|$ and $|\eps_s\sigma_s\delta\sigma_s^{-1}\delta\eps_s|=2|\sigma_s|$ as observed before. Therefore, the exponent above is non-positive for all $\sigma_s,\alpha_s$, and by \cref{lem-PairingGeodesic}, this attains $0$ if and only if $\Pi(\sigma_s)\leq \ker f_s$, $\sigma_{f_s}\in S_{NC}(\alpha_s)$, and $\tau_s\in S_{NC}(\sigma_s)$. Consequently, we have
\begin{align*}
    \lim_{N\to \infty} \E\left[\tr_{\underline{\alpha}}(X_1^{\underline{f}(1)},\ldots, X_p^{\underline{f}(p)})\right] &= \sum_{\substack{\underline{\sigma}\in (S_p)^r,\\ \forall\, s: \;\Pi(\sigma_s)\leq \ker f_s,\\ \;\quad\sigma_{f_s}\in S_{NC}(\alpha_s)}} \sum_{\underline{\tau}\, \leq\, \underline{\sigma}}\varphi_{\underline{\tau}}(x_1,\ldots, x_p)\, \Mob(\underline{\tau}^{-1}\underline{\sigma})\\
    &= \sum_{\substack{\underline{\sigma}\in (S_p)^r,\\ \forall\, s:\;\Pi(\sigma_s)\leq \ker f_s,\\ \;\quad\sigma_{f_s}\in S_{NC}(\alpha_s)}} \kappa_{\underline{\sigma}}(x_1,\ldots, x_p),
\end{align*}
where $(x_1,\ldots, x_p)$ is the joint tensor distribution limit of $(X_1,\ldots, X_p)$. This shows the convergence in tensor distribution of $\bigcup_{\underline{t}\in \{\pm 1\}^r}\W_N^{\underline{t}}$. Now if we take $f_1=f_2=\cdots f_r=f$ so that we consider the tensor distribution of $\W_N\cup \W_N^{\top}$, then
\begin{align*}
    \lim_{N\to \infty} \E\left[\tr_{\underline{\alpha}}(X_1^{f(1)},\ldots, X_p^{f(p)})\right] &= \sum_{\substack{\underline{\sigma}\in (S_p)^r,\\ \bigvee \Pi(\sigma_s)\leq \ker f,\\ (\underline{\sigma})_f \in S_{NC}(\underline{\alpha})}} \kappa_{\underline{\sigma}}(x_1,\ldots, x_p)\\
    &=\sum_{\substack{\underline{\sigma_{\pm}}\in S(f^{-1}(\pm 1))^r\\ \underline{\sigma_+}\sqcup (\underline{\sigma_-})^{-1}\in S_{NC}(\underline{\alpha})} }\kappa_{\underline{\sigma_+}\sqcup \underline{\sigma_{-}}}(x_1,\ldots, x_p),\\
    &=\sum_{\substack{\underline{\sigma_{\pm}}\in S(f^{-1}(\pm 1))^r\\ \underline{\sigma_+}\sqcup (\underline{\sigma_{-}})^{-1}\in S_{NC}(\underline{\alpha})}} \kappa_{\underline{\sigma_+}}((x_j)_{j\in f^{-1}(1)}) \; \kappa_{(\underline{\sigma_{-}})^{-1}}((x_j^{\top})_{j\in f^{-1}(-1)}).
\end{align*}
Therefore, we obtain the asymptotic tensor freeness of $\W_N$ and $\W_N^{\top}$ by \cref{cor-TensorFreeSubsetMoment,prop-LuiTranspose}.

(2) Continuing from \cref{eq-UIPartialTransExpectation}, we can analogously proceed and obtain that
\begin{align*}
    \lim_{N\to \infty} \E\left[\tr_{\underline{\alpha}}(X_1^{\underline{f}(1)},\ldots, X_p^{\underline{f}(p)})\right] &= \sum_{\substack{{\sigma}\in S_p, \\ \forall\, s:\;\Pi(\sigma)\leq \ker f_s,\\ \;\quad\sigma_{f_s}\in S_{NC}(\alpha_s)}} \kappa_{\sigma}(x_1,\ldots, x_p).
\end{align*}
The conditions for $\sigma$ above imply   that we can decompose $\sigma=\bigsqcup_{\underline{t}\in \{\pm 1\}^r}\sigma_{\underline{t}}$ for $\sigma_{\underline{t}}\in S(f^{-1}(\underline{t}))$ such that $\sigma_{f_s}=\bigsqcup_{\underline{t}\in \{\pm 1\}^r} ({\sigma_{\underline{t}}})^{t_s}\in S_{NC}(\alpha_s)$ for all $s\in [r]$. Thus, we can further write the above sum into
    $$\sum_{\substack{\sigma_{\underline{t}}\in S(\underline{f}^{-1}(\underline{t})) \\ \forall\, s:\,\bigsqcup_{\underline{t}\in \{\pm 1\}^r} ({\sigma_{\underline{t}}})^{t_s}\in S_{NC}(\alpha_s)}} \prod_{\underline{t}\in \{\pm 1\}^r}\kappa_{\sigma_{\underline{t}}}((x_j)_{j\in \underline{f}^{-1}(\underline{t})}).$$
Therefore, \cref{cor-TensorFreeSubsetMoment,eq-UITranspCumulant} imply that $(\W_N^{\underline{t}})_{\underline{t}\in \{\pm 1 \}^r}$ are asymptotically tensor free. Furthermore, setting $\alpha_s\equiv \gamma_p$ gives that
\begin{align*}
    \lim_{N\to \infty} \E\left[\tr (X_1^{\underline{f}(1)} \cdots X_p^{\underline{f}(p)})\right] &= \sum_{\substack{\sigma_{\underline{t}}\in S(\underline{f}^{-1}(\underline{t})) \\ \forall\, s:\,\bigsqcup_{\underline{t}\in \{\pm 1\}^r} ({\sigma_{\underline{t}}})^{t_s}\in S_{NC}(\gamma_p)}} \prod_{\underline{t}\in \{\pm 1\}^r}\kappa_{\sigma_{\underline{t}}}((x_j)_{j\in \underline{f}^{-1}(\underline{t})})
\end{align*}
Note that the condition implies $\sigma_{\underline{t}}=(\sigma_{\underline{t}})^{-1}$ whenever $\underline{t}\notin \{(1,\ldots, 1),(-1,\ldots, -1)\}$, which forces that $\sigma_{\underline{t}}$ has only cycles of size 1 or 2. Therefore, $(\W_N^{\underline{t}})_{\underline{t}\in \{\pm 1 \}^r}$ are asymptotically free by \cref{cor-freemoment,cor-UITranspose}.
\end{proof}

\begin{remark}
\begin{enumerate}
    \item An LUI random matrix is not necessarily asymptotically tensor free from its partial transposes. For example, if $X=G\otimes I_N$ where $G$ is a (normalized) $N\times N$ GUE matrix, then $X^{(1,-1)}=(\id\otimes \top)(X)=X$ is not asymptotically tensor free from $X$ as $N\to \infty$.

    \item Similarly, an LUI random matrix is not necessarily asymptotically free from its transpose. For example, if $X=G_1\otimes G_2$ where $G_1$ and $G_2$ are independent GUE matrices, then the pair $(X,X^{\top})$ converges jointly in {(standard)} distribution to $(x,y)=(s_1\otimes s_1,s_2\otimes s_2)$ where $s_1,s_2$ are free standard semicircular elements, which are not freely independent as discussed in \cref{ex-TensorFreeNonFree}{; $X$ and $X^\top$ are however asymptotically tensor freely independent as per the previous result.}.
\end{enumerate}

\end{remark}

We now present the following main result of this section, which encapsulates all our results so far and can be regarded as the most general statement on the behavior of the family of LUI random matrices. We remark that the similar statement for \textit{bipartite} (i.e. $r=2$) UI random matrix has been considered in \cite[Corollary 4.8]{MP24}.

\begin{theorem} \label{thm-indepTranspose}
Let $\W_N^{(1)},\ldots, \W_N^{(L)}$ be independent families of $D_N\times D_N$ random matrices.
\begin{enumerate}
    \item If each $\W_N^{(i)}$ is local-unitary invariant, has the tensor factorization property \cref{eq-condition-TensorFact}, and converges in tensor distribution, then the $L$ families of all partial transposes 
        $$\bigcup_{\underline{t}\in \{\pm 1 \}^r} \big(\W_N^{(1)}\big)^{\underline{t}}, \ldots, \bigcup_{\underline{t}\in \{\pm 1 \}^r} \big(\W_N^{(L)}\big)^{\underline{t}}$$
    are asymptotically tensor free as $N\to \infty$.
    
    \item If each $\W_N^{(i)}$ is unitary invariant, has the factorization property \cref{eq-condition-Fact}, and converges in distribution, then all the $2^r L$ families $\left(\big(\W_N^{(i)}\big)^{\underline{t}} \right)_{i\in [L],\,\underline{t}\in \{\pm 1 \}^r}$ are both asymptotically free and asymptotically tensor free as $N\to \infty$.
\end{enumerate}
\end{theorem}
\begin{proof}
(1) and the tensor freeness in (2) is a direct consequence of \cref{thm-LocOI-tensorfree,thm-LocUITranspose} since each family $\bigcup_{\underline{t}}(\W_N^{(i)})^{\underline{t}}$ is LOI (\cref{ex-UI}) and convergent in tensor distribution. For the freeness in (2), we first observe that the family $\bigcup_{i=1}^L \W_N^{(i)}$ is UI and has the first order limit by \cref{prop-IndepUI,thm-UIasympfree}. Therefore, \cref{thm-LocUITranspose} (2) implies that the families $\Big(\bigcup_{i=1}^L (\W_N^{(i)})^{\underline{t}}\Big)_{\underline{t}\in \{\pm 1\}^r}$ are asymptotically free. Now by the associativity of free independence, it is enough to show the asymptotic freeness between the families $\big((\W_N^{(i)})^{\underline{t}}\big)_{i=1}^L$ for each $\underline{t}\in \{\pm 1\}^r$. The case $\underline{t}=(1,\ldots, 1)$ is again from \cref{thm-UIasympfree}, and the case $\underline{t}=(-1,\ldots, -1)$ is from \cref{eq-TranspTensorCumulants} with $r=1$:
    $$\lim_{N\to \infty}\kappa_p(X_1^{\top},\ldots, X_p^{\top})=\lim_{N\to \infty}\kappa_p(X_p,\ldots, X_1)=0$$
whenever $X_j\in \W_N^{(f(j))}$ with a non-constant function $f$. Finally, we can argue similarly for the remaining cases $\underline{t}\notin\{(1,\ldots, 1),(-1,\ldots, -1)\}$ but using \cref{eq-UITranspCumulant2}.
\end{proof}

We now obtain analogue results for OI random matrices. First of all, \cref{cor-UITranspose} has a natural analogue for orthogonal invariance, with a slight modification of the proof using \cref{prop-oi-tensormoment}. 

\begin{proposition} \label{cor-OITranspose}
Suppose $\W_N$ is a family of $D_N\times D_N$ random matrices which is OI, converges in distribution and has the factorization property \cref{eq-condition-Fact}. If the family $\W_N\cup \W_N^{\top}$ satisfies the bounded moments condition \cref{eq-condition-Bdd}, then for every $\underline{t} \in \{\pm 1\}^r$, the family $\W_N^{\underline{t}}$ converges in tensor distribution whose limit has the same tensor free cumulants with \cref{eq-UITranspCumulant} and the same free cumulants with \cref{eq-UITranspCumulant2} for the cases $\underline{t} \notin \{(1,\ldots, 1),(-1,\ldots, -1)\}$.
\end{proposition} 

The analogue of \cref{eq-UIPartialTransExpectation} is as follows; we leave the details to the reader.

\begin{lemma} \label{lem-OIPartialTransExpectation}
Suppose $\{X_1,\ldots, X_p\}$ is a OI family of $D\times D$ random matrices. Then for a {function} $\underline{f}=(f_1,\ldots, f_r):[p]\to \{\pm 1\}^r$ and $\underline{\alpha}\in (S_p)^r$, we have
\small
\begin{equation} \label{eq-OIPartialTransExpectation}
    \E\left[\Tr_{\underline{\alpha}}(X_1^{\underline{f}(1)},\ldots, X_p^{\underline{f}(p)})\right]=\sum_{\pi,\rho \in \mathcal{P}_2(\pm p)} \E \left[\Tr_{{\rho\vee \delta}}(X_1,\ldots, X_p) \right] \left(\prod_{s=1}^r d_s^{\# (\eps_s\alpha_s\delta \alpha_s^{-1}\eps_s \vee \pi)}\right) \Wg_{D}^{(O)}(\pi,\rho)
\end{equation}
\normalsize
where each $\eps_s=\eps_{f_s}\in S_{\pm p}$ is defined according to the function $f_s:[p]\to \{\pm 1\}$.
\end{lemma}

\begin{theorem} \label{thm-OITranspose}
Suppose $\W_N$ is a family of $D_N\times D_N$ OI random matrices having factorization property \cref{eq-condition-Fact} and convergent in distribution. If the family $\W_N\cup \W_N^{\top}$ satisfies the boundedness property \cref{eq-condition-Bdd}, then 
the $2^{r-1}$ families of partial transposes 
    $$(\W_N^{\underline{t}})_{\underline{t}\in \{\pm 1 \}^r:t_r=1}$$
are both asymptotically free and tensor free.
\end{theorem}
\begin{proof}
Suppose $X_1,\ldots, X_p\in \W_N$ jointly converge in distribution to $x_1,\ldots, x_p$ as $N\to \infty$, and take a function $\underline{f}:[p]\to \{\pm 1\}^r$ with $f_r\equiv 1$. Similarly as in the proof of \cref{thm-LocUITranspose}, we can continue from \cref{eq-OIPartialTransExpectation} to obtain that
\begin{align*}
    \E\left[\tr_{\underline{\alpha}}(X_1^{\underline{f}(1)},\ldots, X_p^{\underline{f}(p)})\right]&=\sum_{\pi, \rho \in \mathcal{P}_2(\pm p)} \E \left[\tr_{\rho\vee \delta}(X_1,\ldots, X_p) \right] \left(\Mob(\rho\vee \pi) + o(1) \right)\\
    &\qquad\qquad \times \prod_{s=1}^r \left(d_s^{\#(\rho\vee\delta)-\#(\alpha_s)}\;d_s^{\# (\eps_s\alpha_s\delta \alpha_s^{-1}\eps_s \vee \pi)} \;d_s^{-p-|\rho\vee\pi|})\right),
\end{align*}
and the exponent of each $d_s$ becomes
\begin{align*}
    &\#(\rho\vee\delta)-\#(\alpha_s) + \# (\eps_s\alpha_s\delta \alpha_s^{-1}\eps_s \vee \pi) -p-|\rho\vee\pi|\\
    &=-\frac{1}{2}(|(\eps_s \alpha_s \delta \alpha_s^{-1}\delta\eps_s)\cdot (\pi\delta)^{-1}|+|\pi\delta|-2|\alpha_s|) -\frac{1}{2}(|(\pi\delta)\cdot(\rho\delta)^{-1}|+|\rho\delta|-|\pi\delta|)\\
\end{align*}
which is non-positive and attains zero if and only if $\rho\delta \leq \pi\delta \leq\eps_s\alpha_s\delta\alpha_s^{-1}\delta \eps_s$. Note that $\eps_r=\id_{\pm p}$ and the condition implies that there exist (unique) $\sigma,\tau\in S_p$ such that 
\begin{center}
    $\tau\leq \sigma\leq \alpha_r$, $\rho\delta=\tau\delta\tau^{-1}\delta$, and $\pi\delta=\sigma\delta\sigma^{-1}\delta$.
\end{center}
Therefore, we have
\small
\begin{align*}
    \lim_{N\to \infty} \E\left[\tr_{\underline{\alpha}}(X_1^{\underline{f}(1)},\ldots, X_p^{\underline{f}(p)})\right] &= \sum_{\sigma\in S_{NC}(\alpha_r)}\Big(\prod_{s=1}^{r-1}\mathds{1}_{\sigma\delta\sigma^{-1}\delta\leq \eps_s\alpha_s\delta\alpha_s^{-1}\delta\eps_s} \Big) \sum_{\tau\in S_{NC}(\sigma)}\varphi_{\tau}(x_1,\ldots, x_p)\, \Mob({\tau}^{-1}{\sigma})\\
    &= \sum_{\substack{{\sigma}\in S_{NC}(\alpha_r),\\ \forall\, s\in [r-1]:\;\Pi(\sigma)\leq \ker f_s,\\ \;\quad\sigma_{f_s}\in S_{NC}(\alpha_s)}} \kappa_{\sigma}(x_1,\ldots, x_p).
\end{align*}
\normalsize
Now we can repeat all the arguments from \cref{thm-LocUITranspose} (2) to conclude the asymptotic freeness and tensor freeness of the families $(\W_N^{\underline{t}})_{\underline{t}\in \{\pm 1 \}^r:t_r=1}$.
\end{proof}

\begin{remark}
\begin{enumerate}
    \item Under the conditions in \cref{thm-OITranspose}, we can actually show the asymptotic freeness and tensor freeness between any $2^{r-1}$ families of partial transposes $(\W_N^{\underline{t}})$ as long as $t_{s_0}\equiv 1$ or $t_{s_0}\equiv -1$ for some $s_0\in [r]$. On the other hand, we cannot conclude the freeness between $\W_N$ and $\W_N^{\top}$.  For example, if $X$ is a (normalized) $N^2\times N^2$ GOE matrix, then $X$ is neither asymptotically free nor tensor free from $X^{\top}=X$.

    \item We cannot expect any type of independence when the assumption of orthogonal invariance is weakened to \emph{local-orthogonal invariance}. For example, if $G$ is a (normalized) $N\times N$ GOE matrix, then $X=G\otimes I_d$ is asymptotically (tensor) free from neither of $X^{(1,-1)}=X^{(-1,1)}=X^{\top}=X$.

    \item If $S$ is a subset of $\{\pm 1\}^r$ having cardinality larger than $2^{r-1}$, then we can always find two elements $\underline{t}_1,\underline{t}_2\in S$ such that $\underline{t}_1=-\underline{t}_2$ by the pigeonhole principle. This implies that $2^{r-1}$ is the maximum cardinality of subsets $S\subseteq \{\pm 1\}^r$ with a property that $(X^{\underline{t}})_{\underline{t}\in S}$ are asymptotically (tensor) free whenever $X$ is an OI random matrix satisfying the conditions in \cref{thm-OITranspose}.
\end{enumerate}
\end{remark}

\subsection{Tensor embedding of random matrices} \label{sec-EmbeddingTensorFree}

We consider a set of $k$ bipartite matrices
    $$\{X_1, \ldots, X_k\} \subset \M{d_t} \otimes \M{d_b}$$
where $d_t$ and $d_b$ are the dimensions of the tensor factors on which these matrices act (corresponding to the ``top'' and ``bottom'' spaces). We shall embed these matrices in the larger space $\M{d_t}^{\otimes q} \otimes \M{d_b}^{\otimes r}$ according to a bipartite graph $G$ on $q+r$ vertices having $k$ edges, as follows. For $i \in [k]$, define the matrix $Y_i$ corresponding to the edge $(t(i), b(i))$ as
$$\M{d_t}^{\otimes q} \otimes \M{d_b}^{\otimes r} \ni Y_i:= X_i^{t(i),b(i)} \otimes I^{[q+r] \setminus \{t(i),b(i)\}},$$
where superscripts denote the spaces on which the operators act; see \cref{fig:XY-embeddings} for two examples. Since we require that the graph $G$ is bipartite, the functions defining the edges are such that
$$t : [k] \to [q] \qquad \text{ and } b:[k] \to q+[r].$$
In other words, the top spaces correspond to indices $1,2, \ldots, q$, while the bottom spaces correspond to the indices $q+1, q+2, \ldots, q+r$. 
\begin{figure}[htb]
    \centering
    \includegraphics[width=1\linewidth]{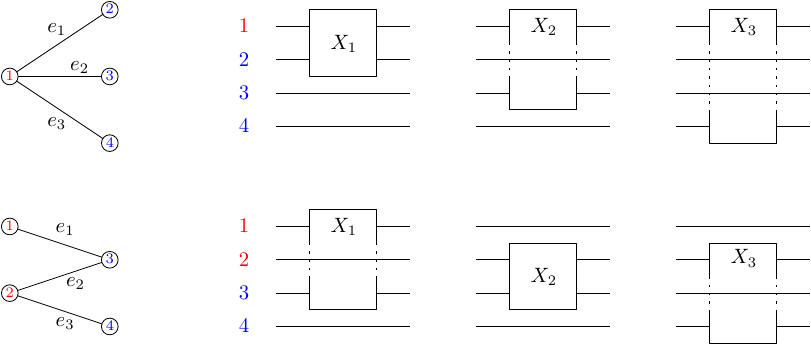}
    \caption{Two examples of embeddings of $k=3$ bipartite matrices $X_{1,2,3}$ into a larger tensor product of matrix algebras. On the top row, we have $q=1$ ``top'' space and $r=3$ ``bottom'' spaces. On the left we display the bipartite star graph, while on the right we draw the diagrams for the matrices $Y_{1,2,3}$ acting on the 4 tensor factors. A different bipartite graph and the corresponding matrices are considered in the bottom panel, with $q=r=2$.}
    \label{fig:XY-embeddings}
\end{figure}

We can now state the main result in this framework. 

\begin{theorem}\label{thm:embed-different-spaces}
    Consider a unitarily invariant family $\{X_1, \ldots, X_k\}$ of $d_td_b\times d_td_b$ random matrices with the property that each element $X_i$ of the family converges in distribution and satisfies the factorization property \cref{eq-condition-Fact} as $d_t, d_b \to \infty$. Given two functions $t : [k] \to [q]$ and $b:[k] \to q+[r]$, embed these matrices as above in $\M{d_t}^{\otimes q} \otimes \M{d_b}^{\otimes r}$, and denote the corresponding family as $Y_1, \ldots, Y_k$. Assume that the pair of functions $(t,b):[k] \to [q] \times (q+[r])$ is \emph{injective}, that is, no two matrices $Y_i, Y_j$ ($i \neq j$) act on the same spaces in the larger tensor product. Assume also that the family $X_1, \ldots, X_k$ satisfies the bounded moments property from \cref{eq-condition-Bdd}
    $$\forall p \geq 1, \, \forall \tau \in S_p, \, \forall f:[p] \to [k], \qquad \sup_{d_t,d_b \geq 1} \Big| \E [\tr_\tau(X_{f(1)},\ldots, X_{f(p)})] \Big| < \infty.$$ 
    Then, as $d_t, d_b \to \infty$, the family $(Y_1, \ldots, Y_k)$ is asymptotically tensor free. 
\end{theorem}
\begin{proof}
    As before, we shall use the method of moments, compute the tensor free cumulants of the family $(Y_1, \ldots, Y_k)$, and conclude by showing that mixed tensor free cumulants vanish (see \cref{def:tensor-freeness}). 
    
    Since the family $(X_1, \ldots, X_k)$ is unitarily invariant, the following equality holds in distribution: 
    $$\forall U \in \mathcal U(d_td_b), \qquad (Y_1, \ldots, Y_k) \stackrel{(d)}{=} \Big( U_{t(1),b(1)} Y_1 U_{t(1),b(1)}^*, \ldots, U_{t(1),b(1)} Y_1 U_{t(1),b(1)}^* \Big),$$
    where $U_{x,y}$ denotes the unitary operator $U \otimes I_{d_t}^{\otimes (q-1)} \otimes I_{d_b}^{\otimes (r-1)}$, with the $U$ operator acting on the tensor factors $x \in [q]$ and $y \in q+[r]$.
    
    We start by applying the (unitary) Weingarten formula from \cref{thm:Weingarten} to write, for any permutation $(q+r)$-tuple $\underline{\alpha}$ and any word given by $f: [p] \to [k]$:
    \begin{align}\label{eq:moment-Yf}
        \E [\Tr_{\underline{\alpha}}(Y_f)] &= \E_{\underline{X}}\Big[\E_{U}\big[\Tr_{\underline{\alpha}}(U_{t\circ f, b\circ f}\, Y_f\, U_{t\circ f, b\circ f}^*) \, \big| \, \underline{X} \big] \Big] \nonumber \\
        &= \sum_{\sigma, \tau \in S_p} \E[\Tr_\tau(X_f)] d_t^{F_t(\sigma, \underline{\alpha},f)} d_b^{F_b(\sigma, \underline{\alpha},f)} \Wg^{(U)}_{d_td_b}(\tau^{-1}\sigma).
    \end{align}
    The sum in the formula above is indexed by a pair of permutations $\sigma, \tau$ in $S_p$ since in the trace $\Tr_{\underline{\alpha}}(Y_f)$ there are $p$ boxes corresponding to the Haar-distributed random unitary matrix $U \in \mathcal U(d_td_b)$ (and, respectively, $p$ $\bar U$ conjugate boxes). For each $i \in [p]$, the diagram around the element $Y_{f(i)}$ is depicted in \cref{fig:diagram-Y-f-i}. 

    \begin{figure}[!htb]
        \centering
        \includegraphics[scale=1]{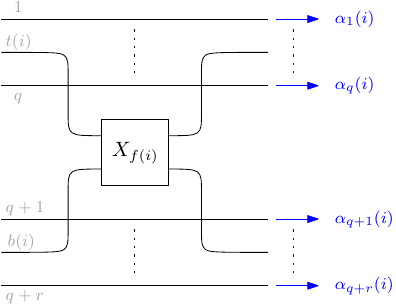}
        \caption{Diagram for $\Tr_{\underline{\alpha}}(Y_f)$ around the $i$-th group.}
        \label{fig:diagram-Y-f-i}
    \end{figure}

{
Using the boundedness property from the statement, we have 
$$\E\Tr_\tau(X_f) = (d_td_b)^{\#\tau} O(1).$$

Let us now consider the functions $F_t(\sigma, \underline{\alpha},f)$, resp.~$F_b(\sigma, \underline{\alpha},f)$, counting the number of loops of value $d_t$ (resp.~$d_b$) in \cref{eq:moment-Yf}. Tracking the loops of different sizes along the tensor legs, we claim that 
$$F_t(\sigma, \underline{\alpha},f) = \#(\tilde \alpha_t^{-1} \sigma_{t,f}) \quad \text{ and } \quad F_b(\sigma, \underline{\alpha},f) = \#(\tilde \alpha_b^{-1} \sigma_{b,f}),$$
where the permutations $\tilde \alpha_t, \sigma_{t,f}, \tilde \alpha_b, \sigma_{b,f}$ are defined as follows. First, let us focus on the ``top'' permutations, $\tilde \alpha_t, \sigma_{t,f} \in S([p] \times [q])$:
\begin{align*}
    \tilde \alpha_t(i,s) &= \big(\alpha_s(i), s \big)\\
    \sigma_{t,f}(i,s) &= \begin{cases}
        \big( \sigma(i), t \circ f(\sigma(i)) \big) &\quad \text{ if } t \circ f(i) = s\\
        (i,s) &\quad \text{ if } t \circ f(i) \neq s.
    \end{cases}
\end{align*}
Let us prove the claim and justify the definitions above; we leave the corresponding claims and definitions for the ``bottom'' permutations $\tilde \alpha_b, \sigma_{b,f} \in S([p] \times (q+[r]))$ to the reader. The permutation $\tilde \alpha_t$ encodes the wiring of the trace invariant $\Tr_{\underline{\alpha}}(\cdot)$: the $s$-th tensor factor in the $i$-th group is connected to the $s$-th tensor factor in the $\alpha_s(i)$-th group, see \cref{eq:def-trace-invariant}. The permutation $\sigma_{t,f}$ encodes the wiring corresponding to the (Weingarten) pairing of the matrices $Y_f$. There are two cases: 
\begin{itemize}
    \item If $t(f(i)) \neq s$, there is no matrix $X$ at the $s$-th tensor factor in the group $i$. Hence, there is no Weingarten wiring, and we use the identity matrix which means $\sigma_{t,f}(i,s) = (i,s)$.
    \item If $t(f(i)) = s$, the (top wire of the) matrix $X_{f(i)}$ is present at that position. The graphical Weingarten formula from \cref{thm-GraphWeingarten} implies that this wire is connected to the corresponding top wire in the group $\sigma(i)$; we have $\sigma_{t,f}(i,s) = ( \sigma(i), t \circ f(\sigma(i)) )$.
\end{itemize}

Combining all the observations above with \cref{eq:moment-Yf}, we obtain:
$$\E [\tr_{\underline{\alpha}}(Y_f)] = O(1)\sum_{\sigma, \tau \in S_p}d_t^{-\tilde \alpha_t + \#\tau + \#(\tilde \alpha_t^{-1} \sigma_{t,f}) - p - |\sigma^{-1}\tau|} d_b^{-\tilde \alpha_b + \#\tau + \#(\tilde \alpha_b^{-1} \sigma_{b,f}) - p - |\sigma^{-1}\tau|}.$$

Let us analyze first the exponent of $d_t \to \infty$: 
\begin{align*}
    -\tilde \alpha_t + \#\tau + \#(\tilde \alpha_t^{-1} \sigma_{t,f}) - p - |\sigma^{-1}\tau| &= -pq + |\tilde \alpha_t| + p-|\tau| + pq - |\tilde \alpha_t^{-1} \sigma_{t,f}| - p - |\sigma^{-1}\tau| \\
    &\leq |\tilde \alpha_t| - |\sigma| - |\tilde \alpha_t^{-1} \sigma_{t,f}| \\
    & \leq | \sigma_{t,f}|- |\sigma|\\
    & =0,
\end{align*}
where we have used the triangle inequality for the length function of permutations in the first two inequalities. The last equality follows from the definition of the permutation $\sigma_{t,f}$ above: $\sigma_{t,f}$ is a collection of fixed points, together with the permutation $\sigma$ embedded on different tensor factors given by the functions $t$ and $f$. 

In the chain of inequalities above, we have equality if and only if: 
\begin{itemize}
    \item first inequality: $\id \to \tau \to \sigma$ geodesic in $S_p$;
    \item second inequality: $\id \to \sigma_{t,f} \to \tilde \alpha_t$ geodesic in $S([p] \times [q])$.
\end{itemize}

From the second condition above and the definition of the permutation $\sigma_{t,f}$ we obtain that the terms that contribute asymptotically in \cref{eq:moment-Yf} correspond to permutations $\sigma$ having a decomposition 
$$\sigma = \bigsqcup_{s \in \Im(t \circ f)} \sigma \big|_{(t \circ f)^{-1}(s)}.$$
Using a similar reasoning for the tensor factors $s \in q+[r]$, we obtain a decomposition 
$$\sigma = \bigsqcup_{s \in \Im(b \circ f)} \sigma \big|_{(b \circ f)^{-1}(s)}.$$
We use now the hypothesis that the function pair $(t,b)$ is \emph{injective} to produce a joint decomposition
\begin{equation}\label{eq:sigma-decomp}
    \sigma = \bigsqcup_{j \in \Im(f)} \underbrace{\sigma \big|_{f^{-1}(j)}}_{=:\sigma_j}
\end{equation}
such that:
\begin{equation} \label{eq:sigma-decomp-properties}
    \begin{aligned}
        \forall s \in [q] \qquad \id \sqcup \bigsqcup_{j \in t^{-1}(s)} \sigma_j &\leq \alpha_s,\\
        \forall s \in q+[r] \qquad \id \sqcup \bigsqcup_{j \in b^{-1}(s)} \sigma_j &\leq \alpha_s.
    \end{aligned}
\end{equation}
From the geodesic condition $\id - \tau - \sigma$ we infer that $\tau$ admits a decomposition of the same type:
\begin{equation}\label{eq:tau-decomp}
\tau = \bigsqcup_{j \in \Im(f)} \tau_j \qquad \text{ with } \qquad \tau_j \leq \sigma_j \quad \forall j \in \Im(f).
\end{equation}
Hence we can rewrite \cref{eq:moment-Yf} as
\begin{align*}
    \E [\tr_{\underline{\alpha}}(Y_f)] &= \big( 1+o(1) \big)  \sum_{\substack{\sigma \text{ as in } \eqref{eq:sigma-decomp}, \eqref{eq:sigma-decomp-properties} \\ \tau \text{ as in } \eqref{eq:tau-decomp}}} \quad \prod_{j \in \Im(f)} \phi_{\tau_j}(x_j) \Mob(\sigma_j,\tau_j)\\
    &= \big( 1+o(1) \big)  \sum_{\sigma \text{ as in } \eqref{eq:sigma-decomp}, \eqref{eq:sigma-decomp-properties}} \quad \prod_{j \in \Im(f)} \kappa_{\sigma_j}(x_j),
\end{align*}
proving the claim: the random matrices $Y_1, \ldots, Y_k$ are tensor free, with marginal distributions $x_1, \ldots, x_k$.
}

\end{proof}
\begin{corollary}\label{cor:embeddings-free}
    With the same notation as in \cref{thm:embed-different-spaces}, with the additional assumption that the function $t$ or $b$ is \emph{constant}, the family $Y_1, \ldots, Y_k$ is asymptotically free (in the usual sense). Note that the condition that one of the embedding functions $t$ or $b$ is constant is necessary for the conclusion above to hold in full generality (since otherwise two distinct elements of the family trivially commute).
\end{corollary}
\begin{remark}
    The boundedness condition in the result above is satisfied when
    the entire family $\{X_1, \ldots, X_k\}$ converges \textit{jointly} in distribution (that is, in the case where the family has a first order limit). In particular, this holds in either of the following cases:
    \begin{itemize}
        \item the random matrices $X_1, \ldots, X_k$ are \emph{independent};
        \item $X_1 = X_2 = \cdots = X_k$.
    \end{itemize}

    On the other hand, we may choose $X_1,\ldots, X_k$ \emph{arbitrarily correlated} when each $X_i$ is a Hermitian random matrix; we refer to \cref{eq-condition-5} and the corresponding discussion. This is remarkable because asymptotic (tensor) freeness is always obtained after the embedding, even when the initial random matrix model has arbitrary correlations.
\end{remark}

{
\begin{remark}
    \cref{thm:embed-different-spaces} generalizes \cite[Theorems H.6 and H.8]{Lan16} to more than two matrices and arbitrary unitarily invariant distributions. Indeed, the results in \cite{Lan16} correspond to the case of three tensor factors $q=1$, $r=2$, and two matrices $k=2$, with $t(1)=t(2) = 1$, $b(1) = 2$, $b(2) = 3$, and $X_1=X_2$ having either a GUE or a Wishart distribution. 
\end{remark}
}
\begin{remark}
    One can generalize the results above from a bipartite to a multipartite setting as follows. 
    Instead of embedding bipartite matrices in a tensor product of spaces as above, one could embed $p$-partite matrices $X_1, \ldots, X_k \in \M{d_1} \otimes \M{d_2} \otimes \cdots \otimes \M{d_p}$ inside a tensor product of spaces $\M{d_1}^{\otimes n_1} \otimes \cdots \otimes \M{d_p}^{\otimes n_p}$, where each matrix is embedded according to a $p$-partite hypergraph on the vertex set $[n_1] \times \cdots \times [n_p]$. The matrix $X_i$ will be embedded according to a hyperedge $e_i$ of this hypergraph, where $e_i = \{x_1^{(i)}, \ldots, x_p^{(i)}\}$ with $x_j^{(i)} \in [n_j]$. We leave the details to the reader. 
\end{remark}

\begin{question}
Can one say anything about the \emph{strong convergence} (i.e., convergence of operator norms) of these families? For example, suppose each $X_j$ is a (normalized) GUE matrix of size $d^2$ and $Y_j:=X_j^{(0j)}\otimes (I_d)^{\otimes [L]\setminus \{j\}}\in \M{d}^{\otimes(L+1)}$. It has recently been shown \cite{CY24,CGVH24} that $Y_1,\ldots, Y_L$ are asymptotically strongly free as $d\to \infty$ if $X_1,\ldots, X_L$ are \emph{independent}. On the other hand, \cite{Lan16} considers the case $X_1=X_2=\cdots=X_L$ and shows that as $d\to \infty$,
    $$\E[\|Y_1+\cdots +Y_L\|_{\infty}]\to 2\sqrt{L}= \|s_1+\cdots +s_L\|$$
where $\{s_1,\ldots, s_L\}$ is a \emph{free semicircular system} in a $C^*$-probability space. A natural question is whether the strong convergence $Y_1,\ldots, Y_L\to s_1,\ldots, s_L$ still holds in this case and, more generally, whether it extends to the case when the family $\{X_1,\ldots, X_L\}$ is unitarily invariant with reasonably nice conditions.
\end{question}

\section{Central limit theorems for tensor free variables}\label{sec:tensor-free-CLT}

The free analogue of central limit theorem by Voiculescu \cite{voiculescu1985symmetries, Voi86} states that semicircular elements, like the normal distributions in classical probability theory, exhibit universality in free probability theory.

\begin{theorem} [Free central limit theorem] \label{thm-FreeCLT}
Let $\{x_i\}_{i=1}^{\infty}$ be a family of centered, identically distributed, and free self-adjoint elements in a $*$-probability space $(\A,\varphi)$. Then $\frac{1}{\sqrt{N}}(x_1+ \cdots + x_N)$ converges in distribution to a semicircular element of mean zero and variance $\varphi(x_1^2)$.
\end{theorem}

The additivity of free cumulants (\cref{eq-CumulantsAdd}) gives a simple proof of \cref{thm-FreeCLT}: we have as $N\to \infty$,
    $$ \tilde{\kappa}_p\left(\frac{x_1+ \cdots + x_N}{\sqrt{N}}\right) = N^{-p/2}\cdot N\tilde{\kappa}_p(x_1) \to \begin{cases}
        \tilde{\kappa}_2(x_1)=\varphi(x_1^2) & \text{if $p=2$,}\\
        0 & \text{otherwise.}
    \end{cases}$$
Since semicircular elements are characterized by the property that $\tilde{\kappa}_p=0$ whenever $p\geq 3$, we obtain the announced convergence.

\medskip

In this section, we derive the tensor free version of the central limit theorem. To this end, let $(\A,\varphi,(\varphi_{\underline{\alpha}}))$ be an $r$-partite algebraic tensor probability space, and let $\{x_i\}_{i=1}^{\infty}\subset \A$ be a family of \textit{identically tensor distributed} and \textit{tensor free} elements. Then we are interested in the tensor distribution limit of the normalized sum
    $$\bar{x}_N:=\frac{1}{\sqrt{N}}(x_1+\cdots+x_N-N\varphi(x_1)).$$
Note that investigating the tensor distribution limit $\displaystyle \lim_{N\to \infty} \varphi_{\underline{\alpha}}(\bar{x}_N)$ requires expanding the sum into an exponential number of terms, which is generally very difficult to control. Instead, we can leverage the additivity of tensor free cumulants (\cref{prop-TensorCumulantAdditivity}) to obtain the following main theorem of this section.

\begin{theorem} [Tensor free central limit theorem] \label{thm-TensorCLT}

Let $\{x_i\}_{i=1}^{\infty}$ be a family of identically tensor distributed, and tensor free elements in an $r$-partite algebraic tensor probability space $(\A, \varphi, (\varphi_{\underline{\alpha}}))$. Then one has
    $$\bar{x}_N=\frac{x_1+ \cdots + x_N-N\varphi(x_1)}{\sqrt{N}}\to \sum_{\underline{\alpha}\in (S_2)^r\setminus \{\underline{\id_2}\}} \sqrt{\kappa_{\underline{\alpha}}(x_1)}\,s_{\underline{\alpha}} \;\;\text{ in tensor distribution as $N\to \infty$},$$
where $(s_{\underline{\alpha}})_{\underline{\alpha}\in (S_2)^r\setminus \{\underline{\id_2}\}}$ are tensor free family and each $s_{\underline{\alpha}}$ is a standard semicircular element having tensor free cumulants
\begin{equation} \label{eq-TensorSemicircular}
    \kappa_{\underline{\beta}}(s_{\underline{\alpha}})=\delta_{\underline{\beta},\underline{\alpha}},\quad \underline{\beta}\in \bigcup_{p\geq 1} (S_p)^r \text{ is irreducible}.
\end{equation}
In particular, the limit of $\bar{x}_N$ is universally governed by $2^r-1$ semicircular elements. Furthermore, the family $(s_{\underline{\alpha}})_{\underline{\alpha}\in (S_2)^r\setminus \{\underline{\id_2}\}}$ can always be constructed from the limit of random matrices (see \cref{lem-TensorLimitRealization,rmk-TensorCLT} below).
\end{theorem}

Before the proof, here are several comments to supplement our theorem.

\begin{remark} \label{rmk-TensorCLT}
\begin{enumerate}
    \item We can understand the family $\{s_{\underline{\alpha}}:\underline{\alpha}\in (S_2)^r\setminus \{\underline{\id_2}\}\}$ as a tensor distribution limit of embedded GUE matrices. Specifically, for each $\underline{\alpha}\in (S_2)^r\setminus \{\underline{\id_2}\}$, let us associate a set $I(\underline{\alpha}):=\{s\in [r]: \alpha_s=\gamma_2\}$, and define a \emph{tensor GUE matrix}
        $$X_{\underline{\alpha}}=X_{d,\underline{\alpha}}:=G_{d}^{I(\underline{\alpha})}\otimes I_d^{\otimes [r]\setminus I(\underline{\alpha})}\in \M{d}^{\otimes I(\underline{\alpha})}\otimes \M{d}^{\otimes [r]\setminus I(\underline{\alpha})}=\M{d}^{\otimes r},$$
    where $G_d^{I(\underline{\alpha})}$ is a normalized GUE matrix of size $d^{|I(\underline{\alpha})|}$ which is embedded into $\M{d}^{\otimes r}$ of the tensor components in $I(\underline{\alpha})$. Then Wigner's semicircle law and \cref{prop-ui-tensormoment} imply that, as $d\to \infty$, $X_{\underline{\alpha}}$ converges in tensor distribution to a semicircular element $s_{\underline{\alpha}}$, with mean $0$ and variance $1$, whose tensor free cumulant satisfies the relation \cref{eq-TensorSemicircular}. Now if we take independent family $(G_d^{I(\underline{\alpha})})$ of GUE matrices and define the corresponding embeddings $(X_{\underline{\alpha}})$, then it converges in tensor distribution to a tensor free semicircular family $(s_{\underline{\alpha}})$, as $d\to \infty$, by \cref{thm-locui-tensorfree}. 
    
    \item The (non-tensor) limit behavior of the family $(X_{\underline{\alpha}})$ has been described via \emph{$\eps$-free independence}; we refer to \cite{Mlo04,SW16} for the precise definition. In our setting, $\eps$ is the adjacency matrix of the finite simple graph $\Gamma=(V,E)$ with
    \begin{center}
        $V=(S_2)^r\setminus \{\underline{\id_2}\}$ and $E=\big\{\{\underline{\alpha},\underline{\beta}\}: I(\underline{\alpha})\cap I(\underline{\beta})=\varnothing \big\}.$
    \end{center}
    Then \cite{CC21} shows that the tensor GUE models $(X_{\underline{\alpha}})_{\underline{\alpha}\in V}$ are asymptotically $\eps$-free independent. In other words, $(s_{\underline{\alpha}})$ are $\eps$-free semicircular family with properties that, $s_{\underline{\alpha}}$ and $s_{\underline{\beta}}$ are classically independent if $I(\underline{\alpha})\cap I(\underline{\beta})=\varnothing$ (i.e., $\eps_{\underline{\alpha},\underline{\beta}}=1$) and freely independent otherwise. Consequently, \cref{thm-TensorCLT} implies that the $\eps$-free semicircular family $(s_{\underline{\alpha}})_{\underline{\alpha}\in V}$ exhibit the universality for tensor free independence.
    
    \item The choice of the signs of coefficients $\sqrt{\kappa_{\underline{\alpha}}(x_1)}\in \Comp$ can be arbitrary since $(s_{\underline{\alpha}})$ has the same tensor distribution with $(\pm s_{\underline{\alpha}})$, regardless of the signs. On the other hand, we have $\kappa_{\underline{\alpha}}(x_1)\geq 0$ for all $\underline{\alpha}\in (S_2)^r\setminus \{\underline{\id_2}\}$ whenever $x_1$ is a self-adjoint element in an $r$-partite space $(\B^{\otimes r},\tau^{\otimes r},  (\bigotimes \tau_{\alpha_s}))$ (see \cref{prop-TensorFreeNonng} below), so we may naturally choose $\sqrt{\kappa_{\underline{\alpha}}(x_1)}\geq 0$ in this case. This fact seems non-trivial even for matrices, e.g., for every bipartite Hermitian matrix $X_{12}\in \M{d}^{\otimes 2}$,
        $$\tilde{\kappa}_{\gamma_2,\gamma_2}(X_{12},X_{12})=\frac{1}{d^2}\Tr(X_{12}^2)-\frac{1}{d^3}\Tr(X_1^2)-\frac{1}{d^3}\Tr(X_2^2)+\frac{1}{d^4}(\Tr X_{12})^2\geq 0,$$
    where $X_1=(\id_d\otimes \Tr_d)(X_{12})$ and $X_2=(\Tr_d\otimes \id_d)(X_{12})$. 
\end{enumerate}
\end{remark}

\begin{proposition} \label{prop-TensorFreeNonng}
For each $s=1,\ldots, r$, let $(\B^{(s)},\tau^{(s)})$ be a $*$-probability space with a tracial state $\tau^{(s)}$, and consider the $r$-partite tensor probability space
    $$(\A,\varphi,(\varphi_{\underline{\alpha}}))=\Big(\bigotimes_{s=1}^r \B^{(s)}, \bigotimes_{s=1}^r\tau^{(s)}, \big(\bigotimes_{s=1}^r \tau_{\alpha_s}^{(s)}\big)\Big).$$
Then for every $x\in \A$ and $\underline{\alpha}\in (S_2)^r$, we have $\kappa_{\underline{\alpha}}(x^*,x)\geq 0$.
\end{proposition}
\begin{proof}
Let us consider a general tensor $x= \sum_{i=1}^k \bigotimes_{s=1}^r a_i^{(s)}$. Then, using \cref{eq-TensorFreeCumulantsFact}, we have, for $\underline{\alpha}\in (S_2)^r$,
\begin{align*}
    \kappa_{\underline{\alpha}}(x^*,x)=\sum_{i,j}\prod_{s=1}^r\tilde{\kappa}_{\alpha_s}(a_i^{(s)*},a_j^{(s)})=\la \mathbf 1, (A_1\odot \cdots \odot A_r) \cdot \mathbf 1\ra,
\end{align*}
where $\mathbf 1=(1,\ldots, 1)^{\top}\in \mathbb{C}^k$ is the all-ones vector, $A_s:=\big(\tilde{\kappa}_{\alpha_s}(a_i^{(s)*},a_j^{(s)})\big)_{1\leq i,j\leq k}\in \M{k}$ for each $s\in [r]$, and $\odot$ denotes the Hadamard product. Since the Hadamard product preserves the positive semi-definite property of matrices, it suffices to show that each matrix $A_s$ is positive semi-definite. If $\alpha_s=\gamma_2$, then for every vector $w=(w_1,\ldots, w_k)^{\top}\in \mathbb{C}^k$,
    $$\la w,A_s w\ra=\sum_{i,j=1}^k \overline{w_i}w_j \tilde \kappa_{\gamma_2}(a_i^{(s)*},a_j^{(s)})
    =\tilde\kappa_{\gamma_2}(c^*,c) = \tau^{(s)}(c^*c)-\tau^{(s)}(c^*)\tau^{(s)}(c)\geq 0,$$
where $c:=\sum_{i=1}^k w_ia_i^{(s)}\in \B^{(s)}$, and we used the formula $\tilde{\kappa}_{\gamma_2}(a,b)=\tau^{(s)}(ab)-\tau^{(s)}(a)\tau^{(s)}(b)$ and the Cauchy-Schwarz inequality
    $$|\tau^{(s)}(a^*b)|^2\leq \tau^{(s)}(a^*a)\tau^{(s)}(b^*b),\quad a,b\in \B^{(s)}$$
to prove that the second cumulant (i.e.~the variance) is non-negative. 
The case $\alpha_s=\id_2$ follows in a similar manner: $\tilde\kappa_{\id_2}(c^*,c) = |\tau^{(s)}(c)|^2 \geq 0$.
\end{proof}

\begin{proof}[\textbf{Proof of \cref{thm-TensorCLT}}]
By \cref{prop-TensorCumulantAdditivity}, one has
    $$\kappa_{\underline{\alpha}}(\bar{x}_N)=N^{-p/2}\kappa_{\underline{\alpha}}\Big(\sum_{i=1}^N (x_i-\varphi(x_i))\Big)=N^{1-p/2}\kappa_{\underline{\alpha}}(x_1-\varphi(x_1))$$
since $\{x_i-\varphi(x_i)\}_{i=1}^{\infty}$ are identically tensor distributed and tensor free. The above clearly becomes $0$ if $p=1$. If $p \geq 2$, $\kappa_{\underline{\alpha}}(x_1-\varphi(x_1))=\kappa_{\underline{\alpha}}(x_1)$ since $x_1$ and $1_\A$ are tensor free (\cref{prop-1-TensorFree}). Therefore, by taking the limit $N\to \infty$, we have for every irreducible $\underline{\alpha}\in (S_p)^r$,
\begin{equation} \label{eq-TensorCLT1}
    \lim_{N\to\infty}\kappa_{\underline{\alpha}}\left(\bar{x}_N\right) = \begin{cases}
        \kappa_{\underline{\alpha}}(x_1) & \text{if $p=2$,}\\
        0 & \text{otherwise.}
    \end{cases}
\end{equation}
The convergence $\bar{x}_N \xrightarrow{\text{$\otimes$-distr}} \sum_{\underline{\alpha}\in (S_2)^r\setminus \{\underline{\id_2}\}} \sqrt{\kappa_{\underline{\alpha}}(x_1)}\,s_{\underline{\alpha}}$ is now clear from the cumulants \cref{eq-TensorSemicircular} and the additivity property of $\kappa_{\underline{\alpha}}$, since \cref{eq-TensorCLT1} can be rephrased as
    $$\lim_N \kappa_{\underline{\beta}}(\bar{x}_N)=\sum_{\underline{\alpha}\in (S_2)^r\setminus \{\underline{\id_2}\}} \kappa_{\underline{\alpha}}(x_1) \delta_{\underline{\beta},\underline{\alpha}}=\kappa_{\underline{\beta}}\Big( \sum_{\underline{\alpha}} \sqrt{\kappa_{\underline{\alpha}}(x_1)} s_{\underline{\alpha}}\Big).$$
\end{proof}

Note that the case $r=1$ recovers the original free central limit theorem (\cref{thm-FreeCLT}). On the other hand, since \cref{thm-TensorCLT} implies that $\bar{x}_N$ converges in distribution, one can naturally ask that if we can find the probability distribution of the limit, i.e., a probability measure $\mu\in {\rm Prob}(\mathbb{R})$ such that
    $$\lim_{N\to \infty}\varphi(\bar{x}_N^p)=\int_{\mathbb{R}} t^p{\rm d}\mu(t), \quad p\geq 1,$$
when an $r$-partite structure $(\A,\varphi, (\varphi_{\underline{\alpha}}))$ has a $*$-probability space structure and $x_i\in \A$ are self-adjoint. 

First of all, the connection between free cumulants and tensor free cumulants (\cref{prop-cumulant-from-tensorcumulant})  gives the following central limit theorem, in the usual sense, for tensor free elements.

\begin{corollary} \label{cor-TensorCLT2}
Let $\{x_i\}$ be as in \cref{thm-TensorCLT}. Then we have
\begin{equation} \label{eq-TensorCLT2}
    \tilde{\kappa}_p\left(\frac{x_1+ \cdots + x_N- N\varphi(x_1)}{\sqrt{N}}\right)\to \begin{cases}
        \sum_{\underline{\alpha}} \kappa_{\underline{\alpha}}(x_1) & \text{if $p$ is even,}\\
        0 & \text{if $p$ is odd,}
    \end{cases}
\end{equation}
where the sum in the first case is taken over {$\alpha_1,\ldots, \alpha_r\in NC_{1,2}(p)$} such that
\begin{center}
    $\alpha_1\vee_{NC}\cdots \vee_{NC} \alpha_r=1_{p}$ and $\alpha_1\vee_{\mathcal{P}}\cdots \vee_{\mathcal{P}} \alpha_r \in \mathcal{P}_2(p)$.
\end{center}
\begin{proof}
By combining \cref{thm-TensorCLT} with \cref{prop-cumulant-from-tensorcumulant}, we have
\begin{align*}
    \lim_{N\to \infty} \tilde{\kappa}_p(\bar{x}_N) = \sum_{\substack{\underline{\alpha}\in (S_{NC}(\gamma_p))^r \\ \alpha_1\vee\cdots\vee \alpha_r=\gamma_p}} \lim_{N\to \infty}\kappa_{\underline{\alpha}}(\bar{x}_N)=\sum_{\underline{\alpha}}\kappa_{\underline{\alpha}}(x_1),
\end{align*}
where the last sum is taken over $\underline{\alpha}\in (S_{NC}(\gamma_p))^r$ satisfying 
\begin{center}
    $\alpha_1\vee\cdots\vee\alpha_r=\gamma_p$ and every irreducible component of $\underline{\alpha}$ is in the set $(S_2)^r\setminus\{\underline{\id_2}\}$. 
\end{center}
Note that after the identification $S_{NC}(\gamma_p)\cong NC(p)$, these conditions are equivalent to that $\alpha_s\in NC_{1,2}(p)$ for each $s\in [r]$, $\bigvee_{NC}\alpha_s=1_p$, and $\bigvee_{\mathcal{P}}\alpha_s\in \mathcal{P}_2(p)$. In particular, such $\underline{\alpha}$ exists if and only if $p$ is even, and $\kappa_{\underline{\alpha}}(x_1-\varphi(x_1))=\kappa_{\underline{\alpha}}(x_1)$ for that cases since $x_1$ is free from $1_{\A}$ and $\kappa_{\underline{\alpha}}(1_\A)=0$.
\end{proof}
\end{corollary} 

While \cref{cor-TensorCLT2} gives the combinatorial expression of the free cumulants for the limit distribution of $\bar{x}_N$, it still seems to be challenging to find the probability distribution explicitly. 

\medskip

For the remaining of this section, let us assume that $\kappa_{\underline{\alpha}}(x_1)\geq 0$ for all $\underline{\alpha}\in (S_2)^r\setminus \{\underline{\id_2}\}$. Then \cref{thm-TensorCLT} implies that 
    $$\bar{x}_N \to \bar{x}_{\infty}:=\sum_{\underline{\alpha}\in (S_2)^r\setminus \{\underline{\id_2}\}} \sqrt{\kappa_{\underline{\alpha}}(x_1)}\,s_{\underline{\alpha}}$$
in distribution, so the problem reduces to finding the probability measure $\mu_{{\infty}}\in {\rm Prob}(\mathbb{R})$ having the $p$-th moments $\varphi(\bar{x}_{\infty}^p)=\int t^p {\rm d}\mu_{{\infty}}(t)$. We first consider the bipartite case $r=2$. Then as in the discussion in \cref{rmk-TensorCLT}, the three tensor free variables $s_{\gamma_2,\id_2}$, $s_{\id_2,\gamma_2}$, and $s_{\gamma_2,\gamma_2}$ can be constructed from the limit $d\to\infty$ of three independent tensor GUE matrices
    $$G_d^{(1)}\otimes I_d, \;\; I_d\otimes G_d^{(2)}, \;\; G_{d}^{(12)}\in \M{d}^{\otimes 2}$$
This implies that $s_{\gamma_2,\id_2}$ and $s_{\id_2,\gamma_2}$ are \textit{classically  independent}, and $\{s_{\gamma_2,\id_2},s_{\id_2,\gamma_2}\}$ are \textit{free} from $s_{\gamma_2,\gamma_2}$ (\cref{thm-UIasympfree}). Consequently, this gives the explicit description of the limit distribution of $\bar{x}_N$. For this, let us denote by $D_t[\mu]$ the push-forward measure of $\mu\in {\rm Prob}(\mathbb{R})$ with respect to the dilation map $x\mapsto tx$, and $\mu_{SC}\in {\rm Prob}(\mathbb{R})$ by the standard semicircle law
    $${\rm d}\mu_{SC}(t)=\frac{1}{2\pi}\sqrt{4-t^2}\,\mathbb{1}_{[-2,2]}(t)\,{\rm d}t.$$

\begin{theorem} (Central limit theorem for bipartite tensor free variables) \label{thm-CLTBipartite}
Let $\{x_i\}_{i=1}^{\infty}$ be a family of identically tensor distributed, and tensor free self-adjoint elements in a bipartite algebraic tensor probability space $(\A, \varphi, (\varphi_{\underline{\alpha}}))$ where $(\A, \varphi)$ is a $*$-probability space and $\kappa_{\underline{\alpha}}(x_1)\geq 0$ for all $\underline{\alpha}\in (S_2)^2\setminus \{\underline{\id_2}\}=\{(\gamma_2,\id_2),(\id_2,\gamma_2),(\gamma_2,\gamma_2)\}$. Then the normalized sum $\bar{x}_N=\frac{1}{\sqrt{N}}(x_1+\cdots x_N-N \varphi(x_1))$ converges in distribution to the probability distribution
\begin{equation}\label{eq:mu-S-infty}
    \mu_{{\infty}}=\big(D_{\sqrt{\kappa_{\gamma_2,\id_2}(x_1)}}[\mu_{SC}] * D_{\sqrt{\kappa_{\id_2,\gamma_2}(x_1)}}[\mu_{SC}]\big)\boxplus D_{\sqrt{\kappa_{\gamma_2,\gamma_2}(x_1)}}[\mu_{SC}],
\end{equation}
where $*$ denotes the \emph{classical convolution} and $\boxplus$ denotes the \emph{(additive) free convolution}.
\end{theorem}

We plot in \cref{fig:mu-S-infty} numerical simulations of the probability measure $\mu_{S_{\infty}}$ for different values of the tensor free cumulants. 

\begin{figure}
    \centering
    \includegraphics[width=0.45\linewidth]{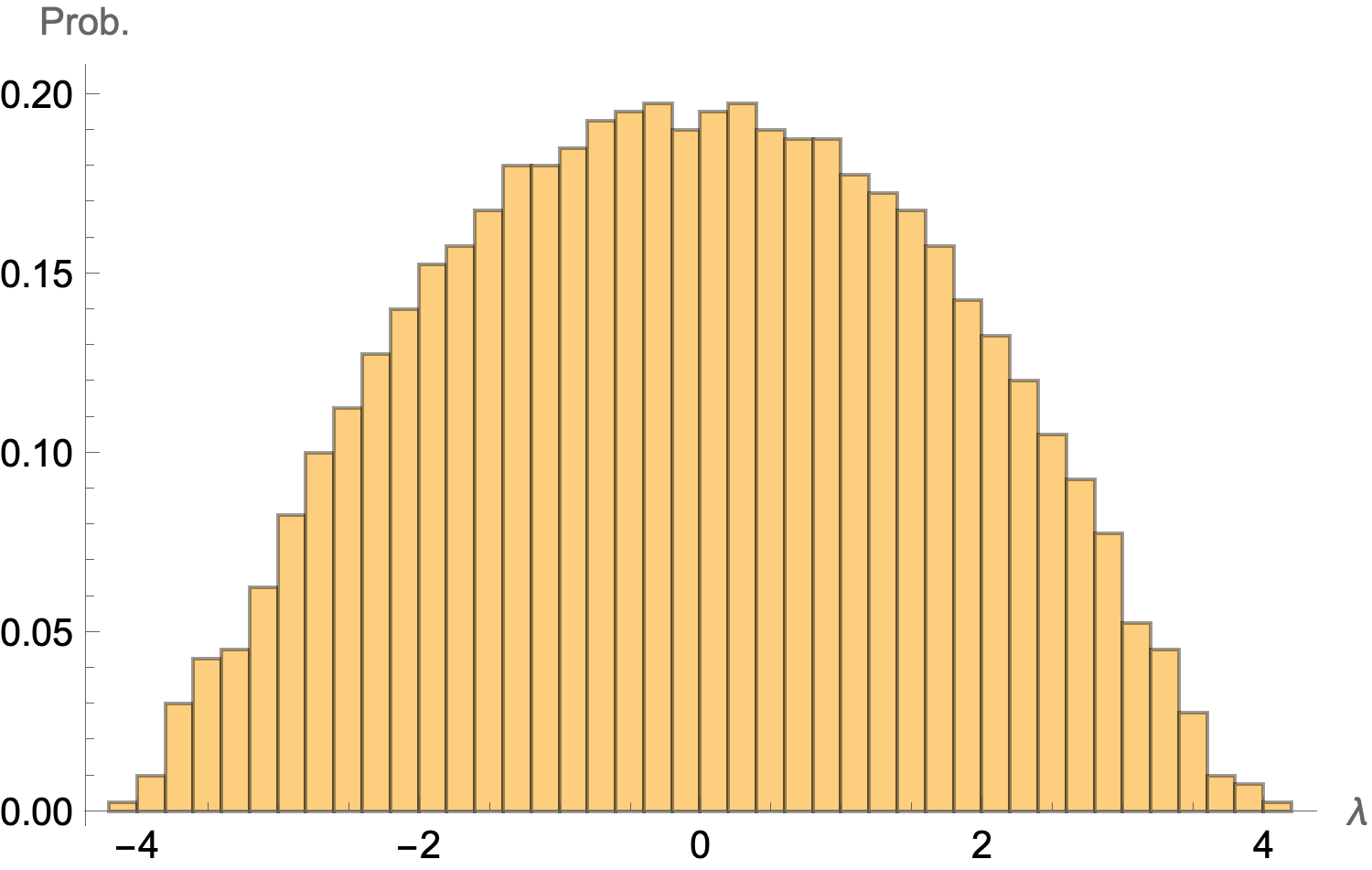} \qquad 
    \includegraphics[width=0.45\linewidth]{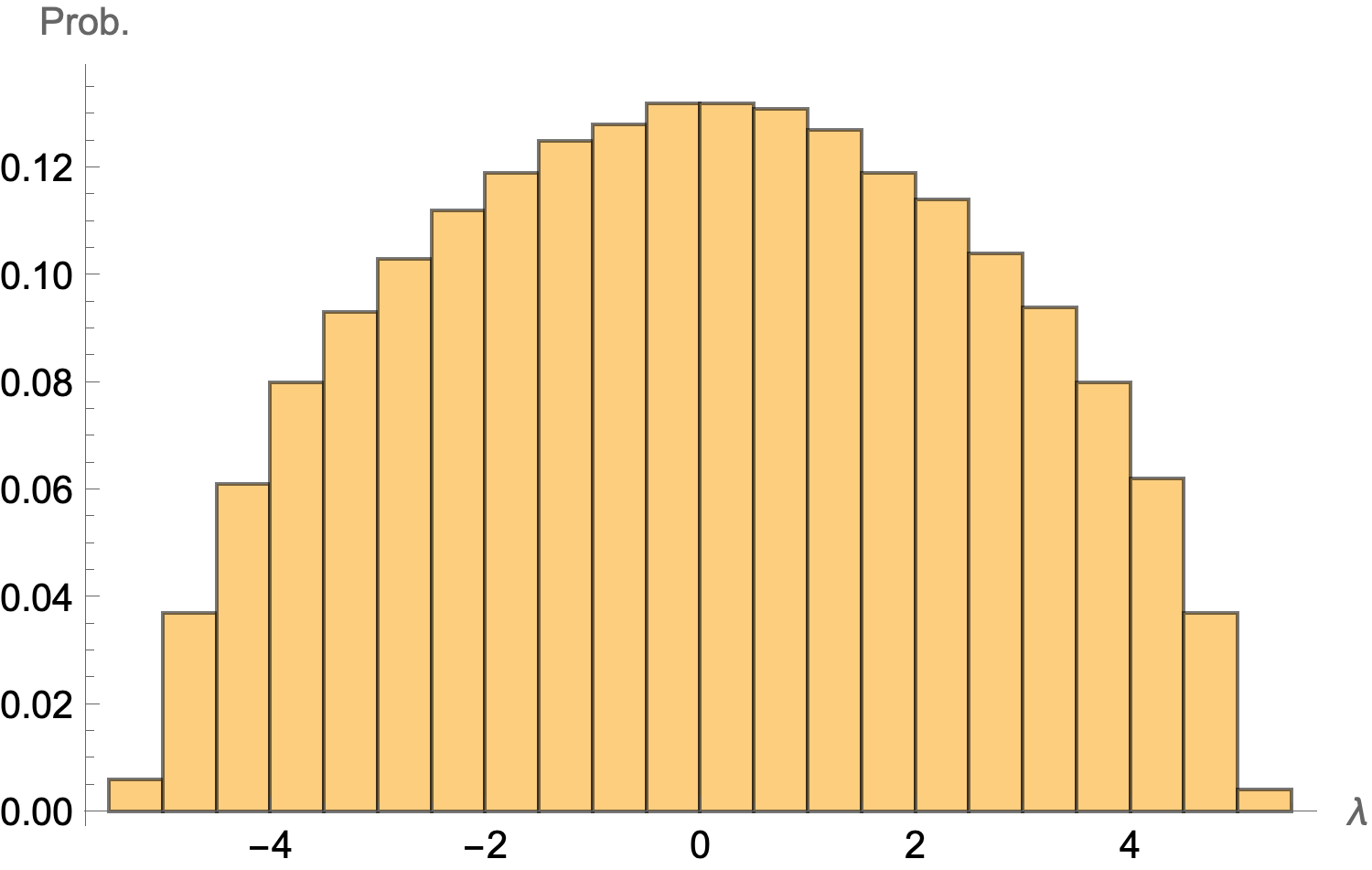} \\ \vspace{1cm}
    \includegraphics[width=0.45\linewidth]{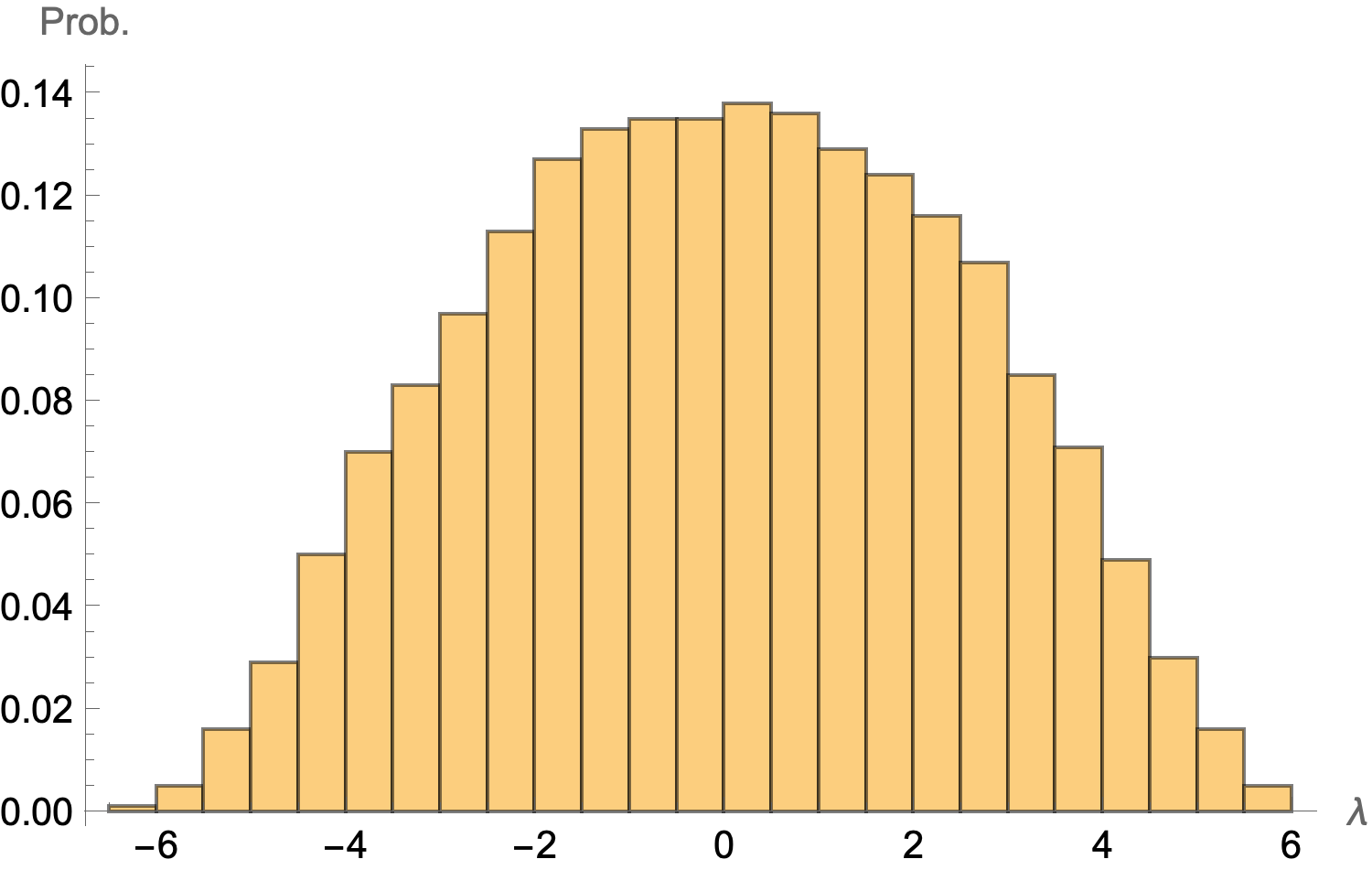} \qquad 
    \includegraphics[width=0.45\linewidth]{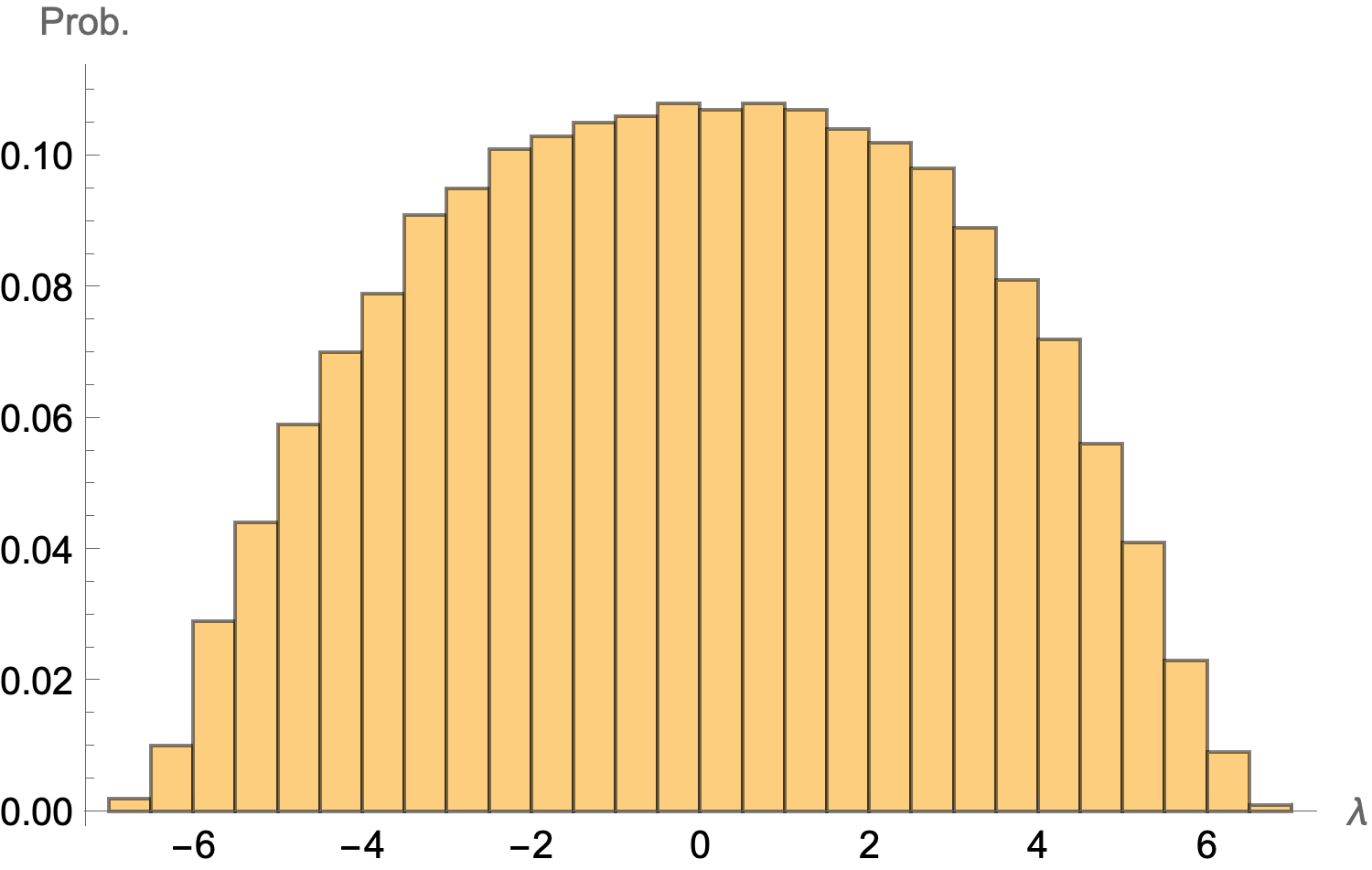}
    \caption{Histograms of $2000$ samples from the distribution $\mu_{S_{\infty}}$ from \cref{eq:mu-S-infty} for different values of the tensor free cumulants $\kappa_{\underline{\alpha}}(x_1)$. From top to bottom, left to right, we have: $(\kappa_{\gamma_2,\id_2},\kappa_{\id_2, \gamma_2},\kappa_{\gamma_2,\gamma_2}) = $ $(1,1,1)$, $(1,1,4)$, $(1,4,1)$, $(1,4,4)$.}
    \label{fig:mu-S-infty}
\end{figure}

\subsection{Central limit theorem for tensor product of free variables} \label{sec-CLTProdFree}

Let $(\B,\tau)$ be a $*$-probability space equipped with a tracial state $\tau$, and suppose that for each $s\in [r]$, $\{a_s^{(i)}\}_{i=1}^{\infty}$ is a family of identically distributed free  self-adjoint elements with mean $\tau(a_s^{(i)})=\lambda_s$ and variance ${\rm var}(a_s^{(i)})=\tau((a_s^{(i)})^2)-\tau(a_s^{(i)})^2=\sigma_s^2$. Consider a family $\{x_i\}_{i=1}^{\infty}$ in the $r$-partite tensor probability space $(\A,\varphi,(\varphi_{\underline{\alpha}}))=(\B^{\otimes r},\tau^{\otimes r},(\bigotimes_{s=1}^r \tau_{\alpha_s}))$ defined by
    $$x_i:=a_1^{(i)}\otimes\cdots \otimes a_r^{(i)}.$$
Then $\{x_i\}_{i=1}^\infty$ are {identically tensor distributed and tensor free}, as noted in \cref{prop:tensor-free-tensor-product}. As before, we are interested in the limit distribution of the normalized sum
    $$\bar{x}_N:=\frac{1}{\sqrt{N}}\big(x_1+\cdots+ x_N-N(\lambda_1\cdots \lambda_r) 1_{\A}\big).$$
Note that for every $\underline{\alpha}\in (S_r)^2\setminus \{\underline{\id_2}\}$,
    $$\kappa_{\underline{\alpha}}(x_1)=\prod_{s=1}^r \tilde{\kappa}_{\alpha_s}(a_s^{(1)})=\prod_{s=1}^r \sigma_s^{2|\alpha_s|}\lambda_s^{2(1-|\alpha_s|)}\geq 0.$$
Therefore, \cref{thm-TensorCLT,cor-TensorCLT2} implies that $\displaystyle \bar{x}_N\to \sum_{\underline{\alpha}\in (S_2)^r\setminus \{\underline{\id_2}\}} \sqrt{\kappa_{\underline{\alpha}}(x_1)}\,s_{\underline{\alpha}}$ in tensor distribution as $N\to \infty$, $\displaystyle \lim_{N\to \infty}\tilde{\kappa}_{2p-1}(\bar{x}_N)=0$, and
\begin{equation} \label{eq-CumulantFreeProd}
    \lim_{N\to \infty}\tilde{\kappa}_{2p}(\bar{x}_N)=\sum_{\substack{\alpha_s\in NC_{1,2}(2p)\\ \bigvee_{NC}\alpha_s=1_{2p} \\ \bigvee_{\mathcal{P}}\alpha_s\in \mathcal{P}_{2}(2p)}}\prod_{s=1}^r \sigma_s^{2|\alpha_s|}\lambda_s^{2p-2|\alpha_s|}, \quad p\geq 1.
\end{equation}
Furthermore, when $r=2$, \cref{thm-CLTBipartite} implies that $\bar{x}_N$ converges in distribution to the probability measure
    $$\mu_{{\infty}}=\big(D_{\sigma_1|\lambda_2|}[\mu_{SC}] * D_{|\lambda_1|\sigma_2}[\mu_{SC}]\big)\boxplus D_{\sigma_1\sigma_2}[\mu_{SC}],$$
which recovers the main result of \cite{LSY24,Sko24}. Indeed, \cite{LSY24} considers the case $\lambda_1=\lambda_2=\lambda$ and $\sigma_1=\sigma_2=\sigma$, and proves that $\bar{x}_N$ converges in distribution as $N\to \infty$ to
    $$\mu_{\lambda, \sigma}=D_{\sigma|\lambda|}[\mu_{SC}*\mu_{SC}] \boxplus D_{\sigma^2}[\mu_{SC}],$$
by computing the moment limit $\displaystyle \lim_{N\to \infty}\varphi(\bar{x}_N^p)$ for every positive integer $p$ and showing that the limit corresponds to the moments of $\mu_{\lambda,\sigma}$ from its free cumulants $\kappa_{2p-1}(\mu_{\lambda,\sigma})=0$ and
    $$\kappa_{2p}(\mu_{\lambda,\sigma})=\begin{cases}
        \sigma^4+ 2\sigma^2\lambda^2 & \text{if $p=1$,} \\
        2\sigma^{2p}\lambda^{2p}M_p & \text{if $p\geq 2$,}
    \end{cases}$$
where $M_p$ is the cardinality of a (combinatorial) set depending only on $p$. On the other hand, another approach using \textit{bi-free} techniques \cite{Voi14} has been taken in \cite{Sko24} to recover the same limit $\bar{x}_N\xrightarrow{\text{distr}} \mu_{\lambda,\sigma}$.

While obtaining the closed form of the general case in \cref{eq-CumulantFreeProd} remains an open problem, we introduce another case in which the explicit limit distribution of $\bar{x}_N$ can be determined.

\begin{corollary}
Suppose $r\geq 2$ and $\lambda_r=0$. Then $\bar{x}_N$ converges in distribution to a semicircular element of mean 0 and variance $\sigma_r^2 \prod_{s=1}^{r-1}(\lambda_s^2+\sigma_s^2)$. 

\end{corollary}
\begin{proof}
It suffices to show that 
    $$\lim_{N\to \infty} \tilde{\kappa}_{2p}(\bar{x}_N)=\begin{cases}
        \sigma^2_r \prod_{s=1}^{r-1}(\lambda_s^2+\sigma_s^2) & \text{if $p=1$,}\\
        0 & \text{if $p\geq 2$.}
    \end{cases}$$
Let us begin from \cref{eq-CumulantFreeProd}. If $p\geq 2$, we claim that $\alpha_r$ in the RHS of \cref{eq-CumulantFreeProd} contains a singleton, causing all the terms in the sum to become zero. Otherwise, the condition $\alpha_r\in NC_{1,2}(2p)$ forces that $\alpha_r\in NC_2(2p)$, but the condition $\alpha_r\leq \bigvee_{\mathcal{P}} \alpha_s\in \mathcal{P}_2(2p)$ implies that $\alpha_s\leq \alpha_r$ for all $s$ this contradicts another condition $\bigvee_{NC}\alpha_s=1_{2p}$ and hence shows the claim.

On the other hand, if $p=1$, then the condition again implies that $\alpha_r=1_{2}$ and $\alpha_s\in \{0_2,1_2\}$ can be arbitrarily taken for $s=1,\ldots, r-1$. This implies the limit value as above.
\end{proof}

{
\subsection{Classical central limit theorem via tensor freeness}

We have seen in the previous subsections how the free central limit theorem of Voiculescu (\cref{thm-FreeCLT}) can be generalized to the setting of tensor free independence (\cref{thm-TensorCLT}). In this section we consider a realization of the \emph{classical} central limit theorem (see, e.g.~\cite[Theorem 8.5]{nica2006lectures}) in the setting of tensor probability spaces. We shall prove a version of the (classical) central limit theorem that follows from the tensor free independence of the tensor product construction from \cref{ex:tensor-product-of-ncps}, see \cref{prop:tensor-free-tensor-product,prop-TensorIndep}.

\begin{theorem}
    Let $(\A, \phi)$ be a non-commutative probability space and consider an element $a \in \A$ which is centered $\phi(a) = 0$ and has variance $\sigma^2 :=\phi(a^2)$. For an arbitrary positive integer $N \geq 1$, consider the algbraic tensor probability space $\A^{\otimes N}$ from \cref{ex:tensor-product-of-ncps} and the variables $x_1, \ldots, x_N \in \A^{\otimes N}$ defined by
    $$\forall i \in [N], \qquad x_i := 1 \otimes \cdots \otimes 1 \otimes \underarrow[\text{ position~}i]{a} \otimes 1 \otimes \cdots \otimes 1.$$
    Define the random variable
    $$z_N := \frac{x_1 + \cdots + x_N}{\sigma \sqrt N} \in (\A^{\otimes N},\varphi^{\otimes N}).$$
    Then, the random variables $z_N$ converge in distribution, as $N \to \infty$, to the standard Gaussian distribution:
    $$\forall p \geq 1, \qquad \lim_{N \to \infty} \phi(z_N^p) = \begin{cases} 
    (p-1)(p-3) \cdots 3 \cdot 1 &\quad \text{ if $p$ is even}\\
    0 &\quad \text{ if $p$ is odd.}
    \end{cases}
    $$
\end{theorem}
\begin{proof}
    Start by applying the tensor free moment-cumulant formula from \cref{eq:tensor-moment-free-cumulant}:
    $$\phi(z_N^p) = \sigma^{-p} N^{-p/2} \sum_{\underline{\alpha} \in NC(p)^N} \kappa_{\underline{\alpha}}(x_1 + \cdots + x_N),$$
    where the sum is over $N$-tuples of non-crossing permutations $\alpha_i \in S_{NC}(\gamma_p) \cong NC(p)$. 
    By \cref{prop:tensor-free-tensor-product}, the elements $x_1, \ldots, x_N$ are tensor freely independent. Hence, for all \emph{irreducible} tuple $\underline{\alpha}$, we have, by \cref{prop-TensorCumulantAdditivity}
    $$\kappa_{\underline{\alpha}}(x_1 + \cdots + x_N) = \sum_{i=1}^N \kappa_{\underline{\alpha}}(x_i).$$
    The tensor free cumulants of the elements $x_i$ are easily found: for any irreducible tuple $\underline{\alpha} \in NC(p)^r$, 
    $$\kappa_{\underline{\alpha}}(x_i) = \mathds 1_{\underline{\alpha} = (\id_p, \cdots, \id_p, \underarrow[\text{ pos.~}i]{\gamma_p}, \id_p, \ldots, \id_p)} \kappa_p(a).$$
    We say that an irreducible tuple $\underline{\alpha}$ \emph{has one $\gamma$} if 
    $$\underline{\alpha} = (\id_p, \cdots, \id_p, \underarrow[\text{ pos.~}i]{\gamma_p}, \id_p, \ldots, \id_p) \qquad \text{ for some $i \in [N]$}.$$
    Similarly, we say that a tuple $\underline{\alpha}$ \emph{has one $\gamma$ per block} if, for every block $b \in \Pi(\underline{\alpha}):=\bigvee_{\mathcal{P}}\Pi(\alpha_s)$, $\underline{\alpha} \big|_{b}$ has one $\gamma=\gamma_b$, the full cycle on $b\subset [p]$ with increasing order. We have then, for all $\underline{\alpha} \in NC(p)^N$, 
    $$\kappa_{\underline{\alpha}}(x_1+\cdots +x_N) = \prod_{b\in \Pi(\underline{\alpha})}\sum_{i=1}^N \kappa_{\underline{\alpha}|_{b}}(x_i) =  \mathds 1_{\underline{\alpha} \text{ has one $\gamma$ per block}} \cdot \kappa_{\Pi(\underline{\alpha})}(a).$$
    
    Putting everything together, we have 
    \begin{align*}
        \phi(z_N^p) &= \sigma^{-p} N^{-p/2} \sum_{\underline{\alpha} \in NC(p)^N} \mathds 1_{\underline{\alpha} \text{ has one $\gamma$ per block}} \cdot \kappa_{\Pi(\underline{\alpha})}(a)\\
        &= \sigma^{-p} N^{-p/2} \sum_{\pi \in \mathcal P(p)} \kappa_\pi(a) \cdot \underbrace{\left| \{ \underline{\alpha} \in NC(p)^N \, : \, \Pi(\underline \alpha) = \pi \text{ and } \underline \alpha \text{ has one $\gamma$ per block} \} \right|}_{=:S_\pi}.
    \end{align*}
    Clearly, if $\pi$ has a singleton, we have $\kappa_\pi(a) = 0$ (recall that $a$ is centered). Moreover, it is easy to see that whenever $p\leq N$,
    $$N(N-1)\cdots (N-\#\pi+1) \leq S_\pi \leq N^{\# \pi},$$
    since, for every block $b \in \pi$, one has to choose on which position $i \in [N]$ the full cycle $\gamma_{b}$ is found, and since we can always choose $\underline{\alpha}\in NC(p)^N$ such that every $\alpha_s$ is either of the form $\alpha_s=b\sqcup 0_{[p]\setminus b}\in NC(p)$ for some $b\in \pi$ or $\alpha_s=0_p$.

    In the limit $N \to \infty$, only $\emph{pair partitions}$ can contribute: since no singletons are allowed, 
    $$S_\pi\sim N^{\# \pi} \leq N^{p/2},$$
    with equality iff $\pi$ is a pair partition. In that case, $\kappa_\pi(a) = \sigma^p$, finishing the proof. 
\end{proof}

}

\noindent\textbf{Acknowledgments.} The authors thank Rémi Bonnin, Cécilia Lancien, Luca Lionni, Patrick Oliveira Santos, Paul Skoufranis, and Pierre Youssef for the helpful discussions and comments. {We would also like to thank the organizers of the IHP program \href{https://tensors-2024.sciencesconf.org/}{Random tensors and related topics} for the kind invitation to present a preliminary version of this work.} Both authors were supported by the ANR project \href{https://esquisses.math.cnrs.fr/}{ESQuisses}, grant number ANR-20-CE47-0014-01.

\bibliography{references}
\bibliographystyle{alpha}
\bigskip
\hrule
\bigskip

\end{document}